\documentclass[a4paper]{article}
\usepackage{amscd}
\usepackage{fancyhdr,latexsym,amsmath,amsfonts,amssymb,amsbsy,amsthm,url}%
\usepackage{natbib}
\usepackage[hscale=0.7,centering]{geometry}

\usepackage{graphicx,epsfig, xy}
\usepackage{hyperref}
\usepackage{lscape,fancybox}
\usepackage{color}
\usepackage{enumitem}
\numberwithin{equation}{section}

\usepackage{float}
\numberwithin{equation}{section}

\newcommand{\BE}{{\mathbb{E}}}
\newcommand{\BCov}{{\mathbb{C}\mathrm{ov}}}
\newcommand{\BVar}{{\mathbb{V}\mathrm{ar}}}

\newtheorem{lemma}{Lemma}
\newtheorem{theorem}{Theorem}
\newtheorem{remark}{Remark}[section]

\title{\textbf{Assessing bivariate independence: Revisiting Bergsma's covariance}}
\date{May 29, 2023}
\begin{document}
		\maketitle
	\def\shortauthors{ABC}
	\title{\textbf{\Large \sc
			%When is a variable free Gaussian
			%Independence of $\bar X$ and $s^2$ in the free world
%			Covariance $\kappa$ its Estimation and Asymptotics
	}}
	\author{
		%\hspace{0.05\textwidth}
			%%\thanks{AAAAA}
		\parbox[t]{0.28\textwidth}{{Divya Kappara}
			\\ {\small 	School of Math. \& Stat. \\
				      	University of Hyderabad\\  
					%INDIA \\
			}}
		%\hspace{0.01\textwidth}
		\parbox[t]{0.28\textwidth}{{Arup Bose}
			%%  \thanks{Research  supported by J.C.~Bose National Fellowship, Dept.~of Science and Technology, Govt.~of India.}
			\\ {\small 	Stat.-Math. Unit, Kolkata\\
					Indian Statistical Institute\\
					%203 B.T. Road\\
					%Kolkata 700108\\
					%INDIA \\ 
					%%bosearu@gmail.com \\
			}}
%		\hspace{0.01\textwidth}
		\parbox[t]{0.35\textwidth}{{{Madhuchhanda Bhattacharjee}
			\\ {\small School of Math. \& Stat. \\
				        University of Hyderabad\\
					%INDIA \\
			}}
		}
	}

\begin{abstract} 
\noindent Bergsma (2006) proposed a covariance $\kappa (X,Y)$ between random variables $X$ and $Y$. He derived their asymptotic distributions under the null hypothesis of independence between $X$ and $Y$. The non-null (dependent) case does not seem to have been studied in the literature.
		
We derive several alternate expressions for $\kappa$. One of them leads us to a very intuitive estimator of $\kappa (X, Y)$ that is a nice function of four naturally arising $U$-statistics. We derive the exact finite sample relation between all three estimates.

The asymptotic distribution of our estimator, and hence also of the other two estimators, in the non-null (dependence) case, is then obtained by using the $U$-statistics central limit theorem. 
%As is typical with $U$-statistics, the variance of the limit distribution cannot be found out explicitly.  

For specific parametric bivariate distributions, the value of $\kappa$ can be derived in terms of the natural dependence parameters of these distributions. In particular, we derive the formula for $\kappa$ when $(X,Y)$ are distributed as Gumbel's bivariate exponential.

We bring out various aspects of these estimators through extensive simulations from several prominent bivariate distributions. In particular, we investigate the empirical relationship between $\kappa$ and the dependence parameters, the distributional properties of the estimators, and the accuracy of these estimators. We also investigate the powers of these measures for testing independence, compare these among themselves, and with other well known such measures. Based on these exercises, the proposed estimator seems as good or better than its competitors both in terms of power and computing efficiency.

\end{abstract}


\section{Introduction}\label{sec:intro} 

\cite{ bergsma2006new} introduced a new covariance $\kappa:=\kappa(X,Y)$, and correlation $\rho:=\rho(X,Y)= \kappa(X,Y)/\sqrt{\kappa (X,X) \kappa(Y,Y)}$, between variables $X$ and $Y$. Suppose $F_i$, $i=1, 2$ are the marginal distributions of $X$ and $Y$ respectively.	Let $(X_i, Y_i)$, $i=1, 2$ be i.i.d. copies of $(X, Y)$. Then $\kappa=\mathbb{E}[h_{F_{1}}(X_1,X_2)h_{F_{2}}(Y_1,Y_2)]$ for some kernel function $h_{F_i}$, $i=1, 2$ which is defined later.

It is known that $\kappa=0$ if and only if $X$ and $Y$ are independent. Based on $n$ i.i.d. observations $(X_i, Y_i)$, $1\leq i \leq n$, Bergsma provided two estimates $\tilde \kappa$ and $\hat \kappa$ of $\kappa$,  which are $U$ and $V$ statistics with \textit{estimated} kernels $h_{\hat F_{1}}$ and $h_{\hat F_{2}}$. Using the theory of degenerate $U$-statistics,  their limit distributions were obtained under independence of $X$ and $Y$. In Section \ref{sec:bergsma} we recall these developments. 

It appears that the limit distributions of these statistics under dependence are not available in the literature. The functional form of $\kappa$ has also not been provided for any standard parametric family. We provide few examples of such derivations enabling us to understand this new concept of correlation in the context of such parametric families of distributions.
	
In Section \ref{sec:altberg}, we first derive several alternate expressions for $\kappa$. One of them leads us to a very intuitive estimator of $\kappa$ that is based on four $U$-statistics which we define. We also derive the exact algebraic relation between the three estimates and the four $U$-statistics.  
% Subsection 3.2 & 3.3?

The exact value of true $\kappa$ is hard to calculate for most distributions. For the non-null case, we give detailed derivations leading to simplified formulae for $\kappa$ for the bivariate exponential and the bivariate normal families of distributions as examples in Section \ref{sec:altberg}. We further prove the convexity of $\kappa$ under bivariate normal, and under bivariate exponential $\kappa$ is a non-decreasing function of $\theta$.  % Subsection 3.1

Next we construct a framework for estimating $\kappa$. The asymptotic normality of our estimator, and hence that of Bergsma's estimators, then follow by using the $U$ statistics central limit theorem and are presented in Section \ref{sec:asymp}. In the null case, this limit is degenerate. With a higher scaling, the non-degenerate limit distribution of the three estimators are obtained using the theory of degenerate $U$-statistics. The necessary background in the asymptotic distribution for $U$-statistics is covered in Appendix \ref{sec:ustat}.
% Subsec 4.1 dependent case,
	
For the null asymptotic distribution, one needs the infinite sequence of eigenvalues of kernels which depend on the marginal distributions of $X$ and $Y$. Again, these cannot be obtained explicitly, and we show how to compute them approximately. In Appendix  \ref{sec:ev} we give the procedure to compute the eigenvalues required for the limit distribution under independence under different distributional assumptions.
%  subsec 4.2 indp case

In Section \ref{sec:simu}, we present some simulations and computational aspects of the estimators, under various dependence scenarios as portrayed by different bivariate distributions. See Appendix \ref{sec:bvd} for details. We also compare the performances of these estimators for testing independence, among themselves and with other established measures of dependency. See \cite{ma2022}, \cite{chatterjee2021} and \cite{tjostheim2022} for more information on, and further discussions of, such methods.

Under independence/degeneracy, we investigate the empirical distributions of the estimates of $\kappa$ when the data is from specific dependent bivariate distributions, and compare these with the theoretical limit distribution (see Section \ref{subsec:independent}). Note that the computation of the limit distributions involve eigenvalues of certain kernels which depend on the underlying marginal distributions of $X$ and $Y$. 
%In Appendix \ref{sec:ev}, we report the discrete approximation method to compute these eigenvalues following \cite{bergsma2006new}.

We explore the relationship between $\kappa(\theta)$ and the natural dependence parameter (say $\theta$) through the three non-parametric estimates of $\kappa$. The changes in the distributional behavior for the three estimators with the six
bivariate distributions, under varying degree of dependence is presented in Section \ref{subsec:dependent}.
%Appendix C, while the summary measure mean is presented in Section \ref{subsec:dependent}.
%the empirical distributions of the three estimators, with the six
%bivariate distributions from top to bottom panel.

For the cases where the exact relationship between $\kappa$ and $\theta$ are available, we explore the qualities of the estimators in capturing the true value, by 
varying the sample size, the value of the dependence parameter, etc.

Measures of dependence are, quite commonly, also used to test independence. We carry out extensive simulations, emulating various scenarios, to study the performance of the three $\kappa$ based measures among themselves, as well as with other known measures from this perspective. We present these results in Section \ref{subsec:power}.
%Results on a selective subset is presented in the main text in Section \ref{subsec:power}, while additional details are presented in Appendix D.

We have also explored computational efficiency aspects of these estimated measures from samples, and details are presented in Section \ref{subsec:comptime}.
% and in Appendix E.

In Section \ref{sec:disc} we conclude by taking a critical look at the estimators and their properties. We also discuss some future directions for potential exploration.

\section{Bergsma's covariance and correlation}\label{sec:bergsma}

\cite{ bergsma2006new} introduced a covariance $\kappa$ and the corresponding correlation coefficient $\rho$, which is defined whenever the variables have finite second moments. It can be defined as follows. 
%It has the attractive feature that $\rho=0$ if and only if the variables are independent. 
%as follows.
%which is different from the usual product-moment correlation. 
%For a bivariate random variable $(X,Y)$, a sufficient condition for $\rho^*(X,Y)$ to exist is that $\mathbb{E}(X)$ and $\mathbb{E}(Y)$ exist. We know that the ordinary correlation $\rho$ has a much stronger requirement of finite marginal variances.
%The Bergsma's correlation involves a kernel function $h_F$ which is defined as follows:

Suppose that $Z$, $Z_1$ and $Z_2$ are i.i.d.~real valued random variables with distribution function $F$ which has finite mean. Let
\begin{equation}
h_F(z_1,z_2):=-\frac{1}{2}\mathbb{E}\big(|z_1-z_2|-|z_1-Z_2|-|Z_1-z_2|+|Z_1-Z_2|\big).\label{hFBrgsm}
\end{equation}

Note that 
\begin{equation}\label{eq:iidcase}
\mathbb{E}h_F(Z_1,Z_2)=0\ \ \text{if}\ \ Z_1, Z_2\ \ \text{are i.i.d.} \ \ F.
\end{equation}

Now, let $F_1$ and $F_2$ be the marginal distributions of a bivariate random variable $(X,Y)$ which have finite variances. Let $(X_1,Y_1)$ and $(X_2,Y_2)$ be i.i.d.~copies of $(X,Y)$. Then Bergsma's covariance and correlation between $X$ and $Y$ are defined respectively by
\begin{equation}
\kappa(X,Y) :=\mathbb{E}\big[h_{F_{1}}(X_1,X_2)h_{F_{2}}(Y_1,Y_2)\big], \ \ \text{and}\ \ \rho(X,Y) :=\frac{\kappa(X,Y)}{\sqrt{\kappa(X,X)\kappa(Y,Y)}}. 
\label{kappa}
\end{equation}

It is known that, $\kappa \geq 0$ and $0 \leq \rho\leq1$. Further, $\kappa=0$ if and only if $X$ and $Y$ are independent, and $\rho=1$ iff $X$ and $Y$ are linearly related, either positively or negatively.
%	 In our measure of spatial association, we will be using the $U$-statistic estimate of $\rho^*$ denoted by $\tilde{\rho}^*$. Therefore we limit our interest to the $\tilde{\rho}^*$ defined as follows.
%\subsection{Bergsma's $U$ and $V$-statistic based estimates of $\rho^*$}
%	The $U$-statistic estimate of a parameter is an unbiased estimator based on taking averages and is used to obtain kernel-based estimators' asymptotic properties. 
%	Consider $Y_1,\ldots,Y_n$ $i.i.d$ from some distribution. Then the $U$-statistic of degree $m$, with symmetric kernel function $h$ is defined as 
%	\begin{equation}
%		U_n= {n \choose m}^{-1}\sum_{1\leq i_1 \leq\cdots\leq i_m \leq n }h(Y_{i_1},\ldots,Y_{i_m}).
%	\end{equation}

We now describe the two estimates of $\kappa$ given by \cite{ bergsma2006new}. Suppose we have $n$ observations $(x_i,y_i)$, $1\leq i \leq n$ from a bivariate distribution $F_{12}$. We first define the sample analogues of the kernel functions $h_{F_1}$ and $h_{F_2}$. 	Let 
\begin{eqnarray*}
A_{1(x_i)}&:=&\frac{1}{n}\sum_{k=1}^{n}|x_i-x_k|, 	\hspace{0.5cm}A_{2(y_i)}:=\frac{1}{n}\sum_{k=1}^{n}|y_i-y_k|,\\ 
B_{1}&:=&\frac{1}{n}\sum_{i=1}^{n}A_{1(x_i)},  	\hspace{1.4cm}B_{2}:=\frac{1}{n}\sum_{i=1}^{n}A_{2(y_i)}.
\end{eqnarray*}

Let $\hat{F}_1$ and $\hat{F}_2$ be the empirical distribution functions (edfs) of $\{x_i\}$ and $\{y_i\}$ respectively.  
Then the sample analogues of the kernel functions $h_{F_1}$ and $h_{F_2}$ are defined as: 
\begin{eqnarray*}
{h}_{\hat{F}_1}(x_i,x_j) &:=&-\frac{1}{2}\left(|x_i-x_j|-A_{1(x_i)}-A_{1(x_j)}+B_1\right),\\
%	\end{equation*}
{h}_{\hat{F}_2}(y_i,y_j)&:=&-\frac{1}{2}\left(|y_i-y_j|-A_{2(y_i)}-A_{2(y_j)}+B_2\right).
\end{eqnarray*}
\cite{bergsma2006new} defined the following $V$-statistics type estimator of $\kappa$: 
\begin{equation}
\hat{\kappa}:=\frac{1}{n^2}\sum_{i,j=1}^{n}{h}_{\hat{F}_1}(x_i,x_j){h}_{\hat{F}_2}(y_i,y_j).
\label{kappav}
\end{equation}

Note that the kernels ${h}_{\hat{F}_i}$, $i=1,2$ depend on the entire sample, and hence $\hat{\kappa}$ is not a $V$-statistic in the usual sense. A second estimate of $\kappa$ was defined by Bergsma as a $U$-statistic type estimator, using a slightly different kernel function.
% with a multiplication factor of $n/(n-1)$. 
Let
\begin{eqnarray*}
\tilde{h}_{\hat{F}_1}(x_i,x_j)&:=&-\frac{1}{2}\big(|x_i-x_j|-\frac{n}{n-1}A_{1(x_i)}-\frac{n}{n-1}A_{1(x_j)}+\frac{n}{n-1}B_1\big),\\
\tilde{h}_{\hat{F}_2}(y_i,y_j)&:=&-\frac{1}{2}\big(|y_i-y_j|-\frac{n}{n-1}A_{2(y_i)}-\frac{n}{n-1}A_{2(y_j)}+\frac{n}{n-1}B_2\big).
\end{eqnarray*}

Then Bergsma's $U$-statistics type estimate of $\kappa$ is 
\begin{equation}
\tilde{\kappa}:={n \choose 2}^{-1}\sum_{1\leq i < j \leq n}\tilde{h}_{\hat{F}_1}(x_i,x_j)\tilde{h}_{\hat{F}_2}(y_i,y_j).
\label{kappau}
\end{equation}

Again, the above is not exactly a $U$-statistics. The corresponding estimators of $\rho$ are then:
\begin{eqnarray}
\hat{\rho}(x,y)&:=&\frac{\hat{\kappa}(x,y)}{\sqrt{\hat{\kappa}(x,x)\hat{\kappa}(y,y)}},\label{rhov}\\
\tilde{\rho}(x,y)&:=&\frac{\tilde{\kappa}(x,y)}{\sqrt{\tilde{\kappa}(x,x)\tilde{\kappa}(y,y)}}.\label{rhou}
\end{eqnarray}

Note that $\hat{\kappa}$ is non-negative, but $\tilde{\kappa}$ may be negative. Moreover, if $X$ and $Y$ are independent, then $\BE[\tilde{\kappa}]=0$. \cite{bergsma2006new} obtained the asymptotic distributions of $\tilde{\kappa}$ and $\hat{\kappa}$ only under the independence of $X$ and $Y$ (that is, when $\kappa=0$), using the well known Theorem \ref{theo:standard_degen_u} on degenerate $U$-statistics given in Appendix \ref{sec:ustat}.
%\textit{Appendix A.1}. 
We shall give a more detailed proof later. We shall also deal with the case where $X$ and $Y$ are not independent.

\begin{theorem}[\cite{bergsma2006new}]\label{kappatildadegen}
Suppose the observations $\{X_i\}$ and $\{Y_i\}$ are independent with distributions $F_1$ and $F_2$ respectively. Suppose 
$\BE_{F_{1}}[ h^2_{F_1}(X_1, X_2)]+ 
\BE_{F_{2}}[ h^2_{F_2}(Y_1, Y_2)]
< \infty$,
%and $h_{F_2}$ are square integrable with respected to $F_1$ and $F_2$ respectively, 
and
%such that 
\begin{equation*}
	h_{F_1}(x_1, x_2)=\sum_{k=0}^{\infty}\lambda_kg_{1k}(x_1)g_{1k}(x_2), \  \text{and}\ \ 
	%		\end{equation*}
%	\begin{equation*}
	h_{F_2}(y_1,y_2)=\sum_{k=0}^{\infty}\eta_k g_{2k}(y_1)g_{2k}(y_2), 
\end{equation*}
where the equality is in the $L^2$ sense. 
Then  
%	then with $R_n \xrightarrow P 0 $, 
\begin{eqnarray}\label{eqn:tildekapplimit}
	n\tilde{\kappa}\xrightarrow D \sum_{i,j=0}^{\infty}\lambda_i\eta_j(Z_{ij}^2-1),
\end{eqnarray}
and
\begin{eqnarray}\label{eqn:hatkapplimit}
	n\hat{\kappa}\xrightarrow D \sum_{i,j=0}^{\infty}\lambda_i\eta_jZ_{ij}^2,
\end{eqnarray}
where \{$Z_{ij}$\} are i.i.d.~standard normal variables. 
\end{theorem}
For some illustration of the relations in (\ref{eqn:tildekapplimit}) and (\ref{eqn:hatkapplimit}) see 
Subsection  \ref{subsec:independent}.%textit{Appendix A.2}

\section{Another look at estimating $\kappa$}\label{sec:altberg}	

When $X$ and $Y$ are dependent, we might expect $\hat\kappa$ and  $\tilde\kappa$ to be asymptotically normal, but such results do not seem to be available in the literature. It is also expected that a proof could be built around the $U$-statistics central limit Theorem \ref{UCLT} given in Appendix \ref{sec:ustat}.
%\textit{Appendix A.1}. 
However, some effort would be needed to identify the first projection of an  underlying $U$-statistics in the definitions of these estimators. 

We take a slightly different approach. We express $\kappa$ in a different way. That leads to a natural estimate, say $\kappa^*$ of $\kappa$, which is a straightforward function of four simple $U$-statistics that we define. We exploit this to obtain the asymptotic distribution of $\kappa^*$ in the null (dependence) and the non-null 
(independence) cases using Theorems \ref{UCLT} and \ref{theo:standard_degen_u}. When $X$ and $Y$ are dependent,  $\kappa^*$ is asymptotic normal, and when they are independent, $\kappa^*$ with a different scaling, has the same distribution as given in (\ref{eqn:tildekapplimit}). Finally, we relate Bergsma's two estimators $\tilde\kappa$ and $\hat\kappa$, to $\kappa^*$. This helps us to show that they are also asymptotically normal under dependence, while  under independence we recover Theorem \ref{kappatildadegen}.

\subsection{Alternate expression for $\kappa$}

We rewrite the kernel function of  \eqref{hFBrgsm} as follows. Let $Z_1, Z_2$ be i.i.d. with distribution $F$. Define 
\begin{eqnarray}g_F(z)&:=& \BE_F[ |z-Z|], \label{eq:g_F}\\
%\end{equation}
%\begin{eqnarray}
g(F)&:=&\BE_{F}[|Z_1-Z_2|]=\BE_{F}[g_F(Z)], \label{eq:gF}\\
%\begin{equation}
h_F(z_1, z_2)&=&-\dfrac{1}{2} \big[|z_1-z_2|-g_F(z_1)-g_F(z_2)+g(F)\big].\label{eq:hdef}
\end{eqnarray}
%\end{equation}
We have from Equation \eqref{kappa}
$$\kappa=\BE_{F_{12}}\big[h_{F_{1}}(X_1, X_2)h_{F_{2}}(Y_1, Y_2)\big].$$
Let $(X_i, Y_i)$ , $i=1, 2$ be i.i.d.~bivariate random variables where $X_1$ and $Y_1$ have distributions $F_1$ and $F_2$ respectively, and joint distribution $F_{12}$. Suppose that 
$\BE_{F_{12}}\big[|X_1Y_1|\big] < \infty$. Define 
\begin{eqnarray*} g_{F_{12}}(x,y)&:=&\BE_{F_{12}} \big[|x-X_1| \  |y-Y_1|\big], \\
g({F_{12}}) &:=& \BE_{F_{12}} \big[ g_{F_{12}} (X_2, Y_2)\big]= \BE_{F_{12}} \big[|X_2-X_1| \  |Y_2-Y_1|\big].
\end{eqnarray*}
Note that 
\begin{equation}
g_{F_{12}}(x,y)=g_{F_{1}}(x)\  g_{F_{2}}(y) \ \ \text{and} \ \ g (F_{12})=g(F_1)\ g(F_2), \ \text{if}\  \ F_{12}=F_1\otimes F_2.\label{ind} 
\end{equation}

For the ease of writing, we shall use the shorthand notation 
$$\mu_{1}:=g(F_1), \ \mu_{2}:=g(F_2), \ \mu_{12}:=g (F_{12}), \ \text{and} \ \ \mu_{3}:=\BE_{F_{12}}\big[g_{F_{1}}(X)g_{F_{2}}(Y)\big].$$

Then $\kappa$ has the following alternate expressions in terms of the above quantities.

\begin{theorem}\label{theo:altkappa} The parameter $\kappa$ can be expressed as:
\begin{eqnarray}\kappa&=&\dfrac{1}{4}
	\big[
	\BE_{F_{12}} \big[g_{F_{12}}(X, Y)]
	-2\BE_{F_{12}}\big[g_{F_{1}}(X)g_{F_{2}}(Y)\big]+\big(\BE_{F_{1}}[ g_{F_{1}}(X)]\big)\big(\BE_{F_{2}}[ g_{F_{2}}(Y)]\big)\big]\nonumber\\
	%&=& \dfrac{1}{4}[
	%\Big[
	%\BE g_{F_{12}}(X, Y)
	%-\BE\big[g_{F_{1}}(X)g_{F_{2}}(Y)\big]
	%-\big(\BE\big[g_{F_{1}}(X)g_{F_{2}}(Y)\big]-\big(\BE g_{F_{1}}(X)\big)\big(\BE g_{F_{2}}(Y)\big)\big)\Big]\\
	%&=& \dfrac{1}{4}
	%\Big[
	%\BE g_{F_{12}}(X, Y)
	%-\big(\BCov \big( g_{F_{1}}(X_1),  g_{F_{2}}(Y_1)\big)+\big(\BE g_{F_{1}}(X)\big)\big(\BE g_{F_{2}}(Y)\big)\big)
	%-\BCov \big( g_{F_{1}}(X_1),  g_{F_{2}}(Y_1)\big)\Big]\\
	&=& \dfrac{1}{4}
	\big[\BCov_{F_{12}} \big(|X_1-X_2|, |Y_1-Y_2|\big)-2\BCov_{F_{12}} \big( g_{F_{1}}(X_1),  g_{F_{2}}(Y_1)\big)\big]\nonumber\\
	&=&\dfrac{1}{4}\big[
	\mu_{12}-2\mu_{3}+ \mu_1\mu_2\big]\nonumber\\
		&=&\int[F_{12}(x,y)-F_1(x)F_2(y)]^2dxdy, \label{kappaalt}
	\end{eqnarray} 
where $(X,Y)$ is distributed as $F_{12}$ with marginal distributions $F_1$ and $F_2$ respectively.  In particular  $F_{12}(x,y)=F_1(x) F_2(y)$ for all $x, y$,  if and only if $\kappa=0$.
\end{theorem} 

\begin{proof}
The last relation appears in \cite{bergsma2006new}. To prove the other relations, note that 
\begin{eqnarray*}
	h_{F_{1}}(X_1, X_2) h_{F_{2}}(Y_1, Y_2)&=&
	\dfrac{1}{4}
	\big(|X_1-X_2|-g_{F_{1}}(X_1)-g_{F_{1}}(X_2)+\mu_{1}\big)\\
	&&\ \ \  \times \big(|Y_1-Y_2|-g_{F_{2}}(Y_1)-g_{F_{2}}(Y_2)+\mu_{2}\big).
\end{eqnarray*}
Opening the product, we have
	\begin{eqnarray*}
	 \kappa&=&\BE[h_{F_1}(X_1,X_2)h_{F_2}(Y_1,Y_2)]\\
	 &=&\dfrac{1}{4}\BE\bigg[
	\Big(|X_1-X_2|-g_{F_{1}}(X_1)-g_{F_{1}}(X_2)+g(F_1)\Big) \times \Big(|Y_1-Y_2|-g_{F_{2}}(Y_1)-g_{F_{2}}(Y_2)+g(F_2)\Big)\bigg]\\
	&=&\frac{1}{4}\big[\BE[|X_1-X_2||Y_1-Y_2|]-\BE[|X_1-X_2|g_{F_2}(Y_1)]-\BE[|X_1-X_2|g_{F_2}(Y_2)]+\BE[|X_1-X_2|g(F_2)]\\&& \ -\BE[|Y_1-Y_2|g_{F_1}(X_1)]+\BE[g_{F_1}(X_1)g_{F_2}(Y_1)]+\BE[g_{F_1}(X_1)g_{F_2}(Y_2)]-\BE[g_{F_1}(X_1)g(F_2)]\\&& \ -\BE[|Y_1-Y_2|g_{F_1}(X_2)]+\BE[g_{F_1}(X_2)g_{F_2}(Y_1)]+\BE[g_{F_1}(X_2)g_{F_2}(Y_2)]-\BE[g_{F_1}(X_2)g(F_2)]\\&& \ -\BE[|Y_1-Y_2|g(F_1)]-\BE[g_{F_2}(Y_1)g(F_1)]-\BE[g_{F_2}(Y_2)g(F_1)]+\BE[g(F_1)g(F_2)]\big].
\end{eqnarray*}
By succesive conditioning we obtain,
\begin{eqnarray*}
	\BE[|X_1-X_2|g_{F_2}(Y_1)]&=&\BE\big[\BE[|X_1-X_2|g_{F_2}(Y_1)\mid (X_1,Y_1)]\big]\nonumber\\ 
	&=&\BE(g_{F_2}(Y_1)g_{F_1}(X_1)). 
\end{eqnarray*}
Similarly,
%	\begin{equation*}
	%		\BE[|Y_1-Y_2|g_{F_1}(X_1)]=\BE(g_{F_1}(X_1)g_{F_2}(Y_1))
	%	\end{equation*}
\begin{eqnarray*}
	\BE[|Y_1-Y_2|g_{F_1}(X_1)]&=&\BE(g_{F_1}(X_1)g_{F_2}(Y_1)).\\
	\BE[|X_1-X_2|g_{F_2}(Y_2)]&=&\BE(g_{F_2}(Y_2)g_{F_1}(X_2)).\\
	\BE[|Y_1-Y_2|g_{F_1}(X_2)]&=&\BE(g_{F_1}(X_2)g_{F_2}(Y_2)).\\
\end{eqnarray*}
%	Similarly,
%	\begin{equation*}
	%		\BE[|Y_1-Y_2|g_{F_1}(X_2)]=\BE(g_{F_1}(X_2)g_{F_2}(Y_2))
	%	\end{equation*}
Also,
\begin{eqnarray*}
	\BE[|X_1-X_2|g(F_2)]&=&	\BE\big[\BE[|X_1-X_2|g(F_2)\mid (X_1,Y_1)]\big]\nonumber\\ 
	&=&\BE(g(F_2)g_{F_1}(X_1))\\ &=&g(F_1)g(F_2).
\end{eqnarray*}
similarly,
\begin{equation*}
	\BE[|Y_1-Y_2|g(F_1)]=g(F_1)g(F_2).
\end{equation*}
Since $(X_i,Y_i) ,i,j=1,2$ are i.i.d we get
$$\BE[g_{F_1}(X_2)g_{F_2}(Y_1)]=\BE[g_{F_1}(X_1)g_{F_2}(Y_2)]=g(F_1)g(F_2).$$
%	Also,
%\begin{eqnarray}
%	\BE[g_{F_1}(X_2)g(F_2)]&=&g(F_2)\BE[g_{F_1}(X_2)]\nonumber\\
%	&=&g(F_1)g(F_2)
%\end{eqnarray}
%Similarly,
%	\begin{equation*}
	%		\BE[g_{F_1}(X_2)g(F_2)]=	\BE[g_{F_2}(Y_1)g(F_1)]=g(F_1)g(F_2)
	%	\end{equation*}
%	We have
%	\begin{eqnarray}
	%		\kappa
	%		&=&\frac{1}{4}\big[\BE[|X_1-X_2||Y_1-Y_2|]-\BE[|X_1-X_2|g_{F_2}(Y_1)]-\BE[|X_1-X_2|g_{F_2}(Y_2)]+\BE[|X_1-X_2|g(F_2)]\nonumber\\&-&\BE[|Y_1-Y_2|g_{F_1}(X_1)]+\BE[g_{F_1}(X_1)g_{F_2}(Y_1)]+\BE[g_{F_1}(X_1)g_{F_2}(Y_2)]-\BE[g_{F_1}(X_1)g(F_2)]\nonumber\\&-&\BE[|Y_1-Y_2|g_{F_1}(X_2)]+\BE[g_{F_1}(X_2)g_{F_2}(Y_1)]+\BE[g_{F_1}(X_2)g_{F_2}(Y_2)]-\BE[g_{F_1}(X_2)g(F_2)]\nonumber\\&+&\BE[|Y_1-Y_2|g(F_1)]-\BE[g_{F_2}(Y_1)g(F_1)]-\BE[g_{F_2}(Y_2)g(F_1)]+\BE[g(F_1)g(F_2)]\big]\nonumber
	%	\end{eqnarray}
By the above equations we can write,
%	\begin{eqnarray}
	%		\kappa
	%		&=&\frac{1}{4}\big[\BE[|X_1-X_2||Y_1-Y_2|]-\BE[g_{F_1}(X_1)g_{F_2}(Y_1)]-\BE[g_{F_1}(X_2)g_{F_2}(Y_2)]\color{red}{+g(F_1)g(F_2)}\nonumber\\&-&\color{red}{\BE[g_{F_1}(X_1)g_{F_2}(Y_1)]+\BE[g_{F_1}(X_1)g_{F_2}(Y_1)]}\color{black}{+\BE[g_{F_1}(X_1)g_{F_2}(Y_2)]}\color{red}{-g(F_1)g(F_2)}\nonumber\\&-&\color{red}{\BE[g_{F_1}(X_2)g_{F_2}(Y_2)]}\color{black}{+\BE[g_{F_1}(X_2)g_{F_2}(Y_1)]}+\color{red}{\BE[g_{F_1}(X_2)g_{F_2}(Y_2)]}\color{red}{-g(F_1)g(F_2)}\nonumber\\&+&\color{red}{g(F_1)g(F_2)}\color{black}{-g(F_1)g(F_2)}\color{red}{-g(F_1)g(F_2)+g(F_1)g(F_2)}\big]
	%	\end{eqnarray}
\begin{eqnarray}
	\kappa
	&=&\frac{1}{4}\big[\BE[|X_1-X_2||Y_1-Y_2|]-\BE[g_{F_1}(X_1)g_{F_2}(Y_1)]-\BE[g_{F_1}(X_2)g_{F_2}(Y_2)]{+g(F_1)g(F_2)}\nonumber\\ && -{\BE[g_{F_1}(X_1)g_{F_2}(Y_1)]+\BE[g_{F_1}(X_1)g_{F_2}(Y_1)]}{+\BE[g_{F_1}(X_1)g_{F_2}(Y_2)]}{-g(F_1)g(F_2)}\nonumber\\&& \ -{\BE[g_{F_1}(X_2)g_{F_2}(Y_2)]}{+\BE[g_{F_1}(X_2)g_{F_2}(Y_1)]}+{\BE[g_{F_1}(X_2)g_{F_2}(Y_2)]}{-g(F_1)g(F_2)}\nonumber\\&& \ +{g(F_1)g(F_2)}{-g(F_1)g(F_2)}{-g(F_1)g(F_2)+g(F_1)g(F_2)}\big].\nonumber
\end{eqnarray}
Upon cancellations we will be left with the terms
\begin{eqnarray}
	\kappa&=&\frac{1}{4}\big[\BE[|X_1-X_2||Y_1-Y_2|]-\BE[g_{F_1}(X_1)g_{F_2}(Y_1)]-\BE[g_{F_1}(X_2)g_{F_2}(Y_2)]\nonumber\\&& \ + \ \BE[g_{F_1}(X_2)g_{F_2}(Y_1)]+\BE[g_{F_1}(X_1)g_{F_2}(Y_2)]-g(F_1)g(F_2)\big]\nonumber\\
	%		&=&\frac{1}{4}\big[\BE[|X_1-X_2||Y_1-Y_2|]-\BE[g_{F_1}(X_1)g_{F_2}(Y_1)]-\BE[g_{F_1}(X_2)g_{F_2}(Y_2)]\nonumber\\&& \ + \ \BE[g_{F_1}(X_2)]\BE[g_{F_2}(Y_1)]+\BE[g_{F_1}(X_1)]\BE[g_{F_2}(Y_2)]-g(F_1)g(F_2)\big]\nonumber\\
	&=&\frac{1}{4}\big[\BE[|X_1-X_2||Y_1-Y_2|]-\BE[g_{F_1}(X_1)g_{F_2}(Y_1)]-\BE[g_{F_1}(X_2)g_{F_2}(Y_2)]\nonumber\\&& \ + \ g(F_1)g(F_2)+g(F_1)g(F_2)-g(F_1)g(F_2)\big]\nonumber\\
	&=&\frac{1}{4}\big[\BE[|X_1-X_2||Y_1-Y_2|]-\BE[g_{F_1}(X_1)g_{F_2}(Y_1)]-\BE[g_{F_1}(X_2)g_{F_2}(Y_2)]+g(F_1)g(F_2)\big]\nonumber\\ 
	&=&\frac{1}{4}\big[\BE[g_{F_{12}}(X_1,Y_1)]-2\BE[g_{F_1}(X_1)g_{F_2}(Y_1)]+g(F_1)g(F_2)\big] \nonumber\\
	&=&\frac{1}{4}\big[g(F_{12})-2\BE[g_{F_1}(X)g_{F_2}(Y)]+g(F_1)g(F_2)\big]\nonumber\\
	&=&\frac{1}{4}\big[\mu_{12}-2\mu_{3}+\mu_{1}\mu_{2}\big].\nonumber
\end{eqnarray}

%	We get,
%	\begin{eqnarray}
	%		\kappa&=&\frac{1}{4}\big[\BE[|X_1-X_2||Y_1-Y_2|]-\BE[g_{F_1}(X_1)g_{F_2}(Y_1)]-\BE[g_{F_1}(X_2)g_{F_2}(Y_2)]+g(F_1)g(F_2)\big]\nonumber\\ 	&=&\frac{1}{4}\big[\BE[g_{F_{12}}(X_2,Y_2)]-2\BE[g_{F_1}(X_1)g_{F_2}(Y_1)]+g(F_1)g(F_2)\big].\nonumber
	%	\end{eqnarray}
%	%	\kappa&=& \frac{1}{4}\big[E\big(g_{F_1}(X_2)g_{F_2}(Y_2)-g_{F_2}(Y_1)-g_{F_1}(X_1)g_{F_2}(Y_2)+g(F_1)g(F_2)\nonumber\\&-&g_{F_1}(X_1)g_{F_2}(Y_1)+g_{F_1}(X_1)g_{F_2}(Y_1)+g_{F_1}(X_1)g_{F_2}(Y_2)-g_{F_1}(X_1)g(F_2)\nonumber\\&-&g_{F_1}(X_2)g_{F_2}(Y_2)+g_{F_1}(X_2)g_{F_2}(Y_1)+g_{F_1}(X_2)g_{F_2}(Y_2)-g_{F_1}(X_2)g(F_2)\nonumber\\&+&g(F_1)g_{F_2}(Y_2)-g(F_1)g_{F_2}(Y_1)-g(F_1)g_{F_2}(Y_2)+g(F_1)g(F_2)\big)\big] \nonumber
\noindent Further the second relation is obtained as follows
\begin{eqnarray*}\kappa&=&\dfrac{1}{4}
	\Big[
	\BE g_{F_{12}}(X_1, Y_1)
	-2\BE\big[g_{F_{1}}(X_1)g_{F_{2}}(Y_1)\big]+\big(\BE g_{F_{1}}(X_1)\big)\big(\BE g_{F_{2}}(Y_1)\big)\Big]\\
	&=& \dfrac{1}{4}
	\Big[
	\BE g_{F_{12}}(X_1, Y_1)
	-\BE\big[g_{F_{1}}(X_1)g_{F_{2}}(Y_1)\big]
	-\big(\BE\big[g_{F_{1}}(X_1)g_{F_{2}}(Y_1)\big]-\big(\BE g_{F_{1}}(X_1)\big)\big(\BE g_{F_{2}}(Y_1)\big)\big)\Big]\\
	&=& \dfrac{1}{4}
	\Big[
	\BE g_{F_{12}}(X_1, Y_1)
	-\big(\BCov \big( g_{F_{1}}(X_1),  g_{F_{2}}(Y_1)\big)+\BE g_{F_{1}}(X_1)\BE g_{F_{2}}(Y_1)\big)\big)
	-2\BCov \big( g_{F_{1}}(X_1),  g_{F_{2}}(Y_1)\big)\Big]\\
	&=& \dfrac{1}{4}
	\Big[\BE[|X_1-X_2||Y_1-Y_2|]- \BE[|X_1-X_2|]\BE[|Y_1-Y_2|]-2\BCov \big( g_{F_{1}}(X_1),  g_{F_{2}}(Y_1)\big)\Big]\\
	&=& \dfrac{1}{4}
	\Big[\BCov \big(|X_1-X_2|, |Y_1-Y_2|\big)-2\BCov \big( g_{F_{1}}(X_1),  g_{F_{2}}(Y_1)\big)\Big].
\end{eqnarray*}
%\begin{eqnarray*}
%	\BE_{F_{12}} \big[|X_1-X_2|g_{F_{2}}(Y_1)\big] &=& \BE_{F_{12}} \big[g_{F_{1}}(X_1)g_{F_{2}}(Y_1)\big],\\
%	\BE_{F_{12}} \big[|Y_1-Y_2|g_{F_{1}}(X_1)\big] &=& \BE_{F_{12}} \big[g_{F_{1}}(X_1)g_{F_{2}}(Y_1)\big], \\
%	\BE_{F_{12}} \big[|X_1-X_2| \  |Y_1-Y_2|\big]&=& \BE_{F_{12}} \big[g_{F_{12}}(X_1, Y_1)\big].
%\end{eqnarray*}
%
%The result then follows by using the above relations, and after a host of cancellations.
This completes the proof.
\end{proof}
\noindent \textbf{Example 1}. 
%It is pertinent to see how the parameter $\kappa$ relates to other known indices of dependence in bivariate models. 
%As an illustration, 
Consider the bivariate exponential model GBED-I of  
%	\subsubsection*{Gumbels Bivariate Exponential Distribution (GBED-I)}
\cite{gumbel1960bivariate}. In this model, $(X,Y)$ has the joint distribution function and the joint density as
%$F(x,y)$ given by
\begin{eqnarray*}F_{12}(x,y)&=&1-e^{-x}-e^{-y}+e^{-(x+y+\theta xy)}, \hspace{0.5cm} 0<x, y<\infty, \\
f_{12}(x,y)&=&[(1+\theta x)(1+\theta y)]e^{-(x+y+\theta xy)},\hspace{0.5cm} 0<x, y<\infty.
\end{eqnarray*}
The marginal distributions of $X$ and $Y$ are standard exponentials, and $0 \leq \theta \leq 1$ is the dependence parameter. Let us denote the value of $\kappa$ in this model by $\kappa(\theta)$. Note that, $X$ and $Y$ independent if and only if $\theta=0$, and in that case $\kappa (0)=0$. For this model, it is convenient to use Equation (\ref{kappaalt}).
%recall and use yet another general expression for $\kappa$ from \cite{bergsma2006new}.
	%gives an alternate formulation of the covariance $\kappa$ given by the following lemma. 
%	Define
%	\[
%	\gamma(x,y):= 
%	\begin{cases}
%		0& x>y\\
%		\frac{1}{2}              & x=y\\
%		1& x<y
%	\end{cases}
%	\]
%	\begin{lemma}[\cite{bergsma2006new}]
%		If $h_F$ exists it exits of $\mathbb{R}^2$. It is then continuous and positive and has the representation
%		\begin{equation}
%			h_F(z_1,z_2)=\int_{-\infty}^{\infty}[\gamma(z_1,\omega)-F(\omega)][\gamma(z_2,\omega)-F(\omega)]d\omega, \hspace{0.5cm}\forall z_1, z_2.
%		\end{equation}
%	\end{lemma}
%Suppose $(X, Y)$ has the bivariate distribution $F_{12}$ with marginal distributions $F_1$ and $F_2$ respectively. Then  
	%\begin{eqnarray}
		%\kappa&=&
		%%\mathbb{E}[h_{F_{1}}(X_1,X_2)h_{F_{2}}(Y_1,Y_2)]\nonumber\\
		%%&=&
		%\int[F_{12}(x,y)-F_1(x)F_2(y)]^2dxdy. \label{kappaalt}
	%\end{eqnarray}
%%\begin{lemma}[\cite{bergsma2006new}]\label{lem:kappa_alt} Suppose $(X, Y)$ has the bivariate distribution $F_{12}$ with marginal distributions $F_1$ and $F_2$ respectively. Then  
	%\begin{eqnarray}
		%\kappa&=&
		%%\mathbb{E}[h_{F_{1}}(X_1,X_2)h_{F_{2}}(Y_1,Y_2)]\nonumber\\
		%%&=&
		%\int[F_{12}(x,y)-F_1(x)F_2(y)]^2dxdy. \label{kappaalt}
	%\end{eqnarray}
%In particular, if $X$ and $Y$ are independent, then $\kappa=0.$
%\end{lemma}
%Here $\theta$ is the association parameter. 
% (say) for (GBED-I) can be obtained from Equation 
Some easy computation using this equation leads to:
\begin{eqnarray*}
	\kappa(\theta)&=&\int_{0}^{\infty}\int_{0}^{\infty}[1-e^{-x}- e^{-y}+e^{-x-y-\theta xy}-(1-e^{-x})(1-e^{-y})]^2dxdy\\
%	&=&\int_{0}^{\infty}\int_{0}^{\infty}[1-\exp(-x)-\exp(-y)+\exp(-x-y-\theta xy)-[1-\exp(-y)-\exp(-x)+\exp(-x)\exp(-y)]]^2dxdy\\
	&=&\int_{0}^{\infty}\int_{0}^{\infty}[e^{-(x+y+\theta xy)}-e^{-(x+y)}]^2dxdy\\
	&=&\int_{0}^{\infty}\int_{0}^{\infty}[e^{-2(x+y+\theta xy)}+e^{-2(x+y)}-2e^{-(2(x+y)+\theta xy)}]dxdy\\&=&\frac{1}{2}\int_{0}^{\infty}\frac{\exp(-2y)}{1+\theta y}dy+\frac{1}{4}-2\int_{0}^{\infty}\frac{e^{-2y}}{2+\theta y}dy
%	\\&=& -\frac{1}{2\theta}e^{\frac{2}{\theta}}Ei(\frac{-2}{\theta})+\frac{1}{4}+ \frac{2}{\theta}e^{\frac{4}{\theta}}Ei(\frac{-4}{\theta}), \hspace{1cm} \theta \neq 0
\end{eqnarray*}
Let us define the function
%where $Ei(x)$ is the exponential integral defined as,
$$G(x):=\int_{1}^{\infty}\frac{e^{-xt}}{t}dt.$$
On simplification we get
$$\frac{1}{2}\int_{0}^{\infty}\frac{\exp(-2y)}{1+\theta y}dy=\frac{1}{2\theta}e^{\frac{2}{\theta}}\int_{1}^{\infty}\frac{e^{-\frac{2}{\theta}t}}{t}dt=\frac{1}{2\theta}e^{\frac{2}{\theta}}G(2/\theta),$$
and similarly
$$\int_{0}^{\infty}\frac{e^{-2y}}{2+\theta y}=\frac{1}{\theta}e^{\frac{4}{\theta}}G(4/\theta).$$
Therefore
%$$\kappa(\theta)=\frac{1}{2\theta}e^{\frac{2}{\theta}}G(2/\theta)+\frac{1}{4}-2\frac{1}{\theta}e^{\frac{4}{\theta}}G(4/\theta), \hspace{1cm} \theta \neq 0. $$

	\[\kappa(\theta)=
\begin{cases}
	0& \text{if}\ \   \theta=0,\\
	\frac{1}{2\theta}e^{\frac{2}{\theta}}G(2/\theta)+\frac{1}{4}-\frac{2}{\theta}e^{\frac{4}{\theta}}G(4/\theta), & \text{if} \ \ \theta \neq 0.			\end{cases}
\]
	\begin{figure}[H]
	\centering
	\includegraphics[width=0.5\linewidth]{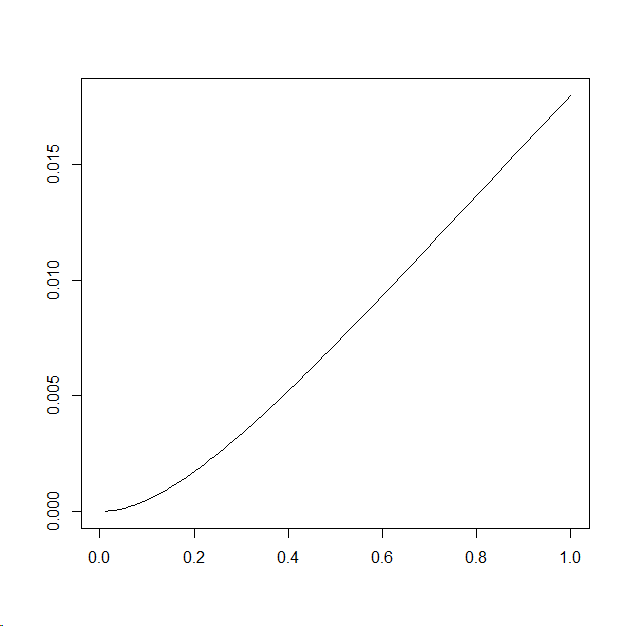}
	\caption{Plot of ($\theta$, $\kappa(\theta)$) for Bivariate exponential distribution}
	\label{fig:kappathethaexp}
\end{figure}
%See Figure \ref{fig:kappa_theta}
\vspace{0.2cm}
\begin{lemma}
	$\kappa(\theta)$ is an increasing function of $\theta\in (0, \ 1)$ when $(X, Y)$ follows GBED-I.
\end{lemma}
\begin{proof}
	For this it is enough to show that $\kappa^\prime(\theta)>0$. It is easy to check that,
	\begin{equation*}
		G^\prime(2/\theta)=\dfrac{1}{\theta}e^{{-2/\theta}} \ \ \text{and} \ \ G^\prime(4/\theta)=\dfrac{1}{\theta}e^{{-4/\theta}}.
	\end{equation*}
	Hence,
	\begin{eqnarray}
		\dfrac{d}{d\theta}\bigg(\frac{1}{2\theta}e^{\frac{2}{\theta}}G(2/\theta)\bigg)&=&\frac{1}{2\theta}e^{\frac{2}{\theta}}(\dfrac{1}{\theta}e^{{-2/\theta}})+\frac{1}{2\theta}\left( e^{2/\theta}\times\frac{-2}{\theta^2}\right)G(2/\theta)-\frac{1}{2\theta^2}e^{\frac{2}{\theta}}G(2/\theta)\nonumber\\
		&=&\frac{1}{2\theta^2}-e^{{2/\theta}}G(2/\theta)\left(\frac{1}{\theta^3}+\frac{1}{2\theta^2}\right).
		%	&=&-\frac{1}{4}-\bigg(\frac{1}{\theta^3}+\frac{1}{2\theta^2}\bigg)e^{\frac{2}{\theta}}\int_{1}^{\infty}\dfrac{e^{\frac{-2t}{\theta}}}{t}dt
	\end{eqnarray}
	Similarly,
	\begin{eqnarray}
		\dfrac{d}{d\theta}\left(\frac{2}{\theta}e^{\frac{4}{\theta}}G(4/\theta)\right)&=&\frac{2}{\theta}e^{\frac{4}{\theta}}(\dfrac{1}{\theta}e^{{-4/\theta}})+\frac{2}{\theta}\left( e^{4/\theta}\times\frac{-4}{\theta^2}\right)G(4/\theta)-\frac{2}{\theta^2}e^{\frac{4}{\theta}}G(4/\theta)\nonumber\\
		&=&\frac{2}{\theta^2}-e^{\frac{4}{\theta}}G(4/\theta)\left(\frac{8}{\theta^3}+\frac{2}{\theta^2}\right).
	\end{eqnarray}
	Combining the above two expression we get the first derivative of $\kappa$ as:
	\begin{eqnarray}
		\kappa^\prime(\theta)
		&=&\frac{1}{2\theta^2}-e^{{2/\theta}}G(2/\theta)\left(\frac{1}{\theta^3}+\frac{1}{2\theta^2}\right)-\left[\frac{2}{\theta^2}-e^{\frac{4}{\theta}}G(4/\theta)\left(\frac{8}{\theta^3}+\frac{2}{\theta^2}\right)\right]\nonumber\\
		&=&\left(\frac{1}{2\theta^2}-\frac{2}{\theta^2}\right)+\bigg(\frac{8}{\theta^3}+\frac{2}{\theta^2}\bigg)e^{\frac{4}{\theta}}G(4/\theta)-\bigg(\frac{1}{\theta^3}+\frac{1}{2\theta^2}\bigg)e^{\frac{2}{\theta}}G(2/\theta)\nonumber\\
		&=&-\frac{3}{2\theta^2}+\bigg(\frac{8}{\theta^3}+\frac{2}{\theta^2}\bigg)e^{\frac{4}{\theta}}G(4/\theta)-\bigg(\frac{1}{\theta^3}+\frac{1}{2\theta^2}\bigg)e^{\frac{2}{\theta}}G(2/\theta)\label{kappadash}
	\end{eqnarray}
From Figure \ref{fig:kappadrvtvbve}, we can observe that the values of $\kappa^\prime(\theta)>0$ for $\theta \in (0,1)$. The analytical proof of this is given as follows.
\begin{figure}[H]
	\centering
	\includegraphics[width=0.5\linewidth]{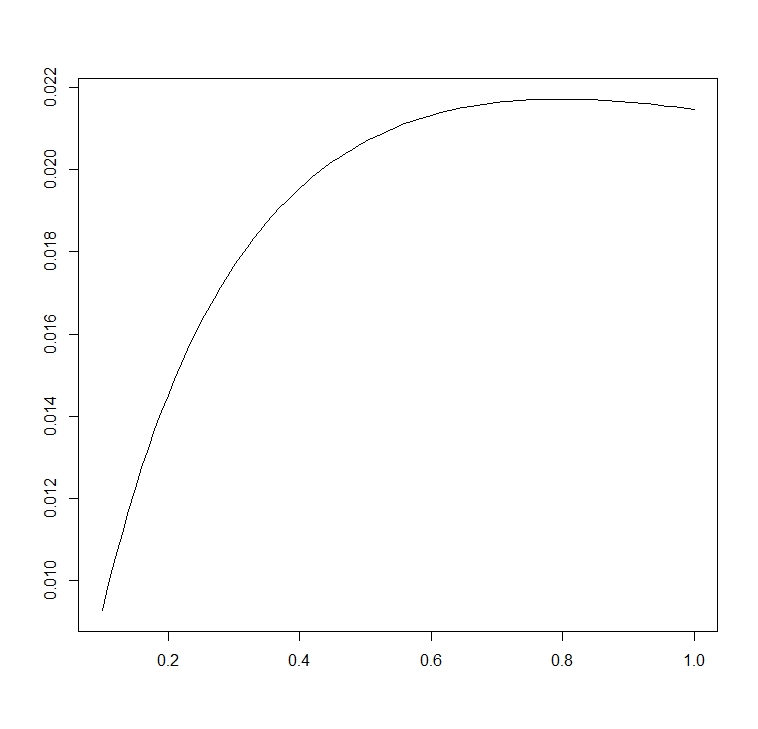}
	\caption{Plot of ($\theta$, $\kappa^\prime(\theta)$), $\theta \in (0,1)$ for GBED-I}
	\label{fig:kappadrvtvbve}
\end{figure}
\noindent	Let $$f(\theta)=2\theta^2 \kappa^\prime(\theta)=\left[-3+\bigg(\frac{16}{\theta}+4\bigg)e^{\frac{4}{\theta}}G(4/\theta)-\bigg(\frac{2}{\theta}+1\bigg)e^{\frac{2}{\theta}}G(2/\theta)\right].$$
	We substitute $\frac{2}{\theta}=x$ and write
	\begin{equation}
		F(x)=\left[-3+(8x+4)e^{2x}G(2x)-(x+1)e^{x}G(x)\right].
	\end{equation}	
	
	\noindent Note that $x\geq2$. We have to show that $F(x)>0$ for all $x \geq 2$.
	\subsubsection*{$F(x)>0$ for large $x$:}
	By repeated integration by parts, we can write,
	$$G(x)=\frac{e^{-x}}{x}-\frac{e^{-x}}{x^2}+\frac{2e^{-x}}{x^3}-6\frac{e^{-x}}{x^4}+24\int_{x}^{\infty}\frac{e^{-t}}{t^5}dt,$$
	$$e^xG(x)=\frac{1}{x}-\frac{1}{x^2}+\frac{2}{x^3}-\frac{6}{x^4}+24\int_{x}^{\infty}\frac{e^{-(t-x)}}{t^5}dt.$$
	Hence we have
	\begin{eqnarray*}
		F(x)&=&\left[-3+(8x+4)e^{2x}G(2x)-(x+1)e^{x}G(x)\right]\\
		&=&-3+(8x+4)\left(\frac{1}{2x}-\frac{1}{4x^2}+\frac{2}{8x^3}-\frac{6}{16x^4}+24\int_{2x}^{\infty}\frac{e^{-(t-2x)}}{t^5}dt\right)\\ && -(x+1)\left(\frac{1}{x}-\frac{1}{x^2}+\frac{2}{x^3}-\frac{6}{x^4}+24\int_{x}^{\infty}\frac{e^{-(t-x)}}{t^5}dt\right)\\
		&=&\frac{2}{x^3}+\frac{9}{2x^4}+(8x+4)24\int_{2x}^{\infty}\frac{e^{-(t-2x)}}{t^5}dt-24(x+1)\int_{x}^{\infty}\frac{e^{-(t-x)}}{t^5}dt>0
	\end{eqnarray*}
	as $ x \rightarrow \infty$. Since the last terms are of $O(x^{-4})$, $F(x)>0$ for large $x$. Note from Figure \ref{fig:kappafx}, that $F(x)$ is a decreasing function of $x$. We prove this using the first derivative of $F(x)$. 
	\begin{figure}[H]
		\centering
		\includegraphics[width=0.5\linewidth]{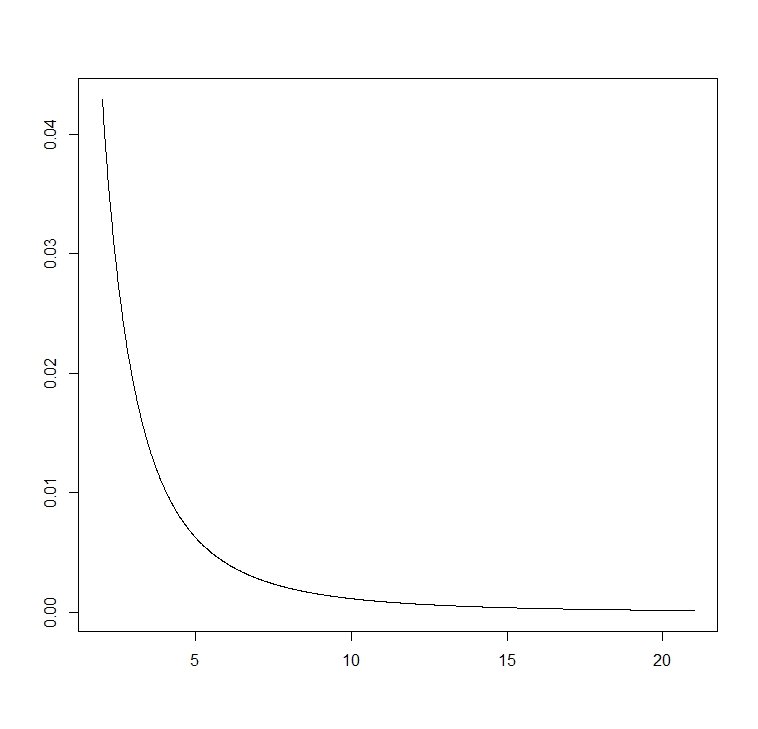}
		\caption{Plot of ($x$,$F(x)$) for $x\geq 2$.}
		\label{fig:kappafx}
	\end{figure}
	\subsubsection*{$F$ is decreasing:}
	Now consider the first derivative of $F(x)$. Clearly,
	$$G'(x)=-\dfrac{e^{-x}}{x}.$$
	Define,
	\begin{equation}
		H(x)=(x+1)e^{x}G(x)
	\end{equation}
	Then
	\begin{eqnarray}
		H'(x)&=&(x+1)e^{x}\left(\frac{-e^{-x}}{x}\right)+(x+1)e^{x}G(x)+e^{x}G(x)\nonumber\\
		&=&(x+2)e^{x}G(x)-\left(\frac{x+1}{x}\right)
	\end{eqnarray}
	We have
	\begin{eqnarray*}
		F(x)&=&-3+(8x+4)e^{2x}G(2x)-(x+1)e^{x}G(x)\\
		&=&-3+4(2x+1)e^{2x}G(2x)-(x+1)e^{x}G(x)\\
		&=&-3+4H(2x)-H(x)
	\end{eqnarray*}
	Hence,
	\begin{eqnarray*}
		F'(x)&=&8H'(2x)-H'(x)\\
		&=&8\left[(2x+2)e^{2x}G(2x)-\left(\frac{2x+1}{2x}\right)\right]-\left[(x+2)e^{x}G(x)-\left(\frac{x+1}{x}\right)\right]\\
		&=&16((x+1)e^{2x}G(2x))-8\left(\frac{2x+1}{2x}\right)-(x+2)e^{x}G(x)+\left(\frac{x+1}{x}\right)\\
		&=&16((x+1)e^{2x}G(2x))-(x+2)e^{x}G(x)-8-\frac{4}{x}+1+\frac{1}{x}\\
		&=&16((x+1)e^{2x}G(2x))-(x+2)e^{x}G(x)-7-\frac{3}{x}
	\end{eqnarray*}
	We use the following inequality from \cite{Luke1969} (see pg-$201$)
	\begin{equation}
		\dfrac{x(x+3)}{x^2+4x+2}<xe^{x}\int_{x}^{\infty}\left(t^{-1}e^{-t}\right)dt<\dfrac{x^2+5x+2}{x^2+6x+6}, \ \ 0<x<\infty.
	\end{equation}
	\begin{eqnarray*}
		xF'(x)&=&\frac{16}{2}(x+1)\left(2xe^{2x}G(2x)\right)-(x+2)\left(xe^{x}G(x)\right)-7x-3\\
		&>&8(x+1)\dfrac{2x(2x+3)}{(2x)^2+4(2x)+2}-(x+2)\left(\dfrac{x^2+5x+2}{x^2+6x+6}
		\right)-7x-3\\
		&=&8(x+1)\dfrac{4x^2+6x}{4x^2+8x+2}-(x+2)\left(1-\dfrac{x+4}{x^2+6x+6}\right)-7x-3\\
		&=&8(x+1)\left(1-\dfrac{2x+2}{4x^2+8x+2}\right)-(x+2)+\left(\dfrac{(x+2)(x+4)}{x^2+6x+6}\right)-7x-3\\
		&=&8x+8-x-2-7x-3-\left(\dfrac{8(x+1)2(x+1)}{2(2x^2+4x+1)}\right)+\left(\dfrac{(x^2+6x+8)}{x^2+6x+6}\right)\\
		&=&3-8\left(\dfrac{x^2+2x+1}{2x^2+4x+1}\right)+\left(1+\dfrac{2}{x^2+6x+6}\right)\\&=&4-4\left(\dfrac{2x^2+4x+2}{2x^2+4x+1}\right)+\left(\dfrac{2}{x^2+6x+6}\right)\\&=&4-4\left(1+\dfrac{1}{2x^2+4x+1}\right)+\left(\dfrac{2}{x^2+6x+6}\right)\\
		&=&\dfrac{-4}{2x^2+4x+1}+\dfrac{2}{x^2+6x+6}<0, \ \ \ x\geq 2.
	\end{eqnarray*}
	Hence $F'(x)<0, \ \ x\geq 2$. Therefore the result, follows immediately since $F$ is decreasing, and $F(x) > 0$ for all large $x$.
\end{proof}

%where $Ei(x)$ is the exponential integral defined as,
%$$Ei(x)=-\int_{-x}^{\infty}\frac{e^{-t}}{t}dt$$
%\begin{eqnarray*}
	%\kappa(0)&=&\frac{1}{2}\int_{0}^{\infty}\frac{\exp(-2y)}{1+\theta y}dy+\frac{1}{4}-2\int_{0}^{\infty}\frac{e^{-2y}}{\theta y+2}dy\\
	%&=&\frac{1}{2}\int_{0}^{\infty}{\exp(-2y)}dy+\frac{1}{4}-2\int_{0}^{\infty}\frac{e^{-2y}}{2}dy=0
%\end{eqnarray*}

\noindent \textbf{Example 2}. Now let us consider the case where $F_{12}$ is the bivariate normal distribution. We need the following lemma for the calculation of the various components of $\kappa$.
% Detailed proof of this lemma and the next one is available in \cite{??}.
\begin{lemma}\label{bvncalc}Suppose $(Z_1,Z_2)~N(0,0,1,1,\theta)$. Then 
	\begin{enumerate}[label=(\roman*)]
	\item $\BE\phi(Z_1)=\BE\phi(Z_2)=\dfrac{1}{2\sqrt{\pi}},$\label{part11}\\
	\item $\BE(Z_1Z_2)=\theta,$\label{part12}\\
	\item $\BE(Z_2\phi(Z_1))=\BE(Z_1\phi(Z_2))=0,$\label{part13}\\
	\item  $\BE[Z_1Z_2\Phi(Z_2)]=\BE[Z_1Z_2\Phi(Z_1)]=\dfrac{\theta}{2},$\label{part14}\\
	\item $\BE(\phi(Z_1)\phi(Z_2))=\dfrac{1}{2\pi}\sqrt{\dfrac{1}{(4-\theta^2)}},$\label{part15}\\
	\item $\BE[Z_2\Phi(Z_2)\phi(Z_1)=\dfrac{2-\theta^2}{4\pi\sqrt{4-\theta^2}},$\label{part16}\\
	\item $\BE(Z_1Z_2\Phi(Z_1)\Phi(Z_2))=
	\dfrac{1}{4}+\dfrac{1}{2\pi}\sin^{-1}\bigg(\dfrac{\theta}{2}\bigg)+\dfrac{\theta}{4\pi\sqrt{4-\theta^2}}+\dfrac{\theta(2-\theta^2)}{4\pi\sqrt{4-\theta^2}}+\dfrac{(1-\theta^2)(2-\theta^2)}{4\pi\sqrt{4-\theta^2}}.$\label{part17}\\
\end{enumerate}
\end{lemma}
\begin{proof}
	
	For $(Z_1,Z_2) \sim N(0,0,1,1,\theta)$ it is easy to verify parts \ref{part11}, and \ref{part12}.
	Expression in part \ref{part13} can be verified by writing $Z_1=\theta Z_2+\sqrt{1-\theta^2}X$, where $X$ is independent of $Z_2$ and $X$ is standard normal random variable.\\
%	 we get,
%		$$\BE(Z_2\phi(Z_1))=\BE(Z_1\phi(Z_2))=0$$
		%	\begin{enumerate}
			%		 $$\BE[Z_1Z_2\Phi(Z_2)]=	\BE[Z_1Z_2\Phi(Z_1)]=\dfrac{\theta}{2}$$
			%		write $Z_1=\theta Z_2+\sqrt{1-\theta^2}X$
			%		Then 
			\indent Proof of part \ref{part14}.		\begin{eqnarray*}
				\BE[Z_1Z_2\Phi(Z_2)]&=&\BE[(\theta Z_2+\sqrt{1-\theta^2}X)Z_2\Phi(Z_2)], \hspace{1cm}\BE(X)=0\\
				&=&\theta\BE(Z_2^2\Phi(Z_2))			\end{eqnarray*}
			\begin{eqnarray*}
				\BE(Z_2^2\Phi(Z_2))&=&\int_{-\infty}^{\infty}t^2\Phi(t)\phi(t)dt=\int_{-\infty}^{\infty}t^2\phi(t)\big[\int_{-\infty}^{t}\phi(x)dx\big]dt,\hspace{0.5cm}x\leq t\\&=&\int_{-\infty}^{\infty}\phi(x)\big(\int_{x}^{\infty}t^2\phi(t)dt\big)dx=\int_{-\infty}^{\infty}\phi(x)\big(\int_{x}^{\infty}t (t\phi(t))dt\big)dx,\hspace{1cm} \\
				%			&=&\int_{-\infty}^{\infty}\phi(x)\big[t (-\phi(t))\Bigg|_{x}^{\infty}-\int_{x}^{\infty}(-\phi(t)dt)\big]dx\\
				&=&\int_{-\infty}^{\infty}\phi(x)\big[x\phi(x)+\int_{x}^{\infty}\phi(t)dt\big]dx\\&=&\int_{-\infty}^{\infty}x\phi^2(x)dx+\int_{-\infty}^{\infty}\phi(x)[1-\Phi(x)]dx=0+\int_{-\infty}^{\infty}\phi(x)[1-\Phi(x)]dx\\&=&\BE[1-\Phi(Z)],
			\end{eqnarray*}
			where $Z$ follows standard normal. We know that $\Phi(Z)\sim U(0,1)$
			so, $\BE[1-\Phi(Z)]=\dfrac{1}{2}$
			Therefore
			$$\BE[Z_1Z_2\Phi(Z_2)]=\theta\BE(Z_2^2\Phi(Z_2))=\dfrac{\theta}{2}$$
			
			Proof of part \ref{part15} 
			Recall multivariate normal density, $X \sim N(0,\Sigma)$. Then with $\Sigma$ positive definite
			$$f(x)=\dfrac{1}{(2\pi)^{k/2}|\Sigma|^{1/2}}e^{-\dfrac{1}{2}x'\Sigma^{-1}x} \text{and},\hspace{0.5cm}\int e^{-\dfrac{1}{2}x'\Sigma^{-1}x}dx=(2\pi)^{k/2}|\Sigma|^{1/2}.$$
			%		So, $$\int e^{-\dfrac{1}{2}x'\Sigma^{-1}x}dx=(2\pi)^{k/2}|\Sigma|^{1/2}$$
			Substituting $\Sigma^{-1}=B$, we get
			\begin{eqnarray}
				\int e^{-\dfrac{1}{2}x'Bx}dx
				%			&=&(2\pi)^{k/2}\dfrac{1}{|\Sigma^{-1}|^{1/2}}\nonumber\\ 
				&=&(2\pi)^{k/2}\dfrac{1}{|B|^{1/2}}.\label{multnorm}
			\end{eqnarray}
			Let $f$ be the density of $\hat{Z}=(Z_1,Z_2)$, Let  \[
			\Sigma = 
			\begin{pmatrix}
				1 & \theta \\
				\theta & 1 
			\end{pmatrix}
			\]
			Then
			\begin{eqnarray*}
				\BE(\phi(Z_1)\phi(Z_2))&=&\int\int \phi(Z_1)\phi(Z_2)f(z)dz\\
				&=&\int\int \big(\dfrac{1}{\sqrt{2\pi}}e^{-\dfrac{{z_1}^2}{2}}\big)\big(\dfrac{1}{\sqrt{2\pi}}e^{-\dfrac{{z_1}^2}{2}}\big)\big(\dfrac{1}{(2\pi)|\Sigma|^{1/2}}e^{-\dfrac{1}{2}z'\Sigma^{-1}z}\big)dz\\
				%			&=&\dfrac{1}{2\pi}\dfrac{1}{2\pi|\Sigma|^{1/2}}\int\int e^{-\dfrac{{z'Iz}}{2}}e^{-\dfrac{{z'\Sigma^{-1}z}}{2}}dz\\
				&=&\dfrac{1}{2\pi|\Sigma|^{1/2}}\bigg(\dfrac{1}{2\pi}\int\int e^{-\dfrac{{z'(I+\Sigma^{-1})z}}{2}}dz\bigg).\\
			\end{eqnarray*}
			By defining $B=I+\Sigma^{-1}$ and by Equation \eqref{multnorm} with $k=2$ we get
			$$\BE(\phi(Z_1)\phi(Z_2))=\dfrac{1}{2\pi|\Sigma|^{1/2}}\dfrac{1}{|B|^{1/2}}$$
			%		further, we have \[
			%		\Sigma = 
			%		\begin{pmatrix}
				%			1 & \theta \\
				%			\theta & 1 
				%		\end{pmatrix}
			%		\] 
			%		$|\Sigma|=(1-\theta^2)$, 
			\[
			B=I+\Sigma^{-1}=\dfrac{1}{(1-\theta^2)}
			\begin{pmatrix}
				(2-\theta^2) & -\theta \\
				-\theta & (2-\theta^2) 
			\end{pmatrix}
			\]
			%		$$|B|=|I+\Sigma^{-1}|=\dfrac{1}{(1-\theta^2)^2}(2-\theta^2)^2-\theta^2$$
			%		$$|B|^{1/2}=\dfrac{\sqrt{(2-\theta^2)^2-\theta^2}}{1-\theta^2}$$
			$$|\Sigma|^{1/2}|B|^{1/2}=\sqrt{1-\theta^2}\dfrac{\sqrt{(2-\theta^2)^2-\theta^2}}{1-\theta^2}=\sqrt{\dfrac{(2-\theta^2)^2-\theta^2}{1-\theta^2}}$$
			%		Therefore, 
			$$\BE(\phi(Z_1)\phi(Z_2))=\dfrac{1}{2\pi}\sqrt{\dfrac{1-\theta^2}{(2-\theta^2)^2-\theta^2}}=\dfrac{1}{2\pi}\sqrt{\dfrac{1}{(4-\theta^2)}}.$$
			%		On Solving the above expression we get
			%		\begin{eqnarray*}
				%			\BE(\phi(Z_1)\phi(Z_2))&=&\dfrac{1}{2\pi}\sqrt{\dfrac{1-\theta^2}{(2-\theta^2)^2-\theta^2}}\\
				%			&=&\dfrac{1}{2\pi}\sqrt{\dfrac{(1+\theta)(1-\theta)}{(2-\theta^2+\theta)(2-\theta^2-\theta)}}\\
				%			&=&\dfrac{1}{2\pi}\sqrt{\dfrac{(1+\theta)(1-\theta)}{(-\theta-1)(\theta-2)(-\theta+1)(\theta+2)}}\\
				%			&=&\dfrac{1}{2\pi}\sqrt{\dfrac{-1}{(\theta-2)(\theta+2)}}\\
				%			&=&\dfrac{1}{2\pi}\sqrt{\dfrac{1}{(4-\theta^2)}}
				%		\end{eqnarray*}
			Proof of part \ref{part16}	To prove the equation in part \ref{part16} let us write the completing square
			%		\subsection*{Completing square}
			%		Completing square,
			\begin{eqnarray}
				(a+bt)^2+t^2&=&t^2(1+b^2)+2abt+a^2\nonumber\\
				&=&(1+b^2)(t+\dfrac{ab}{1+b^2})^2+\dfrac{a^2}{1+b^2}\label{comsq}
			\end{eqnarray}
		Now, we have
			%		Write $Z_1=\theta Z_2+\sqrt{1-\theta^2}X$, where $X$ is independent of $Z_2$ and $X \sim N(0,1)$
			$$\BE(Z_2\Phi(Z_2)\phi(Z_1))=\BE[Z_2\phi(Z_2)\BE\phi(\theta Z_2+\sqrt{1-\theta^2}X)|Z_2]$$
			Now, with $a=\theta Z_2$ and $b=\sqrt{1-\theta^2}$ and by Equation \eqref{comsq},
			\begin{eqnarray}
				\BE[\phi(\theta Z_2+\sqrt{1-\theta^2}X)|Z_2]&=&\BE\phi(a+bX)\nonumber\\&=&\dfrac{1}{\sqrt{2\pi}}\dfrac{1}{\sqrt{2\pi}}\int e^{-\dfrac{1}{2}(a+bt)^2} e^{-\dfrac{t^2}{2}}dt\nonumber\\&=&\dfrac{1}{\sqrt{2\pi}}\dfrac{1}{\sqrt{2\pi}}\int e^{-\dfrac{1}{2}\big((1+b^2)(t+\dfrac{ab}{1+b^2})^2+\dfrac{a^2}{1+b^2}\big)}dt\nonumber\\
				&=&\dfrac{1}{\sqrt{2\pi}}e^{-\dfrac{1}{2}\dfrac{a^2}{1+b^2}}\dfrac{1}{\sqrt{2\pi}}\int e^{-\dfrac{1}{2}(1+b^2)(t+\dfrac{a}{1+b^2})^2}dt
				\nonumber\\
				&=&\dfrac{1}{\sqrt{2\pi}}e^{-\dfrac{1}{2}\dfrac{a^2}{1+b^2}}\dfrac{1}{\sqrt{1+b^2}}\label{eq1part16}
			\end{eqnarray}
			Now, substituting $a=\theta Z_2$, $b=\sqrt{1-\theta^2}$ in \eqref{eq1part16}, we get
			\begin{eqnarray*}
				&&\BE[Z_2\Phi(Z_2)\BE\phi(\theta Z_2+\sqrt{1-\theta^2}X)]=\BE[Z_2\Phi(Z_2)\dfrac{1}{\sqrt{2\pi}}e^{-\dfrac{1}{2}\dfrac{\theta^2 Z_2^2}{2-\theta^2}}\dfrac{1}{\sqrt{2-\theta^2}}]\\&=&
				\dfrac{1}{\sqrt{2\pi}}\dfrac{1}{\sqrt{2-\theta^2}}\BE[Z_2\Phi(Z_2)e^{-\dfrac{1}{2}\dfrac{\theta^2 Z_2^2}{2-\theta^2}}]
				\\&=&\dfrac{1}{\sqrt{2\pi}}\dfrac{1}{\sqrt{2-\theta^2}}\int t \int_{-\infty}^{t} \dfrac{1}{\sqrt2\pi}e^{-\dfrac{x^2}{2}}dx e^{-\dfrac{1}{2}\dfrac{\theta^2 t^2}{2-\theta^2}} \dfrac{1}{\sqrt{2\pi}}e^{-\dfrac{t^2}{2}} dt ,\hspace{0.5cm} x\leq t\\
				%			&=& \dfrac{1}{2\pi}\dfrac{1}{\sqrt{2-\theta^2}}\dfrac{1}{\sqrt2\pi}\int_{-\infty}^{\infty} e^{-\dfrac{x^2}{2}} \int_{x}^{\infty} t e^{-\dfrac{1}{2}t^2\bigg(\dfrac{\theta^2 }{2-\theta^2}+1\bigg)} dtdx\\
				&=& \dfrac{1}{2\pi}\dfrac{1}{\sqrt{2-\theta^2}}\dfrac{1}{\sqrt2\pi}\int_{-\infty}^{\infty} e^{-\dfrac{x^2}{2}} \int_{x}^{\infty} t e^{-\dfrac{1}{2}\big(\dfrac{2 }{2-\theta^2}\big)t^2}  dtdx\\&=& \dfrac{1}{2\pi}\dfrac{1}{\sqrt{2-\theta^2}}\dfrac{1}{\sqrt2\pi}\int_{-\infty}^{\infty} e^{-\dfrac{x^2}{2}} \dfrac{-(2-\theta^2) }{2}\dfrac{d}{dt}e^{-\dfrac{1}{2}\big(\dfrac{2 }{2-\theta^2}\big)t^2}  dtdx\\
				&=&\dfrac{1}{2\pi}\dfrac{1}{\sqrt{2-\theta^2}}\dfrac{1}{\sqrt2\pi}\int_{-\infty}^{\infty}\dfrac{(2-\theta^2) }{2} e^{-\dfrac{x^2}{2}}\bigg(-e^{-\dfrac{1}{2}\big(\dfrac{2 }{2-\theta^2}\big)t^2}\bigg) \Bigg|_{x}^{\infty} dx\\
				&=&\dfrac{1}{2\pi}\dfrac{1}{\sqrt{2-\theta^2}}\dfrac{(2-\theta^2) }{2}\int_{-\infty}^{\infty}\dfrac{1}{\sqrt2\pi}e^{-\dfrac{x^2}{2}\bigg(1+\dfrac{2}{2-\theta^2}\bigg)}dx\\
				&=&\dfrac{1}{4\pi}\dfrac{2-\theta^2}{\sqrt{2-\theta^2}}\sqrt{\dfrac{2-\theta^2}{4-\theta^2}}=\dfrac{2-\theta^2}{4\pi\sqrt{4-\theta^2}}.
			\end{eqnarray*}
			%		We have
			%		$$\dfrac{d}{dt}e^{-\dfrac{1}{2}\big(\dfrac{2}{2-\theta^2}\big)t^2}=e^{-\dfrac{1}{2}\big(\dfrac{2}{2-\theta^2}\big)t^2}\big(\dfrac{-(2) }{2-\theta^2}\big)t$$
			%		$$ t e^{-\dfrac{1}{2}\big(\dfrac{2 }{2-\theta^2}\big)t^2}=\dfrac{-(2-\theta^2) }{2}\dfrac{d}{dt}e^{-\dfrac{1}{2}\big(\dfrac{2 }{2-\theta^2}\big)t^2}$$
			%		Therefore
			%		\begin{eqnarray*}
				%			\BE[Z_2\Phi(Z_2)\BE\phi(\theta Z_2+\sqrt{1-\theta^2}X)]&=& \dfrac{1}{2\pi}\dfrac{1}{\sqrt{2-\theta^2}}\dfrac{1}{\sqrt2\pi}\int_{-\infty}^{\infty} e^{-\dfrac{x^2}{2}} \dfrac{-(2-\theta^2) }{2}\dfrac{d}{dt}e^{-\dfrac{1}{2}\big(\dfrac{2 }{2-\theta^2}\big)t^2}  dtdx\\
				%			&=&\dfrac{1}{2\pi}\dfrac{1}{\sqrt{2-\theta^2}}\dfrac{1}{\sqrt2\pi}\int_{-\infty}^{\infty}\dfrac{(2-\theta^2) }{2} e^{-\dfrac{x^2}{2}}\bigg(-e^{-\dfrac{1}{2}\big(\dfrac{2 }{2-\theta^2}\big)t^2}\bigg) \Bigg|_{x}^{\infty} dx\\
				%			&=&\dfrac{1}{2\pi}\dfrac{1}{\sqrt{2-\theta^2}}\dfrac{(2-\theta^2) }{2}\int_{-\infty}^{\infty}\dfrac{1}{\sqrt2\pi}e^{-\dfrac{x^2}{2}\bigg(1+\dfrac{2}{2-\theta^2}\bigg)}dx\\
				%			&=&\dfrac{1}{4\pi}\dfrac{2-\theta^2}{\sqrt{2-\theta^2}}\sqrt{\dfrac{2-\theta^2}{4-\theta^2}}\\
				%			&=&\dfrac{2-\theta^2}{4\pi\sqrt{4-\theta^2}}
				%		\end{eqnarray*}
			Proof of part \ref{part17} can be obtained by conditioning on $Z_2$ given as follows
			\begin{eqnarray}
				\BE(Z_1Z_2\Phi(Z_1)\Phi(Z_2))&=&\BE(Z_1Z_2\Phi(Z_1)\Phi(Z_2)|Z_2)\nonumber\\
				&=&\BE[Z_2\Phi(Z_2)\BE(\theta Z_2+\sqrt{1-\theta^2}X)\Phi(Z_1)|Z_2]\nonumber\\
				&=&\BE[\theta Z_2^2\Phi(Z_2)\BE[\Phi(Z_1)]|Z_2]+\BE[Z_2\Phi(Z_2)\BE[\sqrt{1-\theta^2}X\Phi(Z_1)]|Z_2]\nonumber\\
				&=&T_1+T_2,\hspace{0.5cm} \text{say.}\label{eqpart171}
			\end{eqnarray}
			%		where,
			%		\begin{eqnarray*}
				%			T1&=&\BE[\theta Z_2^2 \Phi(Z_2)\BE[\Phi(Z_1)]|Z_2]\\
				%			T2&=&\BE[Z_2\Phi(Z_2)\BE[\sqrt{1-\theta^2}X\Phi(Z_1)]|Z_2]	
				%		\end{eqnarray*}
			%			\subsubsection*{Calculation required in $T_1$}
			%		Suppose $Z\sim N(0,1)$
			%		Let us calculate $\BE[\Phi(a+bZ)]$\\
		To calculate $T_1$, let $N$ be a standard normal random variable which is independent of $Z$ then $$\Phi(a+cZ)=P(N\leq a+cZ|Z)$$
			%		So
			\begin{eqnarray*}
				\BE[\Phi(a+cZ)]&=&\BE[P(N\leq a+cZ)|Z]\\
				&=&\BE[\BE[I(N\leq a+cZ)]|Z]\\
				&=&\BE[I(N\leq a+cZ)]\\
				&=&P(N\leq a+cZ)\\
				&=&P(N-cZ \leq a)	
			\end{eqnarray*}
			Note that  $$N-bz\sim N(0,1+c^2) = P\bigg(\dfrac{N-cZ}{\sqrt{1+c^2}}\leq \dfrac{a}{\sqrt{1+c^2}}\bigg)$$\\
			\begin{equation}
				\BE[\Phi(a+cZ)]=\Phi \bigg(\dfrac{a}{\sqrt{1+c^2}}\bigg).\label{T1req}
			\end{equation}
%			\subsubsection*{Calculation of $T_1$}
			We have $T_1=\BE[\theta Z_2^2 \Phi(Z_2) \BE[\Phi(Z_1)|Z_2]]$ and let $Z_1=\theta Z_2 + \sqrt{1-\theta^2}N$
			\begin{eqnarray*}
				\BE[\Phi(Z_1)|Z_2] &=& \BE[\Phi(\theta Z_2+\sqrt{1-\theta^2}N)|Z_2]\\
				&=& \Phi\bigg(\dfrac{\theta Z_2}{\sqrt{2-\theta^2}}\bigg)
			\end{eqnarray*}
			by previous result in \eqref{T1req}, substituting $ a=\theta Z_2 , c=\sqrt{1-\theta ^2}$\\
			So, $$ T_1 = \theta \BE \bigg[Z_2^2 \Phi(Z_2)\Phi \bigg(\dfrac{\theta Z_2}{\sqrt{2-\theta^2}}\bigg)\bigg]$$\\
			%		Let $b=\dfrac{\theta}{\sqrt{2-\theta^2}}$\\
			\begin{eqnarray*}
				\BE [Z^2 \Phi(Z)\Phi (bZ)]&=&\int z^2 \Phi(z) \Phi(bz) \phi(z) dz\\
				&=&\int_{-\infty}^{\infty} [z \Phi(z) \Phi(bz)] [z \phi(z)] dz\\
				%			&=& z \Phi(z) \Phi(bz) (-\phi(z)) \bigg|_{-\infty}^{\infty}\\&& - \int [\Phi(z) \Phi(bz)+ z \phi(z) \Phi(bz) + z \Phi(z)b\phi(bz)](-\phi(z))dz\\
				%	&=& \int_{-\infty}^{\infty} \phi(z) \big[\Phi(z)\Phi(bz) + z\phi(z)\Phi(bz) + bz\Phi(z)\phi(bz)\big]dz\\
				&=& \int_{-\infty}^{\infty}\{ \phi(z)\Phi(z)\Phi(bz) + z\phi^2(z)\Phi(bz) +\phi(z)bz\Phi(z)\phi(bz)\}dz\\
				&=& T_{11}+T_{12}+T_{13} \hspace{0.5cm}\text{(say).}
			\end{eqnarray*}
			%		\subsubsection*{Calculation of $T_{11}$}
			$$	T_{11}=\int \phi(z) \Phi(z) \Phi(bz)dz$$
			Suppose $N_1,N_2$ are independent N(0,1) and also independent of Z.\\
			Then
			\begin{eqnarray*}
				T_{11}&=&\BE[ P(N_1 \leq ZN_2\leq bZ) |Z]\\
				&=&\BE[I(N_1 \leq Z, N_2 \leq bZ)|Z]\\
				%			&=& P(N_1\leq Z, N_2 \leq bZ)\\
				&=& P(N_1 -Z \leq 0, N_2 - bZ \leq o).
			\end{eqnarray*}
			Let $$A_1 = N_1-Z, A_2 = N_2-bZ$$
			Then
			$$	A_1 \sim N(0,2), \hspace{0.5cm}A_2 \sim N(0,1+b^2),\hspace{0.5cm} \hspace{0.5cm} Cov(A_1,A_2) = b.$$
			%		A_2 \sim N(0,1+b^2)\\$$
			%		$$Cov(A_1,A_2) = b$$
			\begin{eqnarray*}
				T_{11}	&=& P\bigg(\dfrac{N_1-Z}{\sqrt{2}}\leq 0, \dfrac{N_2-bZ}{\sqrt{1+b^2}} \leq 0\bigg)=P(B_1 \leq 0, B_2 \leq 0).
			\end{eqnarray*}
			where $(B_1,B_2)\sim N(0,0,1,1,d)$, since  $b=\dfrac{\theta}{\sqrt{2-\theta^2}}.$ $$d=Cov(B_1,B_2)=\dfrac{1}{\sqrt{2}}\dfrac{b}{\sqrt{1+b^2}}=\dfrac{\theta}{2}.$$
			Switching to polar coordinates and solving we get
			\begin{eqnarray}
				T_{11}=P(B_1<0,B_2<0)&=&\dfrac{1}{4}+\dfrac{1}{2\pi}\sin^{-1}(d)\nonumber\\
				&=&\dfrac{1}{4}+\dfrac{1}{2\pi}\sin^{-1}\bigg(\dfrac{\theta}{2}\bigg)\label{eqpart172}
			\end{eqnarray}
			We have,
			\begin{eqnarray}
				T_{12} &=& \dfrac{1}{2\pi} \int z e^{-z^2} \Phi(bz) dz\nonumber\\
				&=& \dfrac{1}{2\pi}\int \Phi(bz) (ze^{-z^2})dz\nonumber\\&=&\dfrac{1}{2\pi}\Phi(bz)\bigg(\dfrac{-e^{-z^2}}{2}\bigg)\bigg|_{-\infty}^{\infty}-\dfrac{1}{2\pi}\int b\phi(bz)\bigg(\dfrac{-e^{-z^2}}{2}\bigg)dz\nonumber\\
				&=& \dfrac{1}{2\pi}\dfrac{b}{2}\int \phi(bz) e^{-z^2}dz\nonumber\\
				%			&=& \dfrac{1}{2\pi}\dfrac{b}{2}\int \dfrac{1}{\sqrt{2\pi}}e^{-\dfrac{b^2z^2}{2}}e^{-z^2}dz\\
				&=& \dfrac{1}{2\pi}\dfrac{b}{2}\int \dfrac{1}{\sqrt{2\pi}}e^{-\dfrac{z^2}{2}(2+b^2)}dz\nonumber\\
				&=& \dfrac{1}{2\pi}\dfrac{b}{2}\dfrac{1}{\sqrt{2+b^2}}
				%			&=&\dfrac{1}{2\pi}\dfrac{\theta}{\sqrt{2-\theta^2}}\dfrac{1}{2}\dfrac{\sqrt{2-\theta^2}}{\sqrt{4-\theta^2}}
				=\dfrac{1}{4\pi}\dfrac{\theta}{\sqrt{4-\theta^2}}.\label{eqpart173}
			\end{eqnarray}
			%		$$\BE(Z_1Z_2\Phi(Z_1)\Phi(Z_2))=\theta\bigg[\dfrac{1}{4}+\dfrac{2}{\pi}arctan\big(\dfrac{\theta}{\sqrt{1-2\theta^2}}\big)+\dfrac{\theta}{4\pi\sqrt{(4-\theta^2)}}+\dfrac{(2-\theta^2)}{4\pi\sqrt{4-\theta^2}}\bigg]+\dfrac{(2-\theta^2)(1-\theta^2)}{4\pi(2+\theta^2)}$$ 
			%	\subsubsection*{Calculation of $T_{13}$}
			%	Similarly solving $T_{13}$ we get, $$T_{13}=\int bz\phi(z)\Phi(z)\phi(bz)dz$$
			Further,
			\begin{eqnarray}
				T_{13}&=& \int bz\phi(z)\Phi(z)\phi(bz)dz\nonumber\\&=&b\dfrac{1}{2\pi}\int z e^{-z^2/2}\Phi(z)e^{-b^2z^2/2}dz\nonumber\\
				&=&b\dfrac{1}{2\pi}\int z e^{-\dfrac{z^2(1+b^2)}{2}}\Phi(z) dz\nonumber\\
				&=&\dfrac{1}{2\pi}\dfrac{b}{(1+b^2)}\dfrac{1}{\sqrt{2\pi}}\int e^{-\dfrac{z^2(2+b^2)}{2}}dz\nonumber\\
				&=&\dfrac{1}{2\pi}\dfrac{b}{(1+b^2)}\dfrac{1}{\sqrt{2+b^2}} =\dfrac{\theta}{4\pi}\dfrac{2-\theta^2}{\sqrt{4-\theta^2}}.\label{eqpart174}
			\end{eqnarray}
			%	substituting $b=\dfrac{\theta}{\sqrt{2-\theta^2}}$ we get, 
			%	$$T_{13}=\dfrac{\theta}{4\pi}\dfrac{2-\theta^2}{\sqrt{4-\theta^2}}$$
%			\subsubsection*{Calculation of $T_{2}$}
Therefore by Equations \eqref{eqpart172}, \eqref{eqpart173} and \eqref{eqpart174} we write,
\begin{eqnarray}
	T_1=T_{11}+T_{12}+T_{13}&=&\dfrac{1}{4}+\dfrac{1}{2\pi}\sin^{-1}\bigg(\dfrac{\theta}{2}\bigg)+\dfrac{\theta}{4\pi\sqrt{4-\theta^2}}+\dfrac{\theta(2-\theta^2)}{4\pi\sqrt{4-\theta^2}}\label{eqpart176}
\end{eqnarray}
Further we have
\begin{eqnarray*}
	T_2&=&\BE[Z_2\Phi(Z_2)\BE[\sqrt{1-\theta^2}X\Phi(Z_1)]|Z_2]\\
	&=&\BE[Z_2\Phi(Z_2)\sqrt{1-\theta^2}\BE[X\Phi(\theta Z_2+\sqrt{1-\theta^2}X)]|Z_2].
\end{eqnarray*}
Suppose $Z\sim N(0,1)$. Let us calculate $\BE[Z\Phi(a+cZ)]$
\begin{eqnarray}
	\BE[Z\Phi(a+cZ)]&=&\int z\Phi(a+cZ)\phi(z)dz\nonumber\\
	&=&\int \Phi(a+cz) (z\phi(z))dz\nonumber\\
	&=& \Phi(a+cz) (-\phi(z))\bigg|_{-\infty}^{\infty}-\int c \phi(a+cz) (-\phi(z))dz\nonumber\\
	&=& c\int \phi(z)\phi(a+cz)dz\nonumber\\
	&=&\dfrac{c}{\sqrt{2\pi}}\dfrac{1}{\sqrt{2\pi}}\int e^{-\dfrac{1}{2}(z^2+(a+cz)^2)}dz\nonumber\\
	%		&=&\dfrac{b}{\sqrt{2\pi}}\dfrac{1}{\sqrt{2\pi}}\int e^{-\dfrac{1}{2}\big((1+b^2)(z+\dfrac{ab}{1+b^2})^2+\dfrac{a^2}{1+b^2}\big)}dz\\
	%		&=&\dfrac{b}{\sqrt{2\pi}}e^{-\dfrac{1}{2}\dfrac{a^2}{1+b^2}}\dfrac{1}{\sqrt{2\pi}}\int e^{-\dfrac{1}{2}\big((1+b^2)(z+\dfrac{ab}{1+b^2})^2\big)}dz\\
	&=&\dfrac{c}{\sqrt{2\pi}}e^{-\dfrac{1}{2}\dfrac{a^2}{1+c^2}}\dfrac{1}{\sqrt{1+c^2}}.\label{eqpart175}
\end{eqnarray}
Now by substitution $a=\theta Z_2$ and $c=\sqrt{1-\theta^2}$ in \eqref{eqpart175} we write
\begin{eqnarray}
	T_2&=&\BE[Z_2\Phi(Z_2)\sqrt{1-\theta^2}\BE[X\Phi(\theta Z_2+\sqrt{1-\theta^2}X)]|Z_2]\nonumber\\
	&=&\BE[Z_2\Phi(Z_2)\sqrt{1-\theta^2}\bigg(\dfrac{\sqrt{1-\theta^2}}{\sqrt{2\pi}}e^{-\dfrac{1}{2}\dfrac{\theta^2Z_2^2}{1+(1-\theta^2)}}\dfrac{1}{\sqrt{1+(1-\theta^2)}}\bigg)]\nonumber\\
	&=&\dfrac{(1-\theta^2)}{\sqrt{2\pi}\sqrt{2-\theta^2}}\BE[Z_2\Phi(z_2)e^{-\dfrac{1}{2}\dfrac{\theta^2Z_2^2}{(2-\theta^2)}}]\nonumber\\
	%		&=&\dfrac{1}{\sqrt{2\pi}}\dfrac{(1-\theta^2)}{\sqrt{2-\theta^2}}\int t \int_{-\infty}^{t} \dfrac{1}{\sqrt2\pi}e^{-\dfrac{x^2}{2}}dx e^{-\dfrac{1}{2}\dfrac{\theta^2 t^2}{2-\theta^2}} \dfrac{1}{\sqrt{2\pi}}e^{-\dfrac{t^2}{2}} dt ,\hspace{0.5cm} x\leq t\\&=& \dfrac{1}{2\pi}\dfrac{(1-\theta^2)}{\sqrt{2-\theta^2}}\dfrac{1}{\sqrt2\pi}\int_{-\infty}^{\infty} e^{-\dfrac{x^2}{2}} \int_{x}^{\infty} t e^{-\dfrac{1}{2}t^2\bigg(\dfrac{\theta^2 }{2-\theta^2}+1\bigg)} dtdx\\
	&=& \dfrac{1}{2\pi}\dfrac{(1-\theta^2)}{\sqrt{2-\theta^2}}\dfrac{1}{\sqrt2\pi}\int_{-\infty}^{\infty} e^{-\dfrac{x^2}{2}} \int_{x}^{\infty} t e^{-\dfrac{1}{2}\big(\dfrac{2 }{2-\theta^2}\big)t^2}  dtdx\nonumber\\
	&=&\dfrac{1}{2\pi}\dfrac{(1-\theta^2)}{\sqrt{2-\theta^2}}\dfrac{(2-\theta^2) }{2}\int_{-\infty}^{\infty}\dfrac{1}{\sqrt2\pi}e^{-\dfrac{x^2}{2}\bigg(1+\dfrac{2}{2-\theta^2}\bigg)}dx\nonumber\\
	&=&\dfrac{(1-\theta^2)\sqrt{2-\theta^2}}{4\pi}\sqrt{\dfrac{2-\theta^2}{4-\theta^2}}=\dfrac{(1-\theta^2)(2-\theta^2)}{4\pi\sqrt{4-\theta^2}}\label{eqpart177}
\end{eqnarray}
%	We have
%	$$\dfrac{d}{dt}e^{-\dfrac{1}{2}\big(\dfrac{2}{2-\theta^2}\big)t^2}=e^{-\dfrac{1}{2}\big(\dfrac{2}{2-\theta^2}\big)t^2}\big(\dfrac{-(2) }{2-\theta^2}\big)t$$
%	$$ t e^{-\dfrac{1}{2}\big(\dfrac{2 }{2-\theta^2}\big)t^2}=\dfrac{-(2-\theta^2) }{2}\dfrac{d}{dt}e^{-\dfrac{1}{2}\big(\dfrac{2 }{2-\theta^2}\big)t^2}$$
%	Therefore
%	\begin{eqnarray*}
%		T_2
%		%	&=& \dfrac{1}{2\pi}\dfrac{(1-\theta^2)}{\sqrt{2-\theta^2}}\dfrac{1}{\sqrt2\pi}\int_{-\infty}^{\infty} e^{-\dfrac{x^2}{2}} \int_{x}^{\infty} t e^{-\dfrac{1}{2}\big(\dfrac{2}{2-\theta^2}\big)t^2}  dtdx\\
%		&=&\dfrac{1}{2\pi}\dfrac{(1-\theta^2)}{\sqrt{2-\theta^2}}\dfrac{1}{\sqrt2\pi}\int_{-\infty}^{\infty} e^{-\dfrac{x^2}{2}}\dfrac{(2-\theta^2) }{2}\bigg(-e^{-\dfrac{1}{2}\big(\dfrac{2 }{2-\theta^2}\big)t^2}\bigg) \Bigg|_{x}^{\infty} dx\\
%		&=&\dfrac{1}{2\pi}\dfrac{(1-\theta^2)}{\sqrt{2-\theta^2}}\dfrac{(2-\theta^2) }{2}\int_{-\infty}^{\infty}\dfrac{1}{\sqrt2\pi}e^{-\dfrac{x^2}{2}\bigg(1+\dfrac{2}{2-\theta^2}\bigg)}dx\\
%		&=&\dfrac{(1-\theta^2)\sqrt{2-\theta^2}}{4\pi}\sqrt{\dfrac{2-\theta^2}{4-\theta^2}}\\
%		&=&\dfrac{(1-\theta^2)(2-\theta^2)}{4\pi\sqrt{4-\theta^2}}
%	\end{eqnarray*}
Now, by Equations \eqref{eqpart176} and \eqref{eqpart177} we obtain
\begin{eqnarray*}
	\BE(Z_1Z_2\Phi(Z_1)\Phi(Z_2))&=&T_1+T_2\\
	&=&\dfrac{1}{4}+\dfrac{1}{2\pi}\sin^{-1}\bigg(\dfrac{\theta}{2}\bigg)+\dfrac{\theta}{4\pi\sqrt{4-\theta^2}}+\dfrac{\theta(2-\theta^2)}{4\pi\sqrt{4-\theta^2}}+\dfrac{(1-\theta^2)(2-\theta^2)}{4\pi\sqrt{4-\theta^2}}.
\end{eqnarray*} 
This completes the proof of Lemma \ref{bvncalc}. 
\end{proof}

\noindent Using the above lemma, we have the following expressions for $\mu_{1}$, $\mu_{2}$, $\mu_{12}$, $\mu_{3}$ and $\kappa$.
\begin{lemma}	\label{bvncalc1}
    Suppose that $(X,Y)\sim N(m_1,m_2,\sigma_1^2,\sigma_2^2,\theta)$ with $F_1=N(m_1,\sigma_1^2)$ and $F_2=N(m_2,\sigma_2^2)$, and let $(X_1,Y_1)$, $(X_2,Y_2)$ be the i.i.d copies of $(X,Y)$. Then 
	\begin{enumerate}[label=(\roman*)]
		\item  $\mu_1 = \BE [|X_1-X_2|]
			%=\BE_{F_1}(g_F(X_1))
			=\dfrac{2\sigma_1}{\sqrt{\pi}}$, \label{part21}\\
		\item 	$\mu_2= \BE [|Y_1-Y_2|]
			=\dfrac{2\sigma_2}{\sqrt{\pi}},$\label{part22}\\
        \item 		$\mu_3=\BE_{F_{12}} [|X_1-X_2||Y_1-Y_3|\big]=\sigma_1\sigma_2\bigg[\dfrac{2}{\pi}\sqrt{\dfrac{1}{(4-\theta^2)}}+\dfrac{2\theta}{\pi} \sin^{-1}\bigg(\dfrac{\theta}{2}\bigg)+\dfrac{2(3-\theta^2)}{\pi\sqrt{4-\theta^2}} \bigg]$,\label{part23}\\
        \item 	$\mu_{12}=\BE_{F_{12}} \big[|X_1-X_2||Y_1-Y_2|\big]=\dfrac{4\sigma_1\sigma_2}{\pi}(\theta \sin^{-1}\theta+\sqrt{1-\theta^2}).$\label{part24}\\
        \item 	$\kappa (\theta)=\dfrac{1}{4}\Big[\mu_{12}+ \mu_1\mu_2-2\mu_{3}\Big]$.\label{part25}\\
%       \begin{eqnarray*}
%			&=&\dfrac{1}{4}\bigg[\dfrac{4\sigma_1\sigma_2}{\pi}(\theta \sin^{-1}\theta+     \sqrt{1-\theta^2})+\dfrac{2\sigma_1}{\sqrt{\pi}}\dfrac{2\sigma_2}{\sqrt{\pi}}-2\sigma_1\sigma_2\bigg[\dfrac{2}{\pi}\sqrt{\dfrac{1}{(4-\theta^2)}}+\dfrac{2\theta}{\pi} \sin^{-1}\bigg(\dfrac{\theta}{2}\bigg)+\dfrac{2(3-\theta^2)}{\pi\sqrt{4-\theta^2}} \bigg]\bigg]
%		\end{eqnarray*}
		\item     \begin{equation}
                    \kappa (\theta) = \dfrac{\sigma_1\sigma_2}{\pi}\bigg[\theta \sin^{-1}\theta +\sqrt{1-\theta^2} + 1 - \theta \sin^{-1}\bigg(\dfrac{\theta}{2} \bigg) - \sqrt{4-\theta^2} \bigg]. \label{eq:bvnkappatheta}
                \end{equation}
	\end{enumerate}
\end{lemma}
\begin{proof} To prove parts \ref{part21} and \ref{part22}, let us
		suppose $X$ is normally distributed with mean $\mu$ and variance $\sigma^2$.  
		%		$$g_F(x)=\BE_{F}[|x-X|]=\sigma\BE_{F}\big[|\dfrac{x-\mu}{\sigma}-Z|\big]$$
		$$g_F(x)=\BE_{F}[|x-X|]$$
		Note that $$\dfrac{x-X}{\sigma}\sim N(\dfrac{x-\mu}{\sigma},1)$$
		%		where $Z\sim N(0,1)$.
		\begin{eqnarray*}
			g_F(x)&=&\sigma\BE_{F}[|\dfrac{x-\mu}{\sigma}-\dfrac{X-\mu}{\sigma}|]\\&=&\sigma\BE_{F}[|\dfrac{x-\mu}{\sigma}-Z|]
		\end{eqnarray*}
		where $Z\sim N(0,1).$
		So, it is enough to calculate $\BE[|z-Z|]$ for $z\in \mathbb{R}$
		\begin{eqnarray*}
			\BE[|z-Z|]&=&\int_{-\infty}^{\infty}|z-t|\phi(t)dt\\
			&=&\int_{-\infty}^{z}(z-t)\phi(t)dt+\int_{z}^{\infty}(t-z)\phi(t)dt\\
			%			&=&z\int_{-\infty}^{z}\phi(t)dt-\int_{-\infty}^{z}t\phi(t)dt+\int_{z}^{\infty}t\phi(t)dt-z\int_{z}^{\infty}\phi(t)dt\\
			&=&z\Phi(z)-\int_{-\infty}^{\infty}t\phi(t)dt+\int_{z}^{\infty}t\phi(t)dt+\int_{z}^{\infty}t\phi(t)dt-z[1-\Phi(z)]\\
			&=&-z+2z\Phi(z)+2\int_{z}^{\infty}t\phi(t)dt\\
			&=&-z+2z\Phi(z)+2\phi(z).
		\end{eqnarray*}
		and hence
		$$g_{F}(x)=\sigma\bigg[-\dfrac{x-\mu}{\sigma}+2\big(\dfrac{x-\mu}{\sigma}\big)\Phi\big(\dfrac{x-\mu}{\sigma}\big)+2\phi\big(\dfrac{x-\mu}{\sigma}\big)\bigg].$$ where $\Phi$ and $\phi$ are distribution function and density functions of standard normal respectively.
		Further let ${Z}\sim N(0,1)$ and 
		\begin{eqnarray*}
			h&=&\BE[-{Z}+2{Z}\Phi({Z})+2\phi({Z})]\\
			&=&2\BE[{Z}\Phi({Z})+\phi({Z})].
		\end{eqnarray*}
		\begin{eqnarray*}
			\BE[{Z}\Phi({Z})]&=&\int_{-\infty}^{\infty}t\Phi(t)\phi(t)dt\\
			&=&\int_{-\infty}^{\infty}t\phi(t)\big[\int_{-\infty}^{t}\phi(x)dx\big]dt=\int_{-\infty}^{\infty}\phi(x)[\int_{x}^{\infty}t\phi(t)dt]dx,\hspace{0.5cm} x\leq t \iff t\geq x\\
			&=&\int_{-\infty}^{\infty}\phi(x)[-\phi(t)]\Big|_{x}^{\infty}dx\\&=&\int_{-\infty}^{\infty}\phi^2(x)dx=\BE\phi(Z).
		\end{eqnarray*}	
		Therefore 
		\begin{eqnarray*}
			h&=&2\BE[Z\Phi(Z)+\phi(Z)]\\
			&=&2[\BE[Z\Phi(Z)+\BE[\phi(Z)]]\\
			&=&2[[\BE[\phi(Z)]+\BE[\phi(Z)]]=4\BE[\phi(Z)],
		\end{eqnarray*}
		
		and since $$\BE\phi(Z)=\dfrac{1}{2\sqrt{\pi}}\hspace{1cm}, h=4\dfrac{1}{2\sqrt{\pi}}=\dfrac{2}{\sqrt{\pi}}.$$
		Therefore,
		$$g(F)=\BE_{F}(g_F(X))=\BE_{F}[\sigma\BE[-Z+2Z\Phi(Z)+2\phi(Z)]]=\sigma h=\dfrac{2\sigma}{\sqrt{\pi}}.$$
		Similarly we get with $X_1\sim N(m_1,\sigma_1^2)$ and $X_2\sim N(m_2,\sigma_2^2).$

			$$\mu_1=\BE_{F_1}(g_F(X_1))=\dfrac{2\sigma_1}{\sqrt{\pi}}.$$
			$$\mu_2=\BE_{F_2}(g_F(Y_1))=\dfrac{2\sigma_2}{\sqrt{\pi}}.$$
			Proof of part \ref{part23}
			Now suppose that $(X,Y)\sim N(m_1,m_2,\sigma_1^2,\sigma_2^2,\theta)$ with $F_1=N(m_1,\sigma_1^2)$ and $F_2=N(m_2,\sigma_2^2)$, and let $(X_1,Y_1)$, $(X_2,Y_2)$ be the i.i.d copies of $(X,Y)$ and $(Z_1,Z_2)\sim N(0,0,1,1,\theta)$. Using the results of Lemma \ref{bvncalc} we obtain,
			\begin{eqnarray*}
				\mu_3&=&\BE_{F_{12}} [|X_1-X_2||Y_1-Y_3|\big]\\
				&=&\BE_{F_{12}} [\BE[|X_1-X_2||Y_1-Y_3|\mid (X_1,Y_1)]\big]\\
				&=&\BE_{F_{12}} [\BE[|X_1-X_2|\mid X_1]\BE[|Y_1-Y_3|\mid Y_1]\big]\\
				&=&\BE\big[g_{F_{1}}(X)g_{F_{2}}(Y)\big]
				\\&=&\sigma_1\sigma_2\BE_{F_{12}}\big[(-Z_1+2Z_1\Phi(Z_1)+2\phi(Z_1))(-Z_2+2Z_2\Phi(Z_2)+2\phi(Z_2))\big]\\
				&=&\sigma_1\sigma_2\BE[Z_1Z_2-2Z_1Z_2\Phi(Z_2)-2Z_1\phi(Z_2)-2Z_1Z_2\Phi(Z_1)+4Z_1Z_2\Phi(Z_1)\Phi(Z_2)\\ && \ +4Z_1\Phi(Z_1)\phi(Z_2)-2Z_2\phi(Z_1)+4Z_2\Phi(Z_2)\phi(Z_1)+4\phi(Z_1)\phi(Z_2)]\\
				&=&\sigma_1\sigma_2\big[\BE(Z_1Z_2)-4\BE(Z_1Z_2\Phi(Z_1))-4\BE(Z_1\phi(Z_2))+8\BE(Z_1\Phi(Z_1)\phi(Z_2))\\ && \ +4\BE(\phi(Z_1)\phi(Z_2))+4\BE(Z_1Z_2\Phi(Z_1)\Phi(Z_2))\big]\\
				&=&\sigma_1\sigma_2\bigg[\theta-2\theta+0+8\dfrac{2-\theta^2}{4\pi\sqrt{4-\theta^2}}+\dfrac{2}{\pi}\sqrt{\dfrac{1}{(4-\theta^2)}}\nonumber \\ && \
				+4\theta\bigg[\dfrac{1}{4}+\dfrac{1}{2\pi}\sin^{-1}\bigg(\dfrac{\theta}{2}\bigg)+\dfrac{\theta}{4\pi\sqrt{4-\theta^2}}+\dfrac{\theta(2-\theta^2)}{4\pi\sqrt{4-\theta^2}}\bigg]+4\dfrac{(1-\theta^2)(2-\theta^2)}{4\pi\sqrt{4-\theta^2}}\bigg]\nonumber \\
				&=&\sigma_1\sigma_2\bigg[\dfrac{2}{\pi}\sqrt{\dfrac{1}{(4-\theta^2)}}+\dfrac{2\theta}{\pi} \sin^{-1}\bigg(\dfrac{\theta}{2}\bigg)+\dfrac{2(3-\theta^2)}{\pi\sqrt{4-\theta^2}}\\
				&=&\dfrac{2\sigma_1\sigma_2}{\pi}\bigg[\theta \sin^{-1}\bigg(\dfrac{\theta}{2}\bigg)+\dfrac{1+(3-\theta^2)}{\sqrt{4-\theta^2}}
				\bigg]=\dfrac{2\sigma_1\sigma_2}{\pi}\left[\theta \sin^{-1}\bigg(\dfrac{\theta}{2}\bigg)+\sqrt{4-\theta^2}
				\right]\\
			\end{eqnarray*}
					Proof of part \ref{part24}
			 We have,
			\begin{eqnarray*} 
				\mu_{12}=	g({F_{12}}) &=& \BE_{F_{12}} \big[ g_{F_{12}} (X_2, Y_2)\big]\\
				&=& \BE_{F_{12}} \Big[\BE_{F_{12}} \big[|x-X_1||y-Y_1|\big]\Big]\\
				&=&\BE_{F_{12}} \big[|X_1-X_2||Y_1-Y_2|\big].
			\end{eqnarray*}
			Now, let $Z_1=X_1-X_2$ and $Z_2=Y_1-Y_2.$ We have $\BE(Z_1)=\BE(Z_2)=0$ and $Var(Z_1)=2\sigma_1^2$ and $Var(Z_2)=2\sigma_2^2$ and $Cov(Z_1,Z_2)=2\theta\sigma_1\sigma_2.$
			We get $\mu_{12}=\BE\big[|Z_1||Z_2|\big].$ 
			The cumulative distribution function of $(|Z_1|,|Z_2|)$ is given by
			$$F_{|Z_1|,|Z_2|}(z_1,z_2)=F_{Z_1,Z_2}(z_1,z_2)-F_{Z_1,Z_2}(z_1,-z_2)-F_{Z_1,Z_2}(-z_1,z_2)+F_{Z_1,Z_2}(-z_1,-z_2)$$
			The probability density function of $(|Z_1|,|Z_2|)$ is obtained by the mixed partial derivatives of the above.
			$$f_{|Z_1|,|Z_2|}(z_1,z_2)=f_{Z_1,Z_2}(z_1,z_2)+f_{Z_1,Z_2}(z_1,-z_2)+f_{Z_1,Z_2}(-z_1,z_2)+f_{Z_1,Z_2}(-z_1,-z_2)$$ 
			\begin{eqnarray*}
				f_{|Z_1|,|Z_2|}(z_1,z_2)&=&2\frac{1}{2\pi(2\sigma_1\sigma_2)\sqrt{1-\theta^2}}\exp\biggl\{-\dfrac{1}{2(1-\theta^2)}\bigg(\dfrac{z_1^2}{2\sigma_1^2}+\dfrac{z_2^2}{2\sigma_2^2}-\frac{2\theta z_1z_2}{2\sigma_1\sigma_2}\bigg)\biggr\}\\ && \ + 2\frac{1}{2\pi(2\sigma_1\sigma_2)\sqrt{1-\theta^2}}\exp\biggl\{-\dfrac{1}{2(1-\theta^2)}\bigg(\dfrac{z_1^2}{2\sigma_1^2}+\dfrac{z_2^2}{2\sigma_2^2}+\frac{2\theta z_1z_2}{2\sigma_1\sigma_2}\bigg)\biggr\}\\
				%	&=&\frac{1}{2\pi(\sigma_1\sigma_2)\sqrt{1-\theta^2}}\exp\biggl\{-\dfrac{1}{2(1-\theta^2)}\bigg(\dfrac{z_1^2}{2\sigma_1^2}+\dfrac{z_2^2}{2\sigma_2^2}\bigg)\biggr\} \\ && \times
				%	\bigg[\exp\big(\frac{\theta z_1z_2}{(1-\theta^2)2\sigma_1\sigma_2}\big)+\exp\big(\frac{-\theta z_1z_2}{(1-\theta^2)2\sigma_1\sigma_2}\big)\bigg]\\
				&=&\frac{1}{2\pi(\sigma_1\sigma_2)\sqrt{1-\theta^2}}\exp\biggl\{-\dfrac{1}{2(1-\theta^2)}\bigg(\dfrac{z_1^2}{2\sigma_1^2}+\dfrac{z_2^2}{2\sigma_2^2}\bigg)\biggr\} 2 cosh\bigg(\dfrac{\theta z_1z_2}{(1-\theta^2)2\sigma_1\sigma_2}\bigg).
			\end{eqnarray*}
			\begin{eqnarray*}
				\mu_{12}&=&\BE\big[|Z_1||Z_2|\big]\\
				&=&\int_{-\infty}^{\infty}\int_{-\infty}^{\infty}|z_1||z_2|f_{Z_1,Z_2}(z_1,z_2)dz_1dz_2\\
				&=&\int_{0}^{\infty}\int_{0}^{\infty}z_1z_2f_{|Z_1|,|Z_2|}(z_1,z_2)dz_1dz_2\\
				&=&\int_{0}^{\infty}\int_{0}^{\infty}z_1z_2\dfrac{1}{2\pi(\sigma_1)(\sigma_2)\sqrt{1-\theta^2}}\exp\biggl\{-\dfrac{1}{2(1-\theta^2)}\bigg(\dfrac{z_1^2}{2\sigma_1^2}+\dfrac{z_2^2}{2\sigma_2^2}\bigg)\biggr\}\\ && \ \ \ \ \ \ \ \ \ \ \ \ \ \ \ \ \ \ \ \ \ \ \ \ \ \ \ \ \ \ \ \ \ \ \ cosh\bigg(\dfrac{\theta z_1z_2}{(1-\theta^2)2\sigma_1\sigma_2}\bigg)dz_1dz_2.
				%			&=&\dfrac{2(\sqrt{2}\sigma_1)(\sqrt{2}\sigma_2)}{\pi}(\theta \ \sin^{-1} \theta+\sqrt{1-\theta^2})\\
				%			&=&\dfrac{4\sigma_1\sigma_2}{\pi}(\theta \ \sin^{-1} \theta+\sqrt{1-\theta^2}).
			\end{eqnarray*}
			Now, by switching to polar coordinates with 
			$$\frac{z_1}{2\sqrt{1-\theta^2\sigma_1}}=r\cos\alpha, \hspace{5mm} \frac{z_2}{2\sqrt{1-\theta^2\sigma_2}}=r\sin\alpha, \hspace{5mm} dz_1dz_2=4(1-\theta^2)\sigma_1\sigma_2rdrd\alpha,$$
			we get,
			$$\mu_{12}=\frac{4(1-\theta^2)^{3/2}\sigma_1\sigma_2}{\pi}\int_0^{\pi/2}\int_0^\infty r^2 sin2\alpha [e^{-r^2(1-\theta sin2\alpha)}+e^{-r^2 (1+\theta sin2\alpha)}]r drd\alpha.$$
			Let $$r^2(1-\theta sin2\alpha)=u, \hspace{0.5cm}rdr=\frac{du}{2(1-\theta sin2\alpha)}.$$
			\begin{eqnarray*}
				\mu_{12}&=&\frac{4(1-\theta^2)^{3/2}\sigma_1\sigma_2}{\pi}\int_0^{\pi/2} \frac{sin2\alpha}{(1-\theta sin2\alpha)^2}\bigg[\int_0^\infty ue^{-u}du+\int_0^\infty ue^{-u\big(\frac{1+\theta sin2\alpha}{1-\theta sin2\alpha}\big)}du \bigg]d\alpha\\
				&=&\frac{4(1-\theta^2)^{3/2}\sigma_1\sigma_2}{\pi}\int_0^{\pi/2} \frac{sin2\alpha}{(1-\theta sin2\alpha)^2}\bigg[1+\bigg(\frac{1-\theta sin2\alpha}{1+\theta sin2\alpha}\bigg)^2\bigg]d\alpha\\
				&=&\frac{4(1-\theta^2)^{3/2}\sigma_1\sigma_2}{\pi}\int_0^{\pi/2}\bigg[ \frac{sin2\alpha}{(1-\theta sin2\alpha)^2}+\frac{sin2\alpha}{(1+\theta sin2\alpha)^2}\bigg]d\alpha\\
				&=&\dfrac{4\sigma_1\sigma_2}{\pi}(\theta \sin^{-1}\theta+\sqrt{1-\theta^2}).
			\end{eqnarray*}
			%\item
			%Combining all the terms, we get
			%	\begin{eqnarray*}
				%	\kappa (\theta)&=&\dfrac{1}{4}\Big[
				%	\mu_{12}+ \mu_1\mu_2-2\mu_{3}\Big]\\
				%	&=&\dfrac{1}{4}\bigg[\dfrac{4\sigma_1\sigma_2}{\pi}(\theta \\sin^{-1}\theta+\sqrt{1-\theta^2})+\dfrac{2\sigma_1}{\sqrt{\pi}}\dfrac{2\sigma_2}{\sqrt{\pi}}-2\sigma_1\sigma_2\bigg[\dfrac{2}{\pi}\sqrt{\dfrac{1}{(4-\theta^2)}}+\dfrac{2\theta}{\pi} \\sin^{-1}\bigg(\dfrac{\theta}{2}\bigg)+\dfrac{2(3-\theta^2)}{\pi\sqrt{4-\theta^2}}
				%	\bigg]\bigg]
				%\end{eqnarray*}
				The expression in Part \ref{part25} can be obtained by combining \ref{part21}, \ref{part22}, \ref{part23}, and, \ref{part24}. This completes the proof of Lemma \ref{bvncalc1}.
		\end{proof}
		It is easily seen that when 
	 $\theta=0$, the above expressions yield
		\begin{equation*}
			\mu_{3}= 
			\dfrac{2\sigma_1}{\sqrt{\pi}}\dfrac{2\sigma_2}{\sqrt{\pi}}=\mu_1\mu_2, \ \ \mu_{12}=\dfrac{4\sigma_1\sigma_2}{\pi}=\mu_1\mu_2.
	\end{equation*}
Hence 
\begin{eqnarray*}
	\kappa (0)&=&\dfrac{1}{4}\bigg[\dfrac{4\sigma_1\sigma_2}{\pi}+\dfrac{4\sigma_1\sigma_2}{\pi}-2\big(\dfrac{4\sigma_1\sigma_2}{\pi}\big)\bigg]=0.
\end{eqnarray*}
	\begin{figure}[H]
	\centering
	\includegraphics[width=0.6\linewidth]{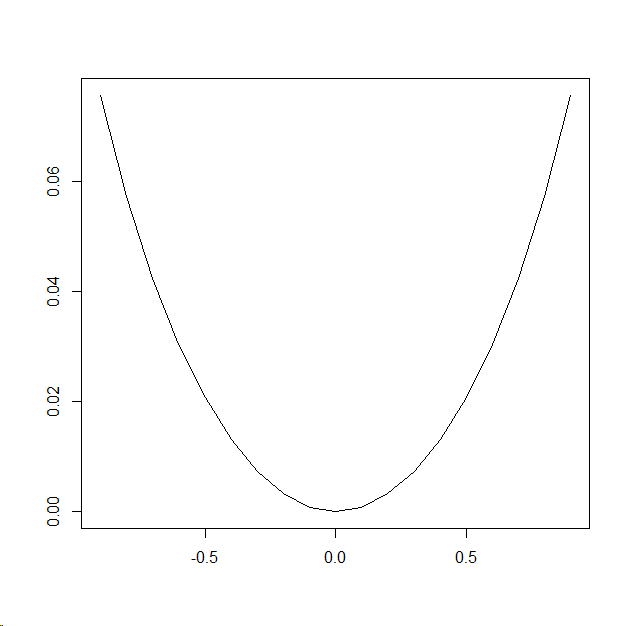}
	\caption{Plot of ($\theta$, $\kappa(\theta)$) for Bivariate normal distribution}
	\label{fig:kappathethabvn}
\end{figure}
\begin{lemma}
	 $\kappa(\theta)$ is a convex function of $\theta\in (-1, \ 1)$ when $(X, Y)$ is bivariate normal.
\end{lemma}
\begin{proof}
	Recall that in this case, we have shown that 
	\begin{equation}
		\kappa (\theta)	=\dfrac{\sigma_1\sigma_2}{\pi}\left(\theta \sin^{-1}\theta+\sqrt{1-\theta^2}+1-{\theta}\sin^{-1}\bigg(\dfrac{\theta}{2}\bigg)-\sqrt{{4-\theta^2}}\right) \label{kappabn}
	\end{equation}
	From equation \eqref{kappabn}, we get 
	\begin{eqnarray*}
		\kappa^\prime(\theta)
		%&=&\dfrac{\sigma_1\sigma_2}{\pi}\left[\left(\sin^{-1}\theta+\dfrac{\theta}{\sqrt{1-\theta^2}}\right)-\dfrac{\theta}{\sqrt{1-\theta^2}}-\left[\sin^{-1}\left(\frac{\theta}{2}\right)+\theta \left(\dfrac{d}{d\theta}\sin^{-1}\bigg(\dfrac{\theta}{2}\bigg)\times \dfrac{d}{d\theta} \left(\frac{\theta}{2}\right) \right)\right]-\left(\dfrac{-2\theta}{2\sqrt{4-\theta^2}}\right)\right]\\
		&=&\dfrac{\sigma_1\sigma_2}{\pi}\left[\left(\sin^{-1}\theta+\dfrac{\theta}{\sqrt{1-\theta^2}}\right)-\dfrac{\theta}{\sqrt{1-\theta^2}}-\left[\sin^{-1}\left(\theta/2\right)+\theta \dfrac{1}{\sqrt{1-\big(\frac{\theta}{2}\big)^2}}\times \frac{1}{2}\right]+\dfrac{\theta}{\sqrt{4-\theta^2}}\right]\\
		&=&\dfrac{\sigma_1\sigma_2}{\pi}\left(\sin^{-1}\theta+\dfrac{\theta}{\sqrt{1-\theta^2}}-\dfrac{\theta}{\sqrt{1-\theta^2}}-\sin^{-1}\left(\theta/2\right)-\dfrac{\theta}{\sqrt{4-\theta^2}}+\dfrac{\theta}{\sqrt{4-\theta^2}}\right)\\
		&=&\dfrac{\sigma_1\sigma_2}{\pi}\bigg[\sin^{-1}\theta-\sin^{-1}\left(\theta/2\right)\bigg].
	\end{eqnarray*}
	Now, it easily follows that the second derivative of $\kappa(\theta)$ is 
	\begin{eqnarray}
		\kappa^{\prime\prime}(\theta)&=&\dfrac{\sigma_1\sigma_2}{\pi}\bigg[\dfrac{1}{\sqrt{1-\theta^2}}-\dfrac{1}{\sqrt{1-\big(\frac{\theta}{2}\big)^2}}\times \frac{1}{2}]\nonumber\\&=&
		\dfrac{\sigma_1\sigma_2}{\pi}\bigg[\dfrac{1}{\sqrt{1-\theta^2}}-\dfrac{1}{\sqrt{4-\theta^2}}\bigg]\geq 0, \ \text{since}\ \ \theta\in (-1, \ 1).\label{kappa2bvn}
	\end{eqnarray}
We can also see that the result in \eqref{kappa2bvn} follows from the plot of  $\kappa^{\prime\prime}(\theta)$ vs $\theta$ provided in Figure \ref{fig:kappadrvtvbvn}.	
	\begin{figure}[H]
		\centering
		\includegraphics[width=0.5\linewidth]{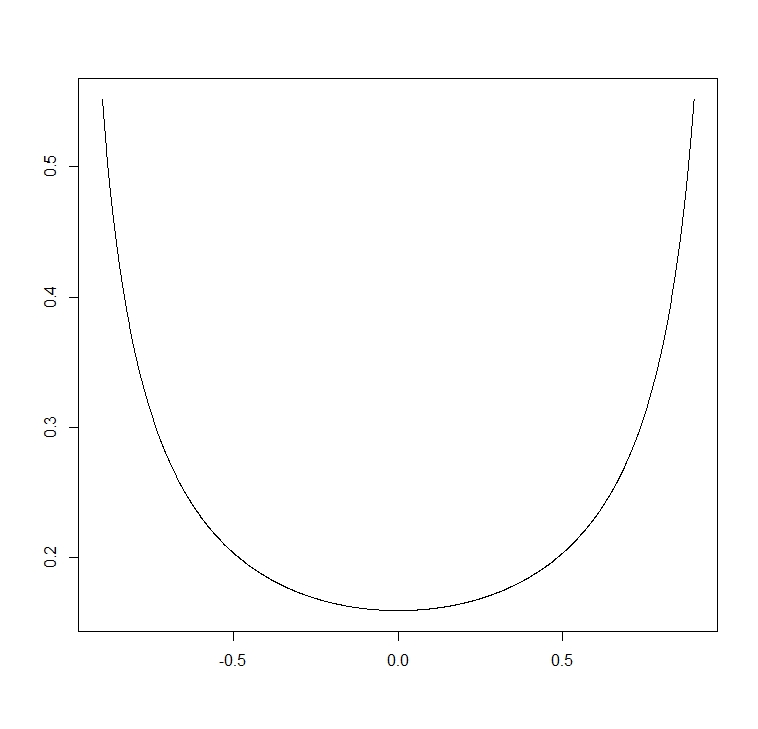}
		\caption{Plot of ($\theta$, $\kappa^{\prime\prime}(\theta)$) for $\theta \in (-1,1)$ for BVN.}
		\label{fig:kappadrvtvbvn}
	\end{figure}
Hence $\kappa(\theta)$ is a convex function of $\theta \in (-1,\ 1)$ when $(X,Y)$ follows bivariate normal distribution. 
\end{proof}
\subsection{Estimation of $\kappa$}

Now suppose we have a sample $\{(X_i, Y_i)\}$ which are i.i.d.~with distribution $F_{12}$ and the respective marginal distributions $F_1$ and $F_2$. Note that the last but one expression of $\kappa$ in Theorem \ref{theo:altkappa}, has \textit{four} unknown quantities.
The appropriate $U$-statistics unbiased estimates of the three quantities $\mu_1$, $\mu_2$ and $\mu_{12}$ are given by 

\begin{eqnarray*} 
U_{1n} &:=&  {n\choose 2}^{-1}\sum_{1\leq i < j \leq n} |X_i-X_j| \ \text{(for}\ \mu_1),\\
U_{2n} &:=&  {n\choose 2}^{-1}\sum_{1\leq i < j \leq n} |Y_i-Y_j| \  \text{(for}\ \mu_2), \\
U_{12n}&:=& {n\choose 2}^{-1}\sum_{1\leq i < j \leq n} |X_i-X_j|\ |Y_i-Y_j| \ \text{(for}\ \mu_{12}).\end{eqnarray*}

%To obtain an unbaised estimate of and $\BE\big[g_{F_{1}}(X)g_{F_{2}}(Y)\big]$, 
To obtain a $U$-statistic unbiased estimate of the fourth quantity $\mu_3$, 
%note that $$\mu_3=\BE_{F_{12}}\big[g_{F_{1}}(X)g_{F_{2}}(Y)\big]=\BE_{F_{12}} [|X_1-X_2| \  |Y_1-Y_3|\big], $$ where $(X_i, Y_i)$ are i.i.d.~with distribution $F_{12}$ (and marginal distributions $F_1$ and $F_2$). 
define the function $f$ as  
$$f\big( (x_1, y_1), (x_2, y_2), (x_3, y_3)\big)=|x_1-x_2|\ |y_1-y_3|.$$
Now, clearly since 	Since $X_2$ and $Y_3$ are independent we have,
	\begin{eqnarray*}
	\BE\big[f\big( (X_1, Y_1), (X_2, Y_2), (X_3, Y_3)\big)\big]&=&\BE\big[\BE[f\big( (X_1, Y_1), (X_2, Y_2), (X_3, Y_3)\big)] | (X_1,Y_1)\big]\\&=&\BE\big[\BE[|X_1-X_2|\ |Y_1-Y_3||] \mid (X_1,Y_1)\big]\\&=&\BE\big[\BE[|X_1-X_2| |X_1]\ \BE [|Y_1-Y_3|| | Y_1] \big]\\&=& \BE\big[g_{F_{1}}(X_1) g_{F_{2}}(Y_1) \big]=\mu_3
\end{eqnarray*} 
%Clearly,  
%$$\BE_{F_{12}}\big[f\big( (X_1, Y_1), (X_2, Y_2), (X_3, Y_3)\big)\big]=\mu_3.$$ 

However, 
$f$ is not invariant under the permutation of the indices $1, 2, 3$. Hence the appropriate $U$-statistics unbiased estimate of $\mu_3$ is obtained by symmetrising $f$. So, let $h$ be the symmetrised version of $f$:
\begin{eqnarray}
h\big( (x_1, y_1), (x_2, y_2), (x_3, y_3)\big) &=&
\dfrac{1}{6}\big[
|x_1-x_2|\ |y_1-y_3|
+|x_2-x_1|\ |y_2-y_3|\nonumber\\
&&\  
+|x_1-x_3|\ |y_1-y_2|+|x_2-x_3|\ |y_2-y_1|\nonumber\\
&&\  
+|x_3-x_1|\ |y_3-y_2|
+|x_3-x_2|\ |y_3-y_1|\big].\label{eq:symmetrich}
\end{eqnarray}

Let 
$$U_{3n}:= {n\choose 3}^{-1}\sum_{1\leq i < j < k \leq n}	h\big( (X_i, Y_i), (X_j, Y_j), (X_k, Y_k)\big).$$

Then by construction, 
$$\BE_{F_{12}}[U_{3n}]=\mu_3=\BE_{F_{12}}\big[g_{F_{1}}(X)g_{F_{2}}(Y)\big].$$  
Thus in view of Theorem \ref{theo:altkappa}, we define an estimate of $\kappa$ as,
$$\kappa^*=\dfrac{1}{4}\big[
U_{12n}+ U_{1n} U_{2n}- 2 U_{3n}\big].$$

Incidentally, $\kappa^{*}$ is not an unbiased estimate of $\kappa$, because $U_{1n} U_{2n}$ is not in general an unbiased estimate of $\mu_1\mu_2$. However, by the SLLN for $U$-statistics, $\kappa^*$ is a strongly consistent estimate of $\kappa$. Since this estimator is a nice explicit function of four $U$-statistics, its asymptotic properties can be derived using the theory of $U$-statistics.
When $\{X_i\}$ and $\{Y_i\}$ are independent of each other, $U_{1n}$ and $U_{2n}$ are independent. In that case $U_{1n}U_{2n}$ is an unbiased estimate of $\mu_1\mu_2$. Moreover $\mu_{12}=\mu_3=\mu_1\mu_2$. Hence $\kappa=0=\BE_{F_{1}\otimes F_{2}}[\kappa^{*}]$.

\subsection{Relation between $\kappa^*$, $\tilde\kappa$,  and $\tilde\kappa$}\label{kappastartilde}

Now we explore the relation between $\kappa^*$, $\tilde\kappa$, and $\hat{\kappa}$. We need the following lemma.

\begin{lemma}[\cite{lee2019u}]\label{VtoUtherm}
Let $V_n$ be a $V$-statistic based on a symmetric kernel $\psi$ of degree $k:$
\begin{equation*}
	V_n=n^{-k}\sum_{i_1=1}^{n}\ldots\sum_{i_k=1}^{n}\psi(X_{i_1},\ldots,X_{i_k}).
\end{equation*} 

Then 
\begin{equation*}
	V_n:=n^{-k}\sum_{j=1}^{k}j!S_{k}^{(j)}{n \choose j}U_n^{(j)},
\end{equation*}
where $U_n^{(j)}$ is a $U$-statistic of degree $j$ with kernel 
%$\phi_{(j)}$ 
%of $U_n^{(j)}$ is 
%given by
%		\begin{align}
	\[	\phi_{(j)}(x_1,\ldots,x_j):=(j!S_{k}^{(j)})^{-1} \sum_{(j)}^{*}\psi(x_{i_1},\ldots,x_{i_k}),
	\]
	%			\end{align}
and the sum $\sum_{(j)}^{*}$ is taken over all $k$-tuples $(i_1,\ldots,i_k)$ formed from ${1,2,\ldots,j}$ having exactly $j$ distinct indices, and $S_{k}^{(j)}$ are Stirling numbers of the second kind.
\end{lemma}
We have 
for $k\geq 1$ $$S_{k}^{(1)}=S_{k}^{(k)}=1$$
$$S_{k}^{(2)}=2^{(k-1)}-1$$
$$S_{k}^{(k-1)}=\frac{k(k-1)}{2}$$
%		We use the above theorem to prove the following result;

The following lemma gives the relation between $\tilde{\kappa}$, $\hat\kappa$,  and $\kappa^{*}$.

\begin{lemma}\label{kappacurlstar}
	\begin{eqnarray}
		\tilde{\kappa}&=&{n \choose 2}^{-1}\sum_{1\leq i <j \leq n}\tilde{h}_{\hat{F}_1}(x_i,x_j)\tilde{h}_{\hat{F}_2}(y_i,y_j)\nonumber \\
		%			&=&\frac{1}{4}\big[U_{12n}-2U_{3n}+U_{1n}U_{2n}\big]+\frac{1}{4}\bigg[\frac{-2n}{(n-1)^2}U_{12n}+\bigg(\frac{2}{(n-1)^2}\bigg)U_{3n}+\bigg(\frac{2}{n-1}\bigg)U_{1n}U_{2n}\bigg]\\
		&=&\kappa^{*}+\frac{1}{4}\big[\frac{-2n}{(n-1)^2}U_{12n}+\frac{2}{(n-1)^2}U_{3n}+\frac{2}{n-1}U_{1n}U_{2n}\big].\label{kappakappa*}\\
		%\end{eqnarray}
		%\begin{eqnarray}
		\hat{\kappa}&=&\frac{1}{n^2}\sum_{i,j=1}^{n}{h}_{\hat{F}_1}(x_i,x_j){h}_{\hat{F}_2}(y_i,y_j)\nonumber \\
		%			&=&\frac{1}{4}\big[U_{12n}-2U_{3n}+U_{1n}U_{2n}\big]+\frac{1}{4}\bigg[\frac{-2n}{(n-1)^2}U_{12n}+\bigg(\frac{2}{(n-1)^2}\bigg)U_{3n}+\bigg(\frac{2}{n-1}\bigg)U_{1n}U_{2n}\bigg]\\
		%		&=&\frac{1}{4}[V_{12n}-2V_{3n}+V_{1n}V_{2n}]\\
		&=&\kappa^{*}+\frac{1}{4}\big[\frac{(2-3n)}{n^2}U_{12n}-2\big(\frac{2-3n}{n^2}\big)U_{3n}+\big(\frac{1-2n}{n^2}\big)U_{1n}U_{2n}\big]
	\end{eqnarray}
\end{lemma}
\begin{proof}
	We shall need the following four $V$ statistics: 
	\begin{eqnarray*} 
		V_{1n} &=&  \frac{1}{n^2}\sum_{i,j=1}^{n} |X_i-X_j|, \hspace{1.5cm}
		V_{2n} =  \frac{1}{n^2}\sum_{i,j=1}^{n} |Y_i-Y_j|, \\
		%\hspace{0.5cm}
		V_{12n}&=& \frac{1}{n^2}\sum_{i,j=1}^{n} |X_i-X_j|\ |Y_i-Y_j|, \ \
		%		\end{eqnarray*}
	V_{3n}= \frac{1}{n^3}\sum_{i,j,k=1}^{n}	h\big( (X_i, Y_i), (X_j, Y_j), (X_k, Y_k)\big),
\end{eqnarray*}
where the function $h$ is as in (\ref{eq:symmetrich}).
By Lemma \ref{VtoUtherm}, using the special nature of the kernels involved, we can write the above $V$-statistics 
as
\begin{eqnarray*} 
	n^2V_{1n}=2{n \choose 2}U_{1n},\;\;
	n^2V_{2n}=2{n \choose 2}U_{2n},\;\;
	n^2V_{12n}=2{n \choose 2}U_{12n},\;\;
	n^3V_{3n}=6{n \choose 3}U_{3n}+2{n \choose 2}U_{12n}.
\end{eqnarray*}
%By some algebra, whose details are available in \cite{???}, 
We have,
\begin{eqnarray}
	\tilde{\kappa}&=&{n \choose 2}^{-1}\sum_{1\leq i , j\leq n}\tilde{h}_{\hat{F}_1}(x_i,x_j)\tilde{h}_{\hat{F}_2}(y_i,y_j)\nonumber\\
	&=&{n \choose 2}^{-1}\frac{1}{2}\big[\sum_{i, j=1}^{n}\tilde{h}_{\hat{F}_1}(x_i,x_j)\tilde{h}_{\hat{F}_2}(y_i,y_j)-\sum_{i=1}^{n}\tilde{h}_{\hat{F}_1}(x_i,x_i)\tilde{h}_{\hat{F}_2}(y_i,y_i)\big]\nonumber
	%	\\&=&\frac{1}{4}{n \choose 2}^{-1}\frac{1}{2}\big[n^2V_{12n}-2\frac{n}{(n-1)^2}n^3V_{3n} +\frac{n+1}{n(n-1)^2}(n^2V_{1n})(n^2V_{2n})\big]
\end{eqnarray}
Consider 
\begin{eqnarray*}
	&&\sum_{i, j=1}^{n}\tilde{h}_{\hat{\mathbb{F}}_1}(x_i,x_j)\tilde{h}_{\hat{\mathbb{F}}_2}(y_i,y_j)\\ \nonumber&=&\frac{1}{4}\sum_{i, j=1}^{n}\bigg(|x_i-x_j|-\frac{n}{n-1}(\frac{1}{n}\sum_{k=1}^{n}|x_i-x_k|+\frac{1}{n}\sum_{k=1}^{n}|x_k-x_j|-\frac{1}{n^2}\sum_{k,l=1}^{n}|x_k-x_l|)\bigg)\\ && \times  \bigg(|y_i-y_j|-\frac{n}{n-1}(\frac{1}{n}\sum_{k=1}^{n}|y_i-y_k|+\frac{1}{n}\sum_{k=1}^{n}|y_k-y_j|-\frac{1}{n^2}\sum_{k,l=1}^{n}|y_k-y_l|)\bigg)
\end{eqnarray*}
%	\\ \nonumber
%	&=&\frac{1}{4}\sum_{i, j=1}^{n}\bigg(|x_i-x_j||y_i-y_j|-|x_i-x_j|\frac{n}{n-1}\big[\frac{1}{n}\sum_{k=1}^{n}|y_i-y_k|+\frac{1}{n}\sum_{k=1}^{n}|y_k-y_j|-\frac{1}{n^2}\sum_{k,l=1}^{n}|y_k-y_l|\big]\\ && \ -|y_i-y_j|\frac{n}{n-1}\big[\frac{1}{n}\sum_{k=1}^{n}|x_i-x_k|+\frac{1}{n}\sum_{k=1}^{n}|x_k-x_j|-\frac{1}{n^2}\sum_{k,l=1}^{n}|x_k-x_l|\big]\\ && \ + \big(\frac{n}{n-1}\big)^2\big[\frac{1}{n}\sum_{k=1}^{n}|x_i-x_k|+\frac{1}{n}\sum_{k=1}^{n}|x_k-x_j|-\frac{1}{n^2}\sum_{k,l=1}^{n}|x_k-x_l|\big]\\ && \ \big[\frac{1}{n}\sum_{k=1}^{n}|y_i-y_k|+\frac{1}{n}\sum_{k=1}^{n}|y_k-y_j|-\frac{1}{n^2}\sum_{k,l=1}^{n}|y_k-y_l|\big] \bigg)\\ \nonumber
%\end{eqnarray*}
\begin{eqnarray*}
	&=&\frac{1}{4}\Bigg(\sum_{i, j=1}^{n}|x_i-x_j||y_i-y_j|-\frac{1}{n-1}\sum_{i,j,k=1}^{n}|x_i-x_j||y_i-y_k|-\frac{1}{n-1}\sum_{i,j,k=1}^{n}|x_i-x_j||y_j-y_k|\\ && \
	+\frac{1}{n(n-1)}\sum_{i,j,k,l=1}^{n}|x_i-x_j||y_l-y_k|-\frac{1}{n-1}\sum_{i,j,k=1}^{n}|x_i-x_k||y_i-y_j|\\ && \ -\frac{1}{n-1}\sum_{i,j,k=1}^{n}|x_j-x_k||y_i-y_j| +\frac{1}{n(n-1)}\sum_{i,j,k,l=1}^{n}|x_i-x_j||y_l-y_k| \\ && \ + \big(\frac{n}{n-1}\big)^2\bigg[\frac{n}{n^2}\sum_{i,k,l=1}^{n}|x_i-x_k||y_i-y_l|+\frac{1}{n^2}\sum_{i,j,k,l=1}^{n}|x_i-x_k||y_j-y_l|\\ && \   -\frac{1}{n^2}\sum_{i,j,k,l=1}^{n}|x_i-x_k||y_j-y_l|+\frac{1}{n^2}\sum_{i,j,k,l=1}^{n}|x_i-x_k||y_j-y_l|-\frac{1}{n^2}\sum_{i,j,k,l=1}^{n}|x_i-x_k||y_j-y_l|\\ && \ + \frac{n}{n^2}\sum_{j,k,l=1}^{n}|x_j-x_k||y_j-y_l|-\frac{1}{n^2}\sum_{i,j,k,l=1}^{n}|x_j-x_k||y_i-y_l|-\frac{1}{n^2}\sum_{i,j,k,l=1}^{n}|x_k-x_l||y_j-y_i| \\ && \ +\frac{1}{n^2}\sum_{i,j,k,l=1}^{n}|x_l-x_k||y_j-y_i|\bigg]\Bigg)\\ \nonumber
	&=&\frac{1}{4}\Bigg(\sum_{i, j=1}^{n}|x_i-x_j||y_i-y_j|-\frac{4}{n-1}\sum_{i,j,k=1}^{n}|x_i-x_j||y_i-y_k|+\frac{2}{n(n-1)}\sum_{i,j,k,l=1}^{n}|x_i-x_j||y_k-y_l|\\ && \ + \big(\frac{n}{n-1}\big)^2\frac{2n}{n^2}\sum_{i,j,k=1}^{n}|x_i-x_j||y_i-y_k|-\big(\frac{n}{n-1}\big)^2\frac{1}{n^2}\sum_{i,j,k,l=1}^{n}|x_i-x_j||y_k-y_l|\Bigg)\\ \nonumber\\
	%	&=&\frac{1}{4}\Bigg(\sum_{i, j=1}^{n}|x_i-x_j||y_i-y_j|+\Bigg[\big(\frac{n}{n-1}\big)^2\frac{2n}{n^2}-\frac{4}{n-1}\Bigg]\sum_{i,j,k=1}^{n}|x_i-x_j||y_i-y_k|\\ && \  +\Bigg[\frac{2}{n(n-1)}-\big(\frac{n}{n-1}\big)^2\frac{1}{n^2}\Bigg]\sum_{i,j,k,l=1}^{n}|x_i-x_j||y_k-y_l|\Bigg)\\
	%\end{eqnarray*}
	%   &=&\frac{1}{4}\Bigg(\sum_{i, j=1}^{n}|x_i-x_j||y_i-y_j|+\Bigg[\frac{2n}{(n-1)^2}-\frac{4}{n-1}\Bigg]\sum_{i,j,k=1}^{n}|x_i-x_j||y_i-y_k|\\ && \  +\Bigg[\frac{2}{n(n-1)}-\frac{1}{(n-1)^2}\Bigg]\sum_{i,j,k,l=1}^{n}|x_i-x_j||y_k-y_l|\Bigg)\\
	%\begin{eqnarray*}
	%	%	&&\sum_{i, j=1}^{n}\tilde{h}_{\hat{\mathbb{F}}_1}(x_i,x_j)\tilde{h}_{\hat{\mathbb{F}}_2}(y_i,y_j)\\ \nonumber
	%	&=&\frac{1}{4}\Bigg(\sum_{i, j=1}^{n}|x_i-x_j||y_i-y_j|+\Bigg[\frac{2n}{(n-1)^2}-\frac{4}{n-1}\Bigg]\sum_{i,j,k=1}^{n}|x_i-x_j||y_i-y_k|\\ && \  +\Bigg[\frac{2}{n(n-1)}-\frac{1}{(n-1)^2}\Bigg]\sum_{i,j,k,l=1}^{n}|x_i-x_j||y_k-y_l|\Bigg)\\
	%	&=&\frac{1}{4}\Bigg(\sum_{i, j=1}^{n}|x_i-x_j||y_i-y_j|+\Bigg[\frac{2n-4(n-1)}{(n-1)^2}\Bigg]\sum_{i,j,k=1}^{n}|x_i-x_j||y_i-y_k| +\Bigg[\frac{2(n-1)-n}{n(n-1)^2}\Bigg]\sum_{i,j,k,l=1}^{n}|x_i-x_j||y_k-y_l|\Bigg)\\
	%	&=&\frac{1}{4}\Bigg(\sum_{i, j=1}^{n}|x_i-x_j||y_i-y_j|+\Bigg[\frac{-2n+4}{(n-1)^2}\Bigg]\sum_{i,j,k=1}^{n}|x_i-x_j||y_i-y_k| +\Bigg[\frac{n-2}{n(n-1)^2}\Bigg]\sum_{i,j,k,l=1}^{n}|x_i-x_j||y_k-y_l|\Bigg)\\
	&=&\frac{1}{4}\bigg[\sum_{i,j=1}^{n}|x_i-x_j||y_i-y_j|-\frac{2(n-2)}{(n-1)^2}\sum_{i,j,k=1}^{n}|x_i-x_j||y_i-y_k|+\frac{n-2}{n(n-1)^2}\sum_{i,j,k,l=1}^{n}|x_i-x_j||y_k-y_l|\bigg]
\end{eqnarray*}
Now,
\begin{eqnarray*}
	&&\sum_{i=1}^{n}\tilde{h}_{\hat{\mathbb{F}}_1}(x_i,x_i)\tilde{h}_{\hat{\mathbb{F}}_2}(y_i,y_i)\\&=&\frac{1}{4}\sum_{i=1}^{n}\bigg(-\frac{n}{n-1}(\frac{1}{n}\sum_{k=1}^{n}|x_i-x_k|+\frac{1}{n}\sum_{k=1}^{n}|x_k-x_i|-\frac{1}{n^2}\sum_{k,l=1}^{n}|x_k-x_l|)\bigg)\\ && \times  \bigg(-\frac{n}{n-1}(\frac{1}{n}\sum_{k=1}^{n}|y_i-y_k|+\frac{1}{n}\sum_{k=1}^{n}|y_k-y_i|-\frac{1}{n^2}\sum_{k,l=1}^{n}|y_k-y_l|)\bigg)\\&=&\frac{1}{4}\sum_{i=1}^{n}\bigg(-\frac{n}{n-1}(\frac{2}{n}\sum_{k=1}^{n}|x_k-x_i|-\frac{1}{n^2}\sum_{k,l=1}^{n}|x_k-x_l|)\bigg)\\ && \times  \bigg(-\frac{n}{n-1}(\frac{2}{n}\sum_{k=1}^{n}|y_k-y_i|-\frac{1}{n^2}\sum_{k,l=1}^{n}|y_k-y_l|)\bigg)\end{eqnarray*}
\begin{eqnarray*}
	%	\\
	%	&=& \frac{1}{4}\sum_{i=1}^{n}\bigg(\frac{n}{n-1}\bigg)^2\bigg(\frac{4}{n^2}\sum_{k=1}^{n}\sum_{l=1}^{n}|x_k-x_i||y_i-y_l|-\frac{2}{n^3}(\sum_{k=1}^{n}|x_i-x_k|)(\sum_{j,l=1}^{n}|y_j-y_l|)\\ &&  -\frac{2}{n^3}(\sum_{k,l=1}^{n}|x_k-x_l)(\sum_{j=1}^{n}|y_i-y_j|)+\frac{1}{n^4}(\sum_{k,l=1}^{n}|x_k-x_l|)(\sum_{j,m=1}^{n}|y_j-y_m|)\bigg)\\
	&=& \frac{1}{4}\bigg(\frac{n}{n-1}\bigg)^2\bigg[\frac{4}{n^2}\sum_{i,k,l=1}^{n}|x_k-x_i||y_i-y_l|-\frac{2}{n^3}\sum_{i,j,k,l=1}^{n}|x_k-x_i||y_j-y_l|
	\\ &&  -\frac{2}{n^3}\sum_{i,j,k,l=1}^{n}|x_k-x_l||y_i-y_j|+\frac{1}{n^3}\sum_{i,j,k,l=1}^{n}|x_k-x_l||y_j-y_i|\bigg]
	\\&=&\frac{1}{4}\bigg[\frac{4}{(n-1)^2}\sum_{i,j,k=1}^{n}|x_i-x_j||y_i-y_k|-\frac{3}{n(n-1)^2}\sum_{i,j,k,l=1}^{n}|x_k-x_l||y_j-y_i|\bigg]
\end{eqnarray*}
\begin{eqnarray*}
	\tilde{\kappa}&=&
	{n \choose 2}^{-1}\frac{1}{2}\bigg[\sum_{i, j=1}^{n}\tilde{h}_{\hat{\mathbb{F}}_1}(x_i,x_j)\tilde{h}_{\hat{\mathbb{F}}_2}(y_i,y_j)-\sum_{i=1}^{n}\tilde{h}_{\hat{\mathbb{F}}_1}(x_i,x_i)\tilde{h}_{\hat{\mathbb{F}}_2}(y_i,y_i)\bigg]\\
	%	&=&\frac{1}{4}{n \choose 2}^{-1}\bigg[\frac{1}{2}\bigg[\sum_{i, j=1}^{n}|x_i-x_j||y_i-y_j|-\frac{2(n-2)}{(n-1)^2}\sum_{i,j,k=1}^{n}|x_i-x_j||y_i-y_k|\\
	%	&& +\frac{n-2}{n(n-1)^2}\sum_{i,j,k,l=1}^{n}|x_i-x_j||y_k-y_l| -\frac{4}{(n-1)^2}\sum_{i,k,l=1}^{n}|x_k-x_i||y_i-y_l|\\
	%	&& +\frac{3}{n(n-1)^2}\sum_{i,j,k,l=1}^{n}|x_k-x_l||y_j-y_i|\bigg]\bigg]\\
	&=&\frac{1}{4}{n \choose 2}^{-1}\bigg[\frac{1}{2}\bigg[\sum_{i,j=1}^{n}|x_i-x_j||y_i-y_j|-2.\frac{n}{(n-1)^2}\sum_{i,j,k=1}^{n}|x_i-x_j||y_i-y_k|\\&& \ +\frac{n+1}{n(n-1)^2}\sum_{i,j,k,l=1}^{n}|x_k-x_l||y_j-y_i|\bigg]\bigg]\\
	&=&\frac{1}{4}{n \choose 2}^{-1}\bigg[\frac{1}{2}\bigg[\sum_{i,j=1}^{n}|x_i-x_j||y_i-y_j|-2.\frac{n}{(n-1)^2}\sum_{i,j,k=1}^{n}|x_i-x_j||y_i-y_k|\\&& \ +\frac{n+1}{n(n-1)^2}\sum_{i,j=1}^{n}|x_i-x_j|\sum_{k,l=1}^{n}|y_j-y_i|\bigg]\bigg]\\
	&=&\frac{1}{4}{n \choose 2}^{-1}\bigg[\frac{1}{2}\bigg[n^2V_{12n}-2.\frac{n}{(n-1)^2}n^3V_{3n} +\frac{n+1}{n(n-1)^2}(n^2V_{1n})(n^2V_{2n})\bigg]\bigg],
\end{eqnarray*}

Expressing the above $V$-statistics in terms of corresponding $U$-statistics we obtain 
\begin{eqnarray*}
	\tilde{\kappa}
	&=&\frac{1}{4}{n \choose 2}^{-1}\bigg[\frac{1}{2}\bigg[2{n \choose 2}U_{12n}-2.\frac{n}{(n-1)^2}\big[6{n \choose 3}U_{3n}+2{n \choose 2}U_{12n}\big] \\ && \ \ \ \  \ \ \ \ \ \ \ \ \ \ \ \ \ \ \ \ \ \ \ \ \ \ \ \ \ \ \ \ \ \ \ \ \ \ \ +\frac{n+1}{n(n-1)^2}(2{n \choose 2}U_{1n})(2{n \choose 2}U_{2n})\bigg]\bigg]\\
	&=&\frac{1}{4}\bigg[\bigg(1-\frac{2n}{(n-1)^2}\bigg)U_{12n}-\frac{2n(n-2)}{(n-1)^2}U_{3n}+\frac{n+1}{n-1}U_{1n}U_{2n}\bigg]\\
	&=&\frac{1}{4}\bigg[\bigg(1-\frac{2n}{(n-1)^2}\bigg)U_{12n}-2\bigg(\frac{n(n-2)}{(n-1)^2}+1-1\bigg)U_{3n}+\bigg(\frac{n+1}{n-1}+1-1\bigg)U_{1n}U_{2n}\bigg]\\
	 &=&\frac{1}{4}\big[U_{12n}-2U_{3n}+U_{1n}U_{2n}\big]+\frac{1}{4}\bigg[\frac{-2n}{(n-1)^2}U_{12n}-2\bigg(\frac{n(n-2)}{(n-1)^2}-1\bigg)U_{3n}+\bigg(\frac{n+1}{n-1}-1\bigg)U_{1n}U_{2n}\bigg]\\
	&=&\frac{1}{4}\big[U_{12n}-2U_{3n}+U_{1n}U_{2n}\big]+\frac{1}{4}\bigg[\frac{-2n}{(n-1)^2}U_{12n}+\bigg(\frac{2}{(n-1)^2}\bigg)U_{3n}+\bigg(\frac{2}{n-1}\bigg)U_{1n}U_{2n}\bigg]
\end{eqnarray*}%\end{proof}
This establishes the relation between $\tilde\kappa$ and $\kappa^*$.\\
The relation between $\hat\kappa$ and $\kappa^*$ is obtained as follows, 
\begin{eqnarray*}
	\hat{\kappa}&=&\frac{1}{n^2}\sum_{i,j=1}^{n}{h}_{\hat{F}_1}(x_i,x_j){h}_{\hat{F}_2}(y_i,y_j)\\
	&=&\frac{1}{4}\frac{1}{n^2} \sum_{i,j=1}^{n}\left(|x_i-x_j|-\frac{1}{n}\sum_{k=1}^{n}|x_i-x_k|-\frac{1}{n}\sum_{k=1}^{n}|x_k-x_j|+\frac{1}{n^2}\sum_{k,l=1}^{n}|x_k-x_l|\right)\\ && \ \times
	\left(|y_i-y_j|-\frac{1}{n}\sum_{k=1}^{n}|y_i-y_k|-\frac{1}{n}\sum_{k=1}^{n}|y_k-y_j|+\frac{1}{n^2}\sum_{k,l=1}^{n}|y_k-y_l|\right)\\
	&=&\frac{1}{4}\Bigg(\frac{1}{n^2}\sum_{i, j=1}^{n}|x_i-x_j||y_i-y_j|-\frac{1}{n^3}\sum_{i,j,k=1}^{n}|x_i-x_j||y_i-y_k|-\frac{1}{n^3}\sum_{i,j,k=1}^{n}|x_i-x_j||y_j-y_k|\\ && \
	+\frac{1}{n^4}\sum_{i,j,k,l=1}^{n}|x_i-x_j||y_l-y_k|-\frac{1}{n^3}\sum_{i,j,k=1}^{n}|x_i-x_k||y_i-y_j|-\frac{1}{n^3}\sum_{i,j,k=1}^{n}|x_j-x_k||y_i-y_j|\\ && \ +\frac{1}{n^4}\sum_{i,j,k,l=1}^{n}|x_i-x_j||y_l-y_k|+ \frac{1}{n^3}\sum_{i,k,l=1}^{n}|x_i-x_k||y_i-y_l|+\frac{1}{n^4}\sum_{i,j,k,l=1}^{n}|x_i-x_k||y_j-y_l|\\ && \   -\frac{1}{n^4}\sum_{i,j,k,l=1}^{n}|x_i-x_k||y_j-y_l|+\frac{1}{n^4}\sum_{i,j,k,l=1}^{n}|x_i-x_k||y_j-y_l|-\frac{1}{n^4}\sum_{i,j,k,l=1}^{n}|x_i-x_k||y_j-y_l|\\ && \ + \frac{1}{n^3}\sum_{j,k,l=1}^{n}|x_j-x_k||y_j-y_l|-\frac{1}{n^4}\sum_{i,j,k,l=1}^{n}|x_j-x_k||y_i-y_l|-\frac{1}{n^4}\sum_{i,j,k,l=1}^{n}|x_k-x_l||y_j-y_i| \\ && \ +\frac{1}{n^4}\sum_{i,j,k,l=1}^{n}|x_l-x_k||y_j-y_i|\Bigg)\\
	&=&\frac{1}{4}\Bigg(\frac{1}{n^2}\sum_{i, j=1}^{n}|x_i-x_j||y_i-y_j|-\frac{2}{n^3}\sum_{i,j,k=1}^{n}|x_i-x_j||y_i-y_k|+\frac{1}{n^4}\sum_{i,j,k,l=1}^{n}|x_i-x_j||y_k-y_l|\Bigg)\\
	&=&\frac{1}{4}\big[V_{12n}-2V_{3n}+V_{1n}V_{2n}\big]
\end{eqnarray*}
It is interesting to observe that $\hat\kappa$ and $\kappa^{*}$ have the same form, but the former involves $V$-statistics while the latter involves $U$-statistics. 
%	$$\hat{\kappa}=\frac{1}{4}[V_{12n}-2V_{3n}+V_{1n}V_{2n}].$$
Expressing the above $V$-statistics in terms of corresponding $U$-statistics we obtain 
\begin{eqnarray*}
	\hat{\kappa}&=&\frac{1}{4}\big[V_{12n}-2V_{3n}+V_{1n}V_{2n}\big]\\
	&=&\frac{1}{4}\bigg[\frac{(n-1)}{n}U_{12n}-2\frac{(n-1)(n-2)}{n^2}U_{3n}-\frac{2(n-1)}{n^2}U_{12n}
	+\big(\frac{n-1}{n}\big)^2U_{1n}U_{2n}\bigg]\\
	&=&\frac{1}{4}\bigg[\frac{(n-1)(n-2)}{n^2}U_{12n}-2\frac{(n-1)(n-2)}{n^2}U_{3n}+\bigg(\frac{n-1}{n}\bigg)^2U_{1n}U_{2n}\bigg]\\
	&=&\frac{1}{4}\bigg[\bigg(\frac{(n-1)(n-2)}{n^2}+1-1\bigg)U_{12n}-2\bigg(\frac{(n-1)(n-2)}{n^2}+1-1\bigg)U_{3n}+\bigg(\bigg(\frac{n-1}{n}\bigg)^2+1-1\bigg)U_{1n}U_{2n}\bigg]\\
%	&=&\frac{1}{4}\big[U_{12n}-2U_{3n}+U_{1n}U_{2n}\big]+\frac{1}{4}\bigg[\bigg(\frac{(n-1)(n-2)}{n^2}-1\bigg)U_{12n}\\ && \ \ \ \ \ \ \ \ \ \ \ \ \ \ \  \-2\bigg(\frac{(n-1)(n-2)}{n^2}-1\bigg)U_{3n}+\bigg(\bigg(\frac{n-1}{n}\bigg)^2-1\bigg)U_{1n}U_{2n}\bigg]\\
%	&=&\frac{1}{4}\big[U_{12n}-2U_{3n}+U_{1n}U_{2n}\big]+\frac{1}{4}\big[\frac{(2-3n)}{n^2}U_{12n}-2\big(\frac{2-3n}{n^2}\big)U_{3n}+\big(\frac{1-2n}{n^2}\big)U_{1n}U_{2n}\big]\\
	&=&\kappa^{*}+\frac{1}{4}\big[\frac{(2-3n)}{n^2}U_{12n}-2\big(\frac{2-3n}{n^2}\big)U_{3n}+\big(\frac{1-2n}{n^2}\big)U_{1n}U_{2n}\big]
\end{eqnarray*}
\end{proof}
%		\begin{theorem}\label{theo:kappacurlnormal} Suppose $\{(X_i, Y_i)\}$ are i.i.d with $\BE_{F_{12}} (X_1^2Y_1^2) < \infty$. Then  
%			$n^{1/2}\big(\tilde{\kappa}-\kappa\big)$ is asymptotically normal with mean zero and variance $\delta_1$ where,
%			$\delta_1=\frac{1}{4}Var(H_1(X_1,Y_1))$.
%		\end{theorem}
%
%It is interesting to observe that $\hat\kappa$ and $\kappa^{*}$ have the same form, but the former involves $V$-statistics while the latter involves $U$-statistics. 

\section{Asymptotic behavior of $\hat \kappa$, $\tilde \kappa$ and $\kappa^{*}$}\label{sec:asymp}
%the three estimates of $\kappa$}

As mentioned earlier, for the dependent case, the asymptotic distributions of $\hat{\kappa}$, and $\tilde\kappa$ are not available in literature. This is presented below. Also for the independence case we will provide an alternate proof for the asymptotic results. Incidentally for the independence case (i.e. the degenerate case) we observed some errors in the calculation of the eigensystem of the kernel $h_F$, in \cite{ bergsma2006new}. For a rectified version please see Appendix \ref{sec:ev} %\textit{ Appendix A.3}.  

\subsection{Dependent case: asymptotic normality of $\kappa^*$, $\hat{\kappa}$, and $\tilde\kappa$}\label{subsec:asynormal}  

The asymptotic normality of the estimates now follow easily.
%$\kappa^{*}$ and $\hat\kappa$ now follow easily. 
%	$n^{1/2}\big(\kappa^{*}-\kappa)$ can be proved by using Theorem \ref{UCLT}.	

\begin{theorem}\label{theo:kappanormal} Suppose $\{(X_i, Y_i)\}$ are i.i.d.~with $\BE_{F_{12}}[X_1^2Y_1^2] < \infty$. Then  as $n \to \infty$, 
$n^{1/2}\big(\kappa^{*}-\kappa\big)$, $n^{1/2}\big(\tilde{\kappa}-\kappa\big)$, and $n^{1/2}\big(\hat{\kappa}-\kappa\big)$ are asymptotically normal with mean $0$ and variance $\delta_1$ where,
\begin{eqnarray}
	\delta_1&=&\frac{1}{4}\big[\BVar_{F_{12}}\big(g_{F_{12}}(X_1,Y_1)+\mu_{1}g_{F_{2}}(Y_1)+\mu_{2}g_{F_{1}}(X_1) - \BE_{F_{12}}[|X_2-X_1|\ |Y_2-Y_3| \  |X_1]\nonumber\\
	&& \ -\BE_{F_{12}}[|X_2-X_3|\ |Y_2-Y_1|\ |Y_1]-g_{F_{1}}(X_1)g_{F_{2}}(Y_1)\big)\big]. \label{asymvar}
\end{eqnarray}
%\textbf{check the last bracket. should there be one bracket after $1/4$?}
%\item Under the same assumptions, 
%$$n^{1/2}\big(\tilde{\kappa}-\kappa\big)\xrightarrow{D}N(0,\delta_1).$$
%when $X$ and $Y$ are possibly dependent random variables.
\end{theorem}

\begin{remark}
When $\{X_i\}$ is independent of $\{Y_i\}$, $F_{12}=F_1\otimes F_2$, $g_{F_{12}}(X_1, Y_1)=g_{F_{1}}(X_1) g_{F_{2}}(Y_1)$, and the conditional expectations are $\mu_{2}g_{F_{1}}(X_1)$ and $\mu_{1}g_{F_{2}}(Y_1)$. Hence $\delta_1=0$. In the next section we investigate the non-degenerate asymptotic distribution of $n\kappa^{*}$ under independence. 
\end{remark}

\begin{proof}[Proof of Theorem \ref{theo:kappanormal}] 
First note that 
\begin{eqnarray} 
	4\big(\kappa^{*}-\kappa\big)&=& \big[U_{12n}-\mu_{12}
	+\mu_1(U_{2n}-\mu_2)
	+ \mu_2(U_{1n}-\mu_1)
	-2(U_{3n}-\mu_{3})\big]\label{eq:kappau}\\
	&&\ \ +(U_{1n}-\mu_1)(U_{2n}-\mu_2)\nonumber\\
	&=&U_n+(U_{1n}-\mu_1)(U_{2n}-\mu_2)\ \text(say),\label{eq:kappauv}
\end{eqnarray}

By the first projections we have,
\begin{eqnarray*}
	U_{1n}-\mu_1&=&2n^{-1} \sum_{i=1}^n \big[g_{F_{1}}(X_i)-\mu_1\big]+R_{1n}= 2\bar g_1+R_{1n} \ \text{say}, \\
	U_{2n}-\mu_2&=&2n^{-1} \sum_{i=1}^n \big[g_{F_{2}}(Y_i)-\mu_2\big]+R_{2n}=2\bar g_2+R_{2n}\ \text{say},  \\
	U_{12n}-\mu_{12}
	&=&2n^{-1} \sum_{i=1}^n \big[g_{F_{12}}(X_i, Y_i)-\mu_{12}\big]+R_{12n}=2\bar g_{12}+R_{12n} \ \text{say},\\
	U_{3n}-\mu_3&=&3 n^{-1}\sum_{i=1}^{n}\big(\dfrac{1}{6}\big[2g_{F_{1}}(x_i)g_{F_{2}}(y_i)
	+2\BE_{F_{12}}[|X_2-X_1|\ |Y_2-Y_3|\ |X_1=x_i]\\&&
	\hspace{2cm}+2\BE_{F_{12}}[|X_2-X_3|\ |Y_2-Y_1| \  |Y_1=y_i]\big]- \mu_{3}\big)+R_{3n}, \text{say},
\end{eqnarray*}
where $n^{1/2}R_{12n}\xrightarrow{P} 0$ and $n^{1/2} R_{in}\xrightarrow{P} 0$, for $i=1, 2, 3$. Note that $U_n$ is a linear combination of 
$U_{1n}$, $U_{2n}$, $U_{12n}$ and $U_{3n}$ and is a $U$-statistics of order $3$. We write the $H$-decomposition of $U_n$ with overall centered first projection of $U_n$ given by $H_1(x_1,y_1)$ as follows:

\begin{eqnarray}
	U_n&=&2n^{-1} \sum_{i=1}^n \big[g_{F_{12}}(X_i, Y_i)-\mu_{12}\big]+\mu_{1}(2n^{-1} \sum_{i=1}^n \big[g_{F_{2}}(Y_i)-\mu_2\big])+\mu_{2}(2n^{-1} \sum_{i=1}^n \big[g_{F_{1}}(X_i)-\mu_1\big])\nonumber\\&& \ -2(3 n^{-1}\sum_{i=1}^n \big(\dfrac{1}{6}\big[2g_{F_{1}}(x_i)g_{F_{2}}(y_i)
	+2\BE_{F_{12}}[|X_2-X_1|\ |Y_2-Y_3| \  |X_1=x_i]\\&&
	+2\BE_{F_{12}}[|X_2-X_3|\ |Y_2-Y_1| \  |Y_1=y_i]\big]- \mu_{3}\big)+R_n,\nonumber\\
	&=&2n^{-1} \sum_{i=1}^n \big[g_{F_{12}}(X_i, Y_i)-\mu_{12}+\mu_{1}(g_{F_{2}}(Y_i)-\mu_2)+\mu_{2}(g_{F_{1}}(X_i)-\mu_1) -(g_{F_{1}}(X_i)g_{F_{2}}(Y_i)) 
	\nonumber \\
	&& \ - \BE_{F_{12}}[|X_2-X_1|\ |Y_2-Y_3| \  |X_1=X_i]-\BE_{F_{12}}[|X_2-X_3|\ |Y_2-Y_1| \  |Y_1=Y_i]+3\mu_{3}\big]+R_n \nonumber\\
	&=&2n^{-1} \sum_{i=1}^nH_1(X_i,Y_i)+R_n, \label{eq:unvn}
\end{eqnarray}
where $R_n=o_P(n^{-1/2})$ and 

\begin{eqnarray*}
		H_1(X_1,Y_1)&=&\big[g_{F_{12}}(X_1, Y_1)-\mu_{12}+\mu_{1}(g_{F_{2}}(Y_1)-\mu_2)+\mu_{2}(g_{F_{1}}(X_1)-\mu_1)-(g_{F_{1}}(X_1)g_{F_{2}}(Y_1))\\
		&& \ - \BE_{F_{12}}[|X_2-X_1|\ |Y_2-Y_3| \  |X_1]-\BE_{F_{12}}[|X_2-X_3|\ |Y_2-Y_1| \  |Y_1]+3\mu_{3}\big].
	\end{eqnarray*}

Now note that 
	\begin{eqnarray}\sqrt{n}(U_{1n}-\mu_1)(U_{2n}-\mu_2)&=&\sqrt{n}(2\bar g_1+ R_{1n}) (2\bar g_2+R_{2n})\nonumber\\
		&=&4n^{1/2}\bar g_1\bar g_2+\epsilon_n.\label{eq:u1u2prod}
	\end{eqnarray}

By the central limit theorem, $n^{1/2}\big(\bar g_{12}, \overline{g_1g_2}, \bar g_1, \bar g_2\big)$ is asymptotically normal with mean $0$. 
Therefore, \ $n^{1/2} \bar g_1 \bar g_2 \xrightarrow{P} 0$ and  $\sqrt{n}\epsilon_n\xrightarrow P 0$.
Hence using (\ref{eq:unvn}) and (\ref{eq:kappauv}), we have 

	\begin{equation}\label{kappahathdecom}
		\sqrt{n}\big(\kappa^{*}-\kappa\big)=\frac{1}{2}n^{-1/2} \sum_{i=1}^nH_1(X_i,Y_i)+\sqrt{n}R_n+n^{1/2}\bar g_1\bar g_2+\sqrt{n}\epsilon_n.
	\end{equation}

As a consequence, 
	\begin{equation}\label{eq:kappanorm} 
		n^{1/2}\big(\kappa^{*}-\kappa\big)\xrightarrow{D}N(0,\delta_1),
	\end{equation}
where 
	%\begin{eqnarray*}
	$\delta_1=\frac{1}{4}\BVar(H_1(X_1,Y_1))$ is as given in (\ref{asymvar}).
	%\\
	%&=&\frac{1}{4}Var(g_{F_{12}}(X_1,Y_1)+\mu_{1}g_{F_{2}}(Y_1)+\mu_{2}g_{F_{1}}(X_1) - \BE[|X_2-X_1|\ |Y_2-Y_3||X_1]\\
	%&& \ -\BE[|X_2-X_3|\ |Y_2-Y_1||Y_1]-g_{F_{1}}(X_1)g_{F_{2}}(Y_1)).
	%\end{eqnarray*}
	
The claim for $\tilde\kappa$ and $\hat{\kappa}$ now follows easily from 
	% (\ref{kappakappa*}) given in 
	Lemma \ref{kappacurlstar}.
	%and \ref{kappahatstar} respectively. 
	We omit the details.  
	%Part $2$:	We have 
	%\begin{eqnarray*}
	%\tilde{\kappa}&=&\kappa^{*}+\frac{1}{4}\bigg[\frac{-2n}{(n-1)^2}U_{12n}+\bigg(\frac{2}{(n-1)^2}\bigg)U_{3n}+\bigg(\frac{2}{n-1}\bigg)U_{1n}U_{2n}\bigg],
	%\end{eqnarray*}
	%\begin{eqnarray*}
	%\sqrt{n}\big(\tilde{\kappa}-\kappa\big)=\sqrt{n}\big(\kappa^{*}-\kappa\big)+\frac{\sqrt{n}}{4}\bigg[\frac{-2n}{(n-1)^2}U_{12n}+\bigg(\frac{2}{(n-1)^2}\bigg)U_{3n}+\bigg(\frac{2}{n-1}\bigg)U_{1n}U_{2n}\bigg].
	%\end{eqnarray*}
	%By SLLN for $U$-statistics $(U_{12n},U_{3n},U_{1n},U_{2n})\xrightarrow{a.s} (\mu_{12},\mu_{3},\mu_{1},\mu_{2})$ which are all assumed to be finite. Also we have as $n\rightarrow \infty$, each of the terms $$\sqrt{n}\frac{-2n}{(n-1)^2},\sqrt{n}\frac{2}{(n-1)^2},\sqrt{n}\frac{2}{n-1}$$ 
	%converges to $0$.
	%Therefore by \eqref{eq:kappanorm}, 
	%$$n^{1/2}\big(\tilde{\kappa}-\kappa\big)\xrightarrow{D}N(0,\delta_1).$$
\end{proof}

%Further under independence by \eqref{ind} we have,
%\begin{eqnarray}
%H_1(X_1,Y_1)&=&\big[(g_{F_{1}}(X_1)g_{F_{2}}(Y_1))-\mu_{1}\mu_{2}+\mu_{1}(g_{F_{2}}(Y_1)-\mu_2)+\mu_{2}(g_{F_{1}}(X_1)-\mu_1)\nonumber\\
%&& \ - (g_{F_{1}}(X_1)g_{F_{2}}(Y_1)) - \mu_{2}(g_{F_{1}}(X_1))-\mu_{1}(g_{F_{2}}(Y_1))+3\mu_{1}\mu_{2}\big]\nonumber\\ &=&0,\label{eq:firstprojsimpl}
%\end{eqnarray}
%the limit is degenerate.
%\begin{remark}
%Note that \eqref{eq:kappanorm} involves the asymptotic variance $\delta_1>0$ which is unknown. In many situations estimation of the asymptotic variance is difficult. This is a problem as in general $U$-statistics theory.  \cite{bose2018u} gives application of various resampling techniques in estimation of the asymptotic variance. 
%\end{remark}

\subsection{Independent (degenerate) case} 
%\textbf{Calculation has mistakes probably. Divya has done the details and there seems to be an extra term. June 21}. 

We now present the asymptotic distribution of $\kappa^*$ in the independent case. 
%We shall now use the Theorem \ref{theo:standard_degen_u} to prove the following:

\begin{theorem}\label{kappastarind}
Suppose the observations $\{X_i\}$ and $\{Y_i\}$ are all independent with distributions $F_1$ and $F_2$ respectively. Suppose $\BE_{F_{1}}[ h^2_{F_1}(X_1, X_2)]+ 
	\BE_{F_{2}}[ h^2_{F_2}(Y_1, Y_2)]
	< \infty$.
	%		$h_{F_1}$ and $h_{F_2}$ are square integrable. 
	Then 
	%\begin{eqnarray}
	$n\kappa^{*}
	%=	\frac{1}{4}n(U_n + V_n)
	%	\simeq n {n \choose 2}^{-1}\sum_{1\leq i <j \leq n}h_{F_1}(x_i, x_j)h_{F_2}(y_i, y_j)+R_n \nonumber \\
	\xrightarrow D \sum_{k,l=0}^{\infty}\lambda_k\eta_l(Z_{kl}^2-1)$, 
	%\end{eqnarray}
	where %
	$\{\lambda_i\}$ and $\{\eta_i\}$ are the eigenvalues of $h_{F_1}$ and $h_{F_2}$ and 
	%		and $\{Y_i\}$ are all independent and independent, 
	%		$h_{F_1}$ and $h_{F_2}$ are square integrable such that 
	%\begin{equation*}
	%h_{F_1}(x_1, x_2)=\sum_{k=0}^{\infty}\lambda_kg_{1k}(x_1)g_{1k}(x_2)
	%\end{equation*}
	%and
	%\begin{equation*}
	%h_{F_2}(y_1,y_2)=\sum_{k=0}^{\infty}\eta_k g_{2k}(y_1)g_{2k}(y_2).
	%\end{equation*}
	\{$Z_{ij}$\} are i.i.d.~standard normal variables.
	%	\textbf{I suspect $\sum\lambda_i=\sum\eta_j=0$. Please check}
	%, we obtain,
\end{theorem}

\begin{proof}
	Note that now $\kappa=0$. Recall from Equation (\ref{eq:kappauv}) that 
	$$4\kappa^*=U_n+V_n, \ \text{where}\ \ V_n=(U_{1n}-\mu_1)(U_{2n}-\mu_2).$$
	%Now, since 
Under independence, 
	%we have 
	the first projection of $U_n$ is $0$. Hence under independence $n^{1/2}\hat \kappa$ converges to $0$ in probability.
	In this case we have to scale up and consider $n\hat \kappa$. 
	Consider the second projection of $U_n$ say $H_2((x_i,y_i),(x_j,y_j))$ is given by combinations of the second projections of $U_{1n},U_{2n}, U_{12n}$ and $U_{3n}$
so that 	$$U_n={n \choose 2}^{-1}\sum_{1\leq i <j \leq n}H_2\big( (X_i, Y_i), (X_j, Y_j)\big)+\epsilon_n,$$
where $n\epsilon_n\xrightarrow{P} 0$.
	By using the second projections of $U_{1n},U_{2n}, U_{12n}$ and $U_{3n}$, we can write $U_n$ as follows
 \begin{eqnarray*}
 	U_n&=&\Big[U_{12n}-\mu_{12}+\mu_1(U_{2n}-\mu_2)+ \mu_2(U_{1n}-\mu_1)-2(U_{3n}-\mu_{3})\Big]\\
 	&=&{n \choose 2}^{-1}\sum_{1\leq i < j \leq 
 		n}[|x_i-x_j||y_i-y_j|-\mu_{1}\mu_{2}]-[g_{F_{1}}(x_i)g_{F_{2}}(y_i)-\mu_{1}\mu_{2}]-[g_{F_{1}}(x_j)g_{F_{2}}(y_j)-\mu_{1}\mu_{2}]\\
 	&&\ + \mu_{1}\bigg({n \choose 2}^{-1}\sum_{1\leq i < j \leq n}[|y_i-y_j|-\mu_2]-(g_{F_2}(y_i)-\mu_2)-(g_{F_2}(y_j)-\mu_2)\bigg)\\
 		&& \ + \mu_{2}\bigg({n \choose 2}^{-1}\sum_{1\leq i < j \leq n}[|x_i-x_j|-\mu_1]-(g_{F_1}(x_i)-\mu_1)-(g_{F_1}(x_j)-\mu_1)\bigg)\\
 	&& \  \ - 2{3 \choose 2}{n \choose 2}^{-1}\sum_{1\leq i < j \leq 
 		n}\bigg(\dfrac{1}{6}\big[
 	|x_i-x_j|\big( g_{F_{2}}(y_i)+ g_{F_{2}}(y_j)\big)-2\mu_{1}\mu_{2}+ \big(g_{F_{1}}(x_i)+g_{F_{1}}(x_j)\big)\\
 	&& \  \ \ \ \ \ \ \ \ \ \ \ \  |y_i-y_j|-2\mu_{1}\mu_{2}+g_{F_{1}}(x_i)\ g_{F_{2}}(y_j)-\mu_{1}\mu_{2}
 	+g_{F_{1}}(x_j)\ g_{F_{2}}(y_i)-\mu_{1}\mu_{2}\big]\\
 	&& \ \ \ \ \ \ \  \ -\dfrac{1}{6}\big[2\big(g_{F_{1}}(x_i)g_{F_{2}}(y_i)-\mu_{1}\mu_{2}) 
 	+2 \mu_2\big(g_{F_{1}}(x_i)-\mu_{1}\big)
 	+2 \mu_1\big(g_{F_{2}}(y_i)-\mu_{2}\big)\big] \\
 	&&  \ \ \ \ \ \ \ \ -\dfrac{1}{6}\big[2\big(g_{F_{1}}(x_j)g_{F_{2}}(y_j)-\mu_{1}\mu_{2}) 
 	+2 \mu_2\big(g_{F_{1}}(x_j)-\mu_{1}\big)
 	+2 \mu_1\big(g_{F_{2}}(y_j)-\mu_{2}\big)\big]\bigg)+\epsilon_n
 \end{eqnarray*}
% \begin{eqnarray*}
% 	U_n&=&{n \choose 2}^{-1}\sum_{1\leq i < j \leq 
% 		n}\bigg[|x_i-x_j||y_i-y_j|-\mu_{1}\mu_{2}-[g_{F_{1}}(x_i)g_{F_{2}}(y_i)-\mu_{1}\mu_{2}]-[g_{F_{1}}(x_j)g_{F_{2}}(y_j)-\mu_{1}\mu_{2}]\\
% 	&& \ \ \ \ \ \ \ \ \ \ + \mu_{2}\bigg([|x_i-x_j|-\mu_1]-(g_{F_1}(x_i)-\mu_1)-(g_{F_1}(x_j)-\mu_1)\bigg)\\
% 	&&\ \ \ \ \ \ \ \ \ \  + \mu_{1}\bigg([|y_i-y_j|-\mu_2]-(g_{F_2}(y_i)-\mu_2)-(g_{F_2}(y_j)-\mu_2)\bigg)\\
% 	&& \  \ \ \ \ \ \ \ \ \ \ - |x_i-x_j|\big( g_{F_{2}}(y_i)+ g_{F_{2}}(y_j)\big)+2\mu_{1}\mu_{2} - |y_i-y_j| \big(g_{F_{1}}(x_i)+g_{F_{1}}(x_j)\big) +2\mu_{1}\mu_{2}\\
% 	&&\ \ \ \ \ \ \ \ \ \ \ \ \ -g_{F_{1}}(x_i)\ g_{F_{2}}(y_j)+\mu_{1}\mu_{2}
% 	-g_{F_{1}}(x_j)\ g_{F_{2}}(y_i)+\mu_{1}\mu_{2}\\
% 	&& \ \ \ \ \ \ \  \ +\big[2\big(g_{F_{1}}(x_i)g_{F_{2}}(y_i)-\mu_{1}\mu_{2}) 
% 	+2 \mu_2\big(g_{F_{1}}(x_i)-\mu_{1}\big)\big]
% 	+2 \mu_1\big(g_{F_{2}}(y_i)-\mu_{2}\big) \\
% 	&&  \ \ \ \ \ \ \ \ +\big[2\big(g_{F_{1}}(x_j)g_{F_{2}}(y_j)-\mu_{1}\mu_{2}) 
% 	+2 \mu_2\big(g_{F_{1}}(x_j)-\mu_{1}\big)\big]
% 	+2 \mu_1\big(g_{F_{2}}(y_j)-\mu_{2}\big)\bigg]+\epsilon_n
% \end{eqnarray*}
 \begin{eqnarray*}
 	U_n&=&{n \choose 2}^{-1}\sum_{1\leq i < j \leq 
 		n}\bigg[|x_i-x_j||y_i-y_j|-\mu_{1}\mu_{2}+[g_{F_{1}}(x_i)g_{F_{2}}(y_i)-\mu_{1}\mu_{2}]+[g_{F_{1}}(x_j)g_{F_{2}}(y_j)-\mu_{1}\mu_{2}]\\
 	&&\ \ \ \ \ \ \ \ \ \  + \mu_{1}[|y_i-y_j|-\mu_2])+\mu_{1}(g_{F_2}(y_i)-\mu_2)+\mu_{1}(g_{F_2}(y_j)-\mu_2)\\
 	&& \ \ \ \ \ \ \ \ \ \ + \mu_{2}[|x_i-x_j|-\mu_1])+\mu_{2}(g_{F_1}(x_i)-\mu_1)+\mu_{2}(g_{F_1}(x_j)-\mu_1)\\
 	&& \  \ \ \ \ \ \ \ \ \ \ - |x_i-x_j|\big( g_{F_{2}}(y_i)+ g_{F_{2}}(y_j)\big)+2\mu_{1}\mu_{2}- |y_i-y_j| \big(g_{F_{1}}(x_i)+g_{F_{1}}(x_j)\big) +2\mu_{1}\mu_{2}\\
 	&&\ \ \ \ \ \ \ \ \ \ \ \ \ -g_{F_{1}}(x_i)\ g_{F_{2}}(y_j)+\mu_{1}\mu_{2}
 	-g_{F_{1}}(x_j)\ g_{F_{2}}(y_i)+\mu_{1}\mu_{2}\bigg]+\epsilon_n
 \end{eqnarray*}
 \begin{eqnarray*}
 	U_n&=&{n \choose 2}^{-1}\sum_{1\leq i < j \leq 
 		n}\bigg[|x_i-x_j||y_i-y_j|+[g_{F_{1}}(x_i)g_{F_{2}}(y_i)]+[g_{F_{1}}(x_j)g_{F_{2}}(y_j)]\\
 	&&\ \ \ \ \ \ \ \ \ \  + \mu_{1}[|y_i-y_j|])+(2\mu_{1}(g_{F_2}(y_i))-\mu_{1}(g_{F_2}(y_i)))+(2\mu_{1}(g_{F_2}(y_j)-\mu_{1}(g_{F_2}(y_j)))\\
 	&& \ \ \ \ \ \ \ \ \ \ + \mu_{2}[|x_i-x_j])+(2\mu_{2}(g_{F_1}(x_i)-\mu_{2}(g_{F_1}(x_i))+(2\mu_{2}g_{F_1}(x_j)-\mu_{2}g_{F_1}(x_j))\\
 	&& \  \ \ \ \ \ \ \ \ \ \ - |x_i-x_j|\big( g_{F_{2}}(y_i)+ g_{F_{2}}(y_j)\big)- |y_i-y_j| \big(g_{F_{1}}(x_i)+g_{F_{1}}(x_j)\big)(-2g_{F_{1}}(x_i)\ g_{F_{2}}(y_j)\\&&\ \ \ \ \ \ \ \ +g_{F_{1}}(x_i)\ g_{F_{2}}(y_j))
 	+(-2g_{F_{1}}(x_j)\ g_{F_{2}}(y_i)+g_{F_{1}}(x_j)\ g_{F_{2}}(y_i))-3\mu_{1}\mu_{2}\bigg]+\epsilon_n
 \end{eqnarray*}
 \begin{eqnarray*}
 	U_n&=&{n \choose 2}^{-1}\sum_{1\leq i < j \leq 
 		n}\big[|x_i-x_j|- g_{F_{1}}(x_i)- g_{F_{1}}(x_j)+\mu_{1}\big] \times \big[|y_i-y_j|-g_{F_{2}}(y_i)- g_{F_{2}}(y_j)+\mu_{2}\big]\nonumber\\
 	&& \ \ \ \ \ \ \ \ \ \ \ \ \  -2 \big(g_{F_{1}}(x_j)g_{F_{2}}(y_i)-\mu_{2}g_{F_{1}}(x_j)-\mu_{1}g_{F_{2}}(y_i)+\mu_{1}\mu_{2}\big)\\
 	&& \ \ \ \ \ \ \ \ \ \ \ \ \  -2 \big(g_{F_{1}}(x_i)g_{F_{2}}(y_j)-\mu_{2}g_{F_{1}}(x_i)-\mu_{1}g_{F_{2}}(y_j)+\mu_{1}\mu_{2}\big)+\epsilon_n.\\ &=&{n \choose 2}^{-1}\sum_{1\leq i < j \leq 
 		n}\big[|x_i-x_j|- g_{F_{1}}(x_i)- g_{F_{1}}(x_j)+\mu_{1}] \times \big[|y_i-y_j|-g_{F_{2}}(y_i)- g_{F_{2}}(y_j)+\mu_{2}\big]\nonumber\\
 	&& \ \ \ \ \ \ \ \ \ \ \ \ \ -2\big((g_{F_{1}}(x_i)-\mu_{1})\ (g_{F_{2}}(y_j)-\mu_{2})+(g_{F_{1}}(x_j)-\mu_{1})\ (g_{F_{2}}(y_i)-\mu_{2})\big)+\epsilon_n.
 \end{eqnarray*}
Therefore,
	\begin{eqnarray}
		H_2((x_i,y_i),(x_j,y_j))&=&\big[|x_i-x_j|- g_{F_{1}}(x_i)- g_{F_{1}}(x_j)+\mu_{1}] \nonumber\\ && \ \ \ \ \ \ \ \ \ \ \ \ \times   \big[|y_i-y_j|-g_{F_{2}}(y_i)- g_{F_{2}}(y_j)+\mu_{2}\big]\nonumber\\
		&& \ \  -2\big((g_{F_{1}}(x_i)-\mu_{1})\ (g_{F_{2}}(y_j)-\mu_{2})+(g_{F_{1}}(x_j)-\mu_{1})\ (g_{F_{2}}(y_i)-\mu_{2})\big).\nonumber\label{secproj}
	\end{eqnarray}
	%we consider the 
	%$H_2((x_i,y_i),(x_j,y_j))$ be its second projection so that
	%given by combinations of the second projections of $U_{1n},U_{2n}, U_{12n}$ and $U_{3n}$

	%in probability.
	%		We have $$\BE[H_2((x_i,y_i),(X_j,Y_j))]=0.$$
	Note that, 
	%\begin{eqnarray}\label{eq:kappadegen}n\ 4 \kappa^{*} &=&nU_n+nV_n\nonumber\\
	%&=&nU_n+ n(U_{1n}-\mu_1)(U_{2n}-\mu_2).
	%\end{eqnarray}
\begin{eqnarray}
		%n(U_{1n}-\mu_1)(U_{2n}-\mu_2)
		nV_n&=&(n^{1/2}2\bar g_1+ n^{1/2}R_{1n}) (n^{1/2}2\bar g_2+n^{1/2}R_{2n})\nonumber\\
		&=&4n^{1/2}\bar g_1n^{1/2}\bar g_2+\delta_n,\label{eq:u1u2prod}
	\end{eqnarray}
where $\delta_n\xrightarrow{P} 0$. Considering the leading terms from $U_n$ and $V_n$ in the further calculations since $\delta_n\to 0$ and $n\epsilon_n\to 0$  we write
\begin{eqnarray}
		\frac{n}{4}(U_n + V_n) &\simeq& \frac{n}{4} {n \choose 2}^{-1}\sum_{1\leq i <j \leq n}H_2 \big((x_i, y_i), (x_j, y_j)\big) 
		%+\epsilon_n
		+ n \bar{g_1}\bar{g_2}\\
		%+\delta_n\nonumber\\
		%		&\simeq&n{n \choose 2}^{-1}\sum_{1\leq i <j \leq n}\frac{1}{4} H_2 \big((x_i, y_i), (x_j, y_j)\big) + n \bar{g_1}\bar{g_2}\nonumber\\\nonumber
		\\&=&n{n \choose 2}^{-1}\sum_{1\leq i <j \leq n}\big[\frac{1}{4}	\big[|x_i-x_j|- g_{F_{1}}(x_i)- g_{F_{1}}(x_j)+\mu_{1}] \nonumber\\ && \ \ \ \ \ \ \ \ \ \ \ \ \times   \big[|y_i-y_j|-g_{F_{2}}(y_i)- g_{F_{2}}(y_j)+\mu_{2}\big]\nonumber\\
		&& \ \  -2\big((g_{F_{1}}(x_i)-\mu_{1})\ (g_{F_{2}}(y_j)-\mu_{2})+(g_{F_{1}}(x_j)-\mu_{1})\ (g_{F_{2}}(y_i)-\mu_{2})\big)\big]\nonumber\	
		%+R_n
		+n \bar{g_1}\bar{g_2}\nonumber  \\
		&=&n {n \choose 2}^{-1}\sum_{1\leq i <j \leq n}h_{F_1}(x_i, x_j)h_{F_2}(y_i, y_j)+n \bar{g_1}\bar{g_2}\\
		%+R_n\nonumber\\
		&& \  \ \ \ \ \ \ -\frac{2n}{4}{n \choose 2}^{-1}\big[\sum_{1\leq i <j \leq n}(g_{F_{1}}(x_i)-\mu_{1})\ (g_{F_{2}}(y_j)-\mu_{2})+(g_{F_{1}}(x_j)-\mu_{1})\ (g_{F_{2}}(y_i)-\mu_{2})\big]\nonumber\\
		%\end{eqnarray}
		%\begin{eqnarray*}
		%\frac{n}{4}(U_n + V_n) 
		&=& n {n \choose 2}^{-1}\sum_{1\leq i <j \leq n}h_{F_1}(x_i, x_j)h_{F_2}(y_i, y_j)+n \bar{g_1}\bar{g_2}
		%+R_n\nonumber
		\\
		&& \  \ \ \ \ \ \ -\big[n\bar{g_1}\bar{g_2}-\frac{1}{n}\sum_{i=1}^{n}\big((g_{F_{1}}(x_i)-\mu_{1})\ (g_{F_{2}}(y_i)-\mu_{2})\big)\big]\nonumber\\
		&=& n {n \choose 2}^{-1}\sum_{1\leq i <j \leq n}h_{F_1}(x_i, x_j)h_{F_2}(y_i, y_j)+\frac{1}{n}\sum_{i=1}^{n}\big((g_{F_{1}}(x_i)-\mu_{1})\ (g_{F_{2}}(y_i)-\mu_{2})\big).\label{HdecomplusRn}
		%+R_n.
	\end{eqnarray}
	%Therefore we have,
	%\begin{eqnarray}
	%n\kappa^*=\frac{n}{4}(U_n + V_n)&=&n {n \choose 2}^{-1}\sum_{1\leq i <j \leq n}h_{F_1}(x_i, x_j)h_{F_2}(y_i, y_j)\nonumber\\&&\ \ \ \ \ \ \ \ +\frac{1}{n}\sum_{i=1}^{n}\big((g_{F_{1}}(x_i)-\mu_{1})\ (g_{F_{2}}(y_i)-\mu_{2})\big)+R_n.\label{HdecomplusRn}
	%\end{eqnarray}
	%As $n \rightarrow \infty$, $R_n \xrightarrow P 0 $. 
	%Hence 
	%$$\frac{1}{n}\sum_{i=1}^{n}\big((g_{F_{1}}(x_i)-\mu_{1})\ (g_{F_{2}}(y_i)-\mu_{2})\big) \xrightarrow{P} 0.$$
	%			Let us consider the remaining term
By the strong law of large numbers, the second term goes to $0$ almost surely.		
	%		$${n \choose 2}^{-1}\sum_{1\leq i <j \leq n}h_{F_1}(x_i, x_j)h_{F_2}(y_i, y_j)$$
	%is a dengenerate $U$-statistic of order $2$. 
	By Theorem \ref{theo:standard_degen_u}
	$$ n {n \choose 2}^{-1}\sum_{1\leq i <j \leq n}h_{F_1}(x_i, x_j)h_{F_2}(y_i, y_j)\xrightarrow D \sum_{i,j=0}^{\infty}\lambda_i\eta_j(Z_{ij}^2-1),$$
where, $\{\lambda_i\}$,$\{\eta_j\}$ are the eigenvalues corresponding to spectral decomposition of $h_{F_1}$ and $h_{F_2}$ respectively and $Z_{ij}$ are i.i.d.~standard normal variables.
	This completes the proof.
	%Therefore
	%As $n \rightarrow \infty$, 
	%\begin{eqnarray}
	%n\kappa^{*}=	\frac{1}{4}n(U_n + V_n) \nonumber &\xrightarrow D& \sum_{i,j=0}^{\infty}\lambda_i\eta_j(Z_{ij}^2-1).
	%\end{eqnarray}
\end{proof}

%		\begin{theorem}
	%			Suppose $h_{F_1}$ and $h_{F_2}$ are square integrable with spectral decompositions 
	%			\begin{equation*}
		%				h_{F_1}(x_1, x_2)=\sum_{k=0}^{\infty}\lambda_kg_{1k}(x_1)g_{1k}(x_2)
		%			\end{equation*}
	%			and
	%			\begin{equation*}
		%				h_{F_2}(y_1,y_2)=\sum_{k=0}^{\infty}\eta_k g_{2k}(y_1)g_{2k}(y_2),
		%			\end{equation*}
	%			then if $X$ and $Y$ are independent and \{$Z_{ij}$\} are i.i.d. standard normal variables, 
	%			%	then with $R_n \xrightarrow P 0 $, 
	%			\begin{eqnarray}
		%				n\tilde{\kappa}\xrightarrow D \sum_{i,j=0}^{\infty}\lambda_i\eta_j(Z_{ij}^2-1)
		%			\end{eqnarray}
	%		\end{theorem}

\begin{remark}\label{kappatildeindproof}
	%Further note that under independence 
	(i) The asymptotic distribution result for $\tilde{\kappa}$ in Theorem \ref{kappatildadegen} follows from Theorem \ref{kappastarind}. 
	%\begin{proof}
	To see this, from Lemma \ref{kappacurlstar}, 
	%observe that We have
	%\begin{eqnarray*}
	%\tilde{\kappa}&=&\kappa^{*}+\frac{1}{4}\bigg[\frac{-2n}{(n-1)^2}U_{12n}+\bigg(\frac{2}{(n-1)^2}\bigg)U_{3n}+\bigg(\frac{2}{n-1}\bigg)U_{1n}U_{2n}\bigg].
	%\end{eqnarray*}
\begin{eqnarray*}
n\tilde{\kappa}&=&n\kappa^{*}+\frac{n}{4}\big[\frac{-2n}{(n-1)^2}U_{12n}+\big(\frac{2}{(n-1)^2}\big)U_{3n}+\big(\frac{2}{n-1}\big)U_{1n}U_{2n}\big].
\end{eqnarray*}
	%Note that 
Under independence $\mu_{12}=\mu_{1}\mu_{2}$, and 
	%			From Theorem \ref{theo:standard_degen_u} we have
	%			\begin{eqnarray}
		%				n\kappa^{*}\xrightarrow D \sum_{i,j=0}^{\infty}\lambda_i\eta_j(Z_{ij}^2-1)
		%			\end{eqnarray}
	%			consider,
	%			$$\frac{1}{4}\bigg[\frac{-2n^2}{(n-1)^2}U_{12n}+\bigg(\frac{2n}{(n-1)^2}\bigg)U_{3n}+\bigg(\frac{2n}{n-1}\bigg)U_{1n}U_{2n}\bigg]$$
	as $n \rightarrow \infty$,  the second term above converges almost surely to 
	%		\begin{eqnarray*}
		%&& \frac{1}{4}\bigg[\frac{-2n^2}{(n-1)^2}U_{12n}+\bigg(\frac{2n}{(n-1)^2}\bigg)U_{3n}+\bigg(\frac{2n}{n-1}\bigg)U_{1n}U_{2n}\bigg] \xrightarrow{a.s}\\ && 
		$\frac{1}{4}[-2\mu_{12}+0+2\mu_{1}\mu_{2}]
		%\frac{1}{4}[-2\mu_{1}\mu_{2}+2\mu_{1}\mu_{2}]
		=0$.
		%\end{eqnarray*}
Therefore the result follows from Theorem \ref{kappastarind}.
		%		\textbf{what about the relation between $\kappa^*$ and $\tilde\kappa$?}
		%\end{remark}
		%\begin{remark}
		\vskip3pt
		
\noindent 
(ii) The asymptotic distribution result of $\hat{\kappa}$ in Theorem \ref{kappatildadegen} also follows from Theorem \ref{kappastarind}. 
		%can be obtained 	as follows.
		%	\begin{remark}\label{sumeigenvl}
First note that by the orthonormality of $\{g_{1k}(\cdot)\}$ and $\{g_{2k}(\cdot)\}$ we have
			%\begin{equation*}
			%\BE h_{F_1}(X_1, X_1)= \BE \sum\lambda_{k}g_{1k}(X_1)^2=\sum\lambda_{k},
			%\end{equation*}
			%and
			%\begin{equation*}
			%\BE h_{F_2}(Y_1, Y_1)= \BE \sum\eta_{k}g_{2k}(Y_1)^2=\sum\eta_{k}.
			%\end{equation*}
			%\end{remark}
			%	From Remark \ref{sumeigenvl},
\begin{equation*}
	\BE_{F_{1}} [h_{F_1}(X_1, X_1)]=\sum_{i=0}^{\infty}\lambda_{i}\ \ \text{and}\ \  \BE_{F_{2}}[ h_{F_2}(Y_1, Y_1)]=\sum_{j=0}^{\infty}\eta_{j}.
			\end{equation*}

Also note that from \eqref{eq:hdef} we obtain,
\begin{equation}
h_{F_1}(x_1, x_1)=-\dfrac{1}{2} \big[-2g_{F_1}(x_1)+\mu_{1}\big]\ \ \text{and}\ \  \BE_{F_{1}}[ h_{F_1}(X_1, X_1)]=\frac{\mu_{1}}{2}=\sum_{i=0}^{\infty}\lambda_{i}, 
\end{equation}
\begin{equation}
h_{F_2}(y_1, y_1)=-\dfrac{1}{2} \big[-2g_{F_2}(y_1)+\mu_{2}\big]\ \ \text{and} \ \ \BE_{F_{2}} [h_{F_1}(Y_1, Y_1)]=\frac{\mu_{2}}{2}=\sum_{j=0}^{\infty}\eta_{j}.
\end{equation}

By Lemma \ref{kappacurlstar},
			%\ref{kappahatstar} 
			%	we have,
$$n\hat{\kappa}	=n\kappa^{*}+\frac{n}{4}\big[\frac{(2-3n)}{n^2}U_{12n}-2\big(\frac{2-3n}{n^2}\big)U_{3n}+\big(\frac{1-2n}{n^2}\big)U_{1n}U_{2n}\big].$$

Hence, as $n \rightarrow \infty$,
\begin{eqnarray*}
		n\hat{\kappa}&\xrightarrow{D}&\sum_{i,j=0}^{\infty}\lambda_i\eta_j(Z_{ij}^2-1)+\frac{1}{4}[\mu_{1}\mu_{2}]\\
		&=&\sum_{i,j=0}^{\infty}\lambda_i\eta_j(Z_{ij}^2-1)+\sum_{i,j=0}^{\infty}\lambda_i\eta_j=\sum_{i,j=0}^{\infty}\lambda_i\eta_jZ_{ij}^2.
\end{eqnarray*}
\end{remark}
These results are in agreement with the results obtained in Bergsma (2006).
		
%	\section{Value of $\kappa$ in some dependent models}\label{sec:truevl}
%\subsubsection*{$\kappa$ for bivariate normal}

\section{Simulations}\label{sec:simu}

Extensive simulations were carried out under both dependence and independence scenarios by varying the values of the parameter $\theta$. We investigated the estimators also from the perspectives of bias and power (for testing of independence). In addition, for specific cases, we also explored the effect of sample size.

\subsection{Distribution under independence}\label{subsec:independent}

The case where $X$ and $Y$ are independent may be termed the ``null case'', and has been partly explored in \cite{bergsma2006new}. From Theorems \ref{kappatildadegen} and \ref{kappastarind}, the asymptotic distributions of $n\kappa^{*}$, $n\tilde\kappa$ and $n\hat\kappa$ involve the eigenvalues $\{\lambda_i\}$ and $\{\eta_j\}$, computation of which requires solving functional integral equations. This issue was addressed by \cite{bergsma2006new} by a suitable discrete approximation to the eigenvalue equations. The details are in Appendix B, where we have also rectified some errors in the original algorithm. The asymptotic distributions of the estimates can be approximated by using these discrete approximations, and serve as alternatives to simulation based approximation.

In Figures \ref{fig:ustatasym} and \ref{fig:vstatasym}, we illustrate the (approximate) asymptotic distributions of the three estimates using the largest $1000$ eigenvalues. Incidentally, it was observed that (results not shown), the first $10$ largest eigenvalues approximate up to $90$ percent of the total eigenvalues.

We compared the simulated distributions of $n\kappa^{*}$, $n\tilde\kappa$ and $n\hat{\kappa}$ based on $1000$ replications, for the six chosen bivariate distributions, with $n=500$ samples generated for each replicate of the estimate.

\begin{figure}[h!]
	\centering
	\includegraphics[height=60mm, width=0.9\linewidth]{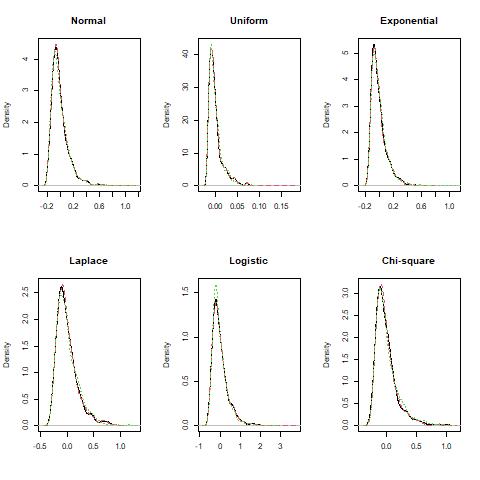}
	\caption{Distribution of $U$-statistic based estimates, for the six chosen distributions. Asymptotic limit based on $k=1000$ eigenvalues, in green. Simulated limit for $\tilde{\kappa}$ in black, and for $\kappa^*$ in red.}
	\label{fig:ustatasym}
\end{figure}
\begin{figure}[h!]
	\centering
	\includegraphics[height=60mm, width=0.9\linewidth]{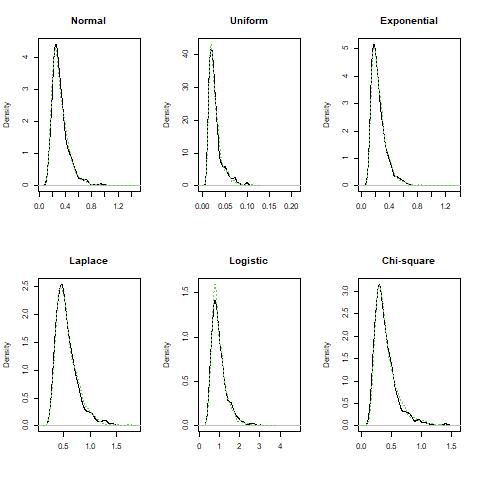}
	\caption{
	$V$-statistic based estimate $\hat \kappa$, for the six chosen distributions. Asymptotic limit, based on $k=1000$ eigenvalues, in green. Simulated limit in black.}
	\label{fig:vstatasym}
\end{figure}

As may be observed, there is a strong agreement between the theoretical approximation (of the asymptotic distributions) based on discrete approximation, with the simulation based approximations. Thus, either may be used as a reliable estimate for the ``null'' distribution, for further exploration.

\subsection{Distribution under dependence}\label{subsec:dependent}

We explore various aspects, by varying the bivariate distributions as well as the value of $\theta$, the latter to cover a spectrum of dependence.

In many cases an explicit relationship between $\theta$ and $\kappa (\theta)$ may not be available. We explore their relationship through the three ``non-parametric'' estimates of $\kappa$, namely $\hat\kappa$, $\tilde\kappa$, and $\kappa^{*}$. By using the asymptotic results for these estimates based on large sample, we explore the relationship between (empirical) $\kappa$ values and (true) $\theta$ for the six distributions (details given in Appendix \ref{sec:bvd}) under various degrees of dependence, see Figure \ref{fig:3ests_6dst_mean}. 
\begin{figure}[h!]
	\centering
	\includegraphics[height=60mm, width=0.75\linewidth]{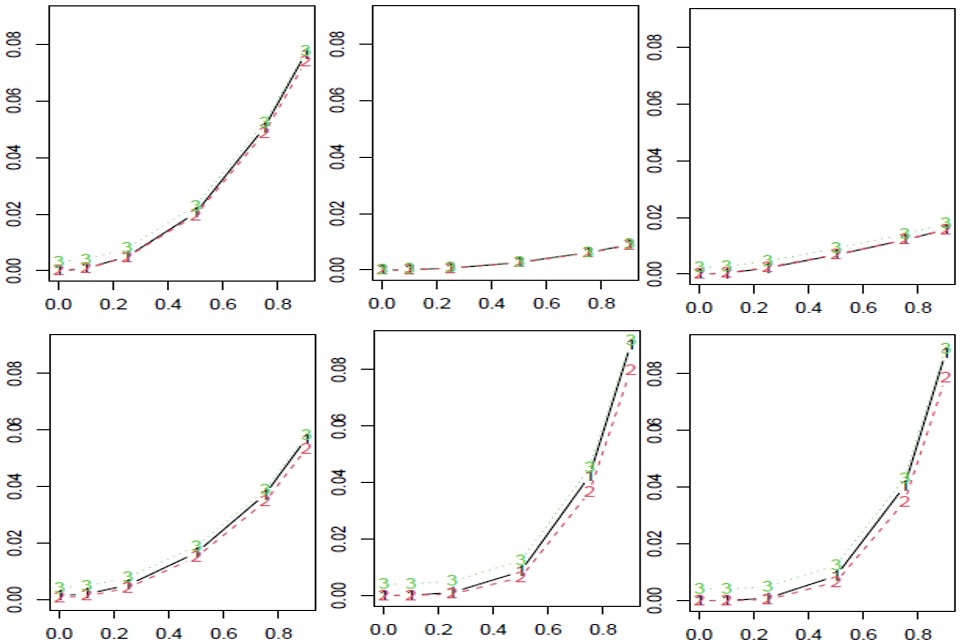}
	\caption{Empirical means (Y-axes) of, 
		%the three estimates, 
		$\kappa^*$ (marked-1), $\tilde\kappa$ (marked-2), and $\hat{\kappa}$ (marked-3), for the 6 distributions with varying values of the dependency parameter $\theta$ (X-axes). In each case, the sample size is $n=100$, and the number of replicates is $1000$.}
	\label{fig:3ests_6dst_mean}
\end{figure}

Recall that though $\theta$ is associated with dependence, in physical sense this parameter represents a different characteristic for each of the six bivariate distributions. In each case the distributions of the three estimators of $\kappa$ clearly  appear to be  sensitive to the magnitude of $\theta$. 

Figure \ref{fig:3ests_6dst_dep} presents the box-plots of the empirical asymptotic distributions of the three  estimates with various underlying distributions having varying degrees of dependence.
\begin{figure}[H]
    \centering
    \includegraphics[height=150mm, width=0.7\linewidth]{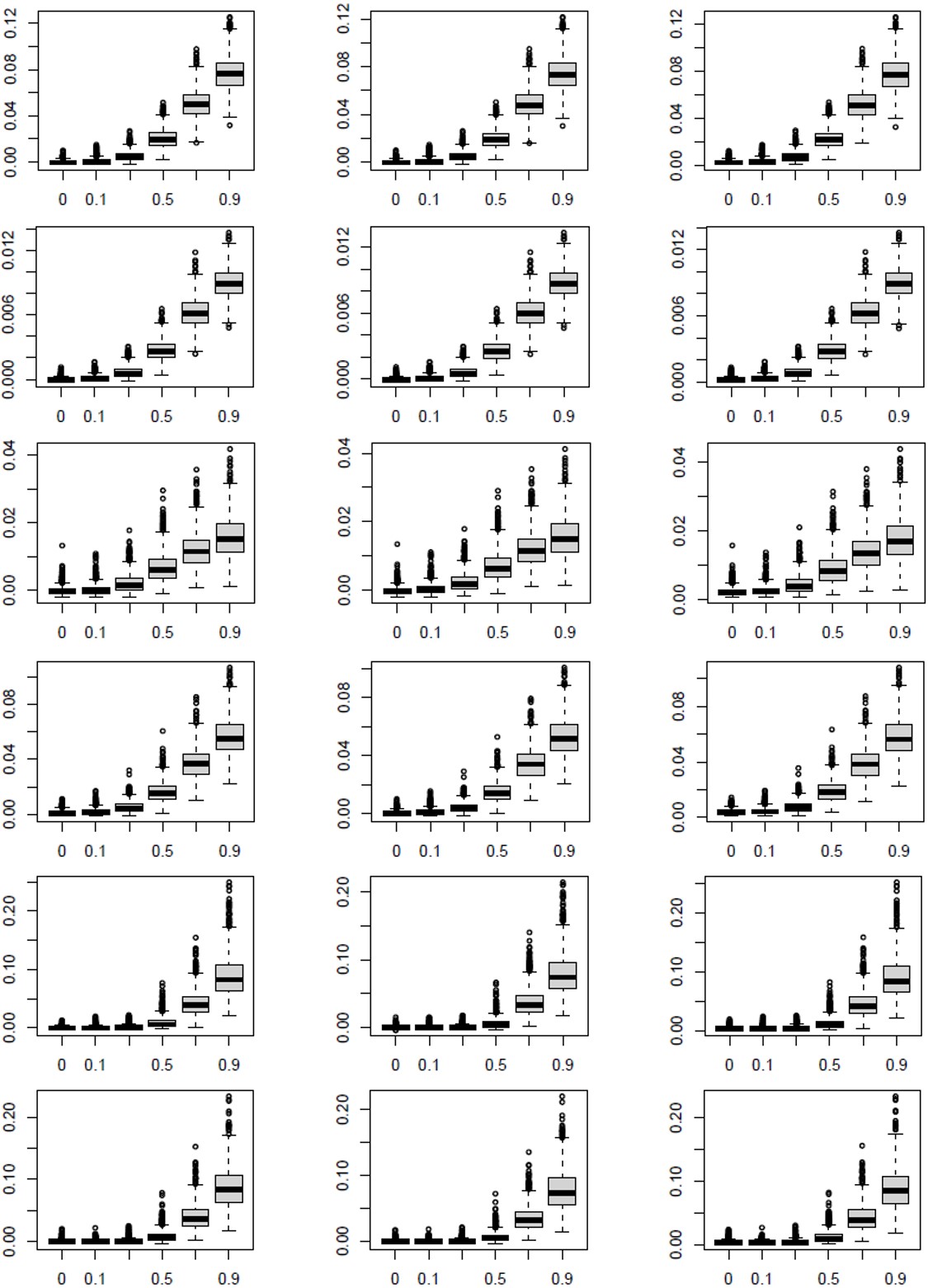}
    \caption{Box plots of 
	$\kappa^*$ (left), $\tilde\kappa$ (middle), and $\hat{\kappa}$ (right) for the six bivariate distributions.\newline
	The true values of the dependency parameter $\theta$ are on the horizontal axes.	\newline
	Sample size $n=100$ and $1000$ replicates.}
\label{fig:3ests_6dst_dep}
\end{figure}

This systematic sensitivity to the underlying parameter values, make these estimates promising to be useful for further inferential purposes. 

Next we explored the accuracy of these approximate distributions in capturing the true value (of $\kappa$). For ease of comparison, we considered two distributions, namely the normal and the exponential. In Figure \ref{fig:3ests_2ns}, we have presented the empirical distributions, based on 2500 samples, for the three estimators. 
\begin{figure}[h!]
	\centering
	\includegraphics[width=0.8\linewidth]{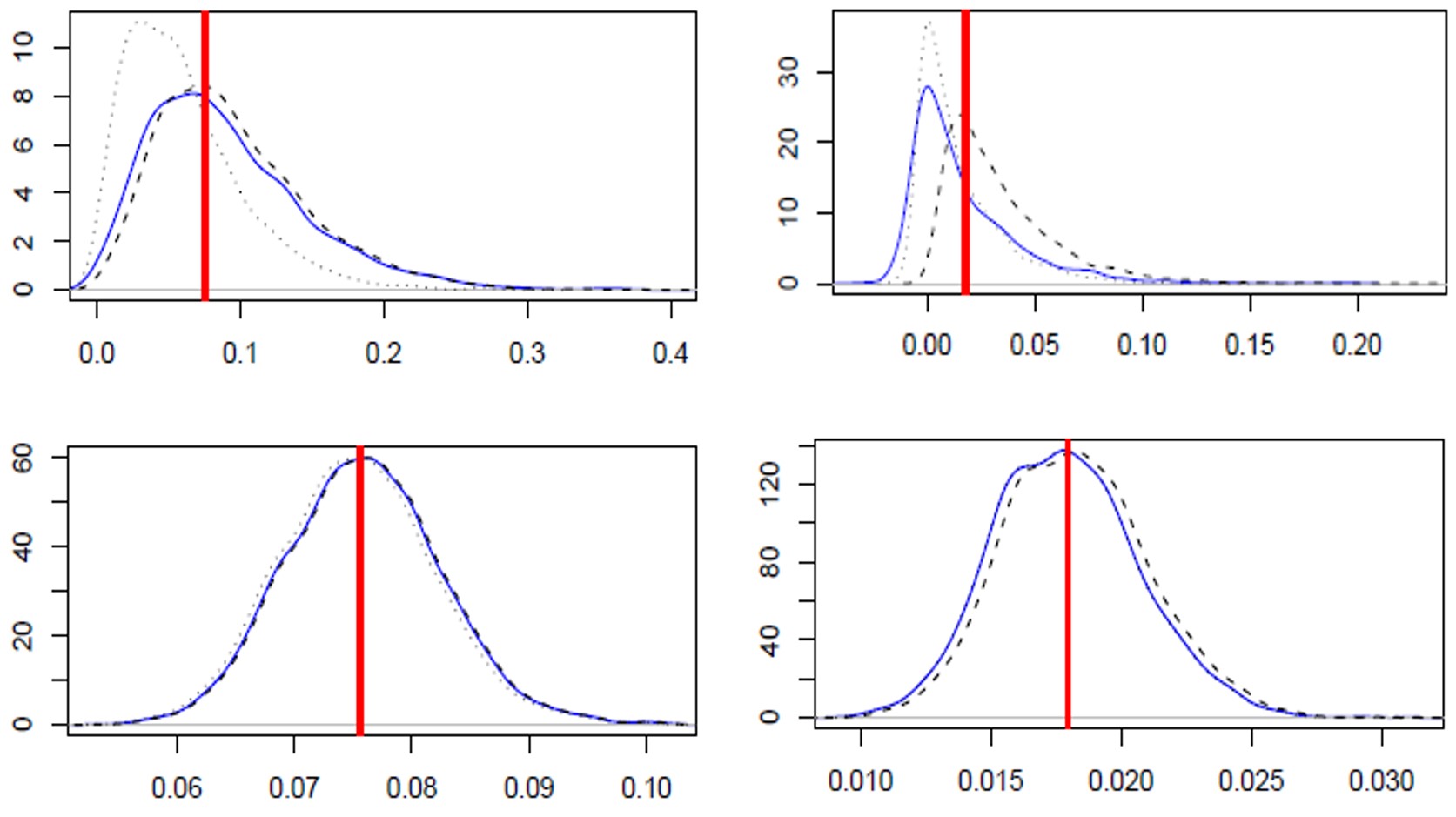}
	\caption{Empirical distributions of 
		%the three estimates, 
		$\kappa^*$ (blue), $\tilde\kappa$ (dotted), and $\hat{\kappa}$ (dashed). True value of $\kappa$ at red vertical line. Left: bivariate normal. Right: bivariate exponential. Sample sizes $n=10$ (top), and $n=500$ (bottom). }
	\label{fig:3ests_2ns}
\end{figure}
As may be expected, for larger sample sizes, all three estimates behave rather similarly, whereas there are distinct differences when the sample size is small. The true value however seems to be well captured by all three estimates even when the sample size is small.

Nevertheless, for small sample sizes there is an effect of the underlying distribution, since the pattern of variation seem to be differing in the two families of bivariate normal and bivariate exponential. See Figure \ref{fig:3ests_2dist_n10}. Also there seems to be an indication that, in general, the two $U$-statistics based estimators, $\tilde\kappa$ and $\kappa^*$, are distributionally similar. However the $V$-statistics based estimator $\hat\kappa$ differs slightly in the relationship with $\theta$, as compared to $\tilde\kappa$ and $\kappa^{*}$. 

\begin{figure}[h!]
	\centering
	\includegraphics[height=40mm, width=0.7\linewidth]{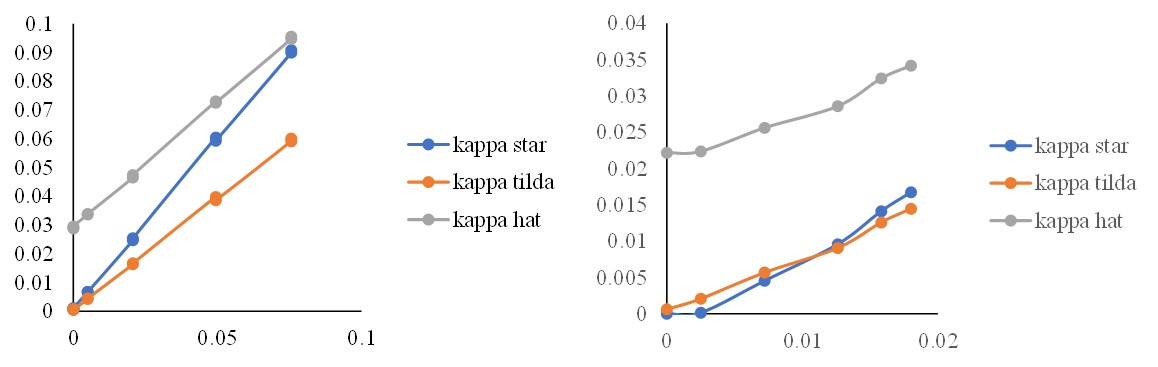}
	\caption{Empirical means ($Y$-axes) under dependence for $\hat\kappa$, $\tilde\kappa$ and $\kappa^{*}$  as $\theta$ ($X$-axes) varies. 
	\newline Sample size is $n=10$. Left: bivariate normal. Right: bivariate exponential.}
	\label{fig:3ests_2dist_n10}
\end{figure}

The estimator $\hat\kappa$ appears to have a consistent upward bias. This is rather clearly visible for the bivariate Exponential distribution (see Figure \ref{fig:3ests_2dist_n10}). The bias of $\hat{\kappa}$ has been briefly explored via simulations, and a summary is shown in Figure \ref{fig:biascomp}. 

For bivariate exponential, at first glance it may appear that $\hat{\kappa}$ is substantially different from the other two estimates due to this bias. However our simulation suggests that although the underlying parameter ($\theta$) has some effect on this bias (see Figure \ref{fig:biascomp} right panel) with increased sample size this bias can be brought within an acceptable threshold.
\begin{figure}[h!]
	\centering
	\includegraphics[height=40mm, width=0.7\linewidth]{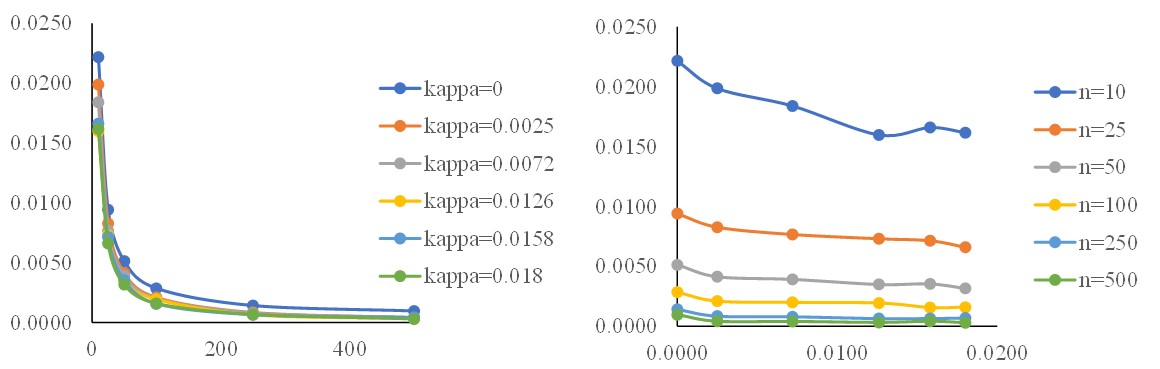}
	\caption{Empirical bias ($Y$-axes) of $\hat{\kappa}$ with samples from bivariate exponential distributions. \newline $X$-axes presents  sample sizes (left panel), and true $\kappa$ (right panel).}
	\label{fig:biascomp}
\end{figure}

\subsection{Power for testing independence}\label{subsec:power}

Since it is a common practice to use measures of dependence also for the purpose of testing independence, we explored the powers of these estimators of $\kappa$ for testing independence in various setup.
As expected the two $U$-statistics based estimators $\kappa^*$ and $\tilde\kappa$ perform quite comparably, and they perform better than $\hat{\kappa}$ in this regard. See Figure \ref{fig:2est_6dist_pwr_reg} for further information.

\begin{figure}[h!]
	\centering
	\includegraphics[height=50mm, width=0.6\linewidth]{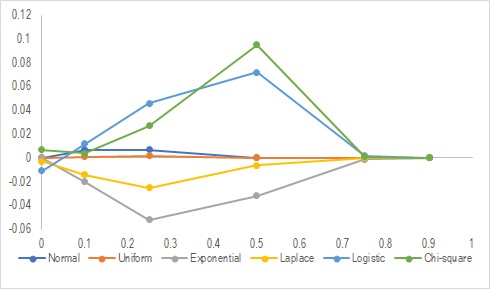}
	\caption{Differences in empirical powers ($Y$-axes) of $\kappa^*$ and $\tilde\kappa$ with samples from various bivariate distributions and varying degree dependence as measured by $\theta$ (X-axes), with sample size $n=100$ and $1000$ replicates.}
	\label{fig:2est_6dist_pwr_reg}
\end{figure}

Based on Figure \ref{fig:2est_6dist_pwr_reg} there are distributional families for which each estimator marginally performs better than the others. For Normal and uniform, the performances are comparable. For Logistic and Chi-square, $\kappa^*$ has noticeably higher power in detecting departure from independence. For logistic, it appeared that for small values of $\theta$,  $\tilde\kappa$ might perform better. However an appropriately increased sample size overcomes this issue in favor of $\kappa^*$ (results not shown). We also carried out assessment of contiguous hypotheses and there are no mentionable changes in the conclusions.
\vskip5pt

We now considered comparisons with other eight measures of dependence available in the literature. They include distance/kernel based methods   (a) Distance Correlation (dCor), (b) Hoeffding's $D$, (c) Ball Correlation Statistics; rank based methods (d) Bergsma's $\tau$, (e) Chatterjee’s $\xi$, and copula based methods  (f)) Quantification of Asymmetric Dependence (QAD),  (g) Multilinear Copula,  and finally a mutual information bases method, (h) Copula Entropy. These have been chosen due to their prominence in the literature and applications, and also due to the variety of techniques they cover between them. One may refer to \cite{ma2022}, \cite{chatterjee2021} and \cite{tjostheim2022} for further information on, and discussions of, such methods.

We carried out extensive simulation studies with these eight, and the three estimators of Bergsma's $\kappa$. It appears that the proposed estimator of Bergsma's $\kappa$ either outperforms the competing methods, or performs equally well as them, in terms of the empirically estimated powers (see Figure \ref{fig:power_comp}).

\begin{figure}[h!]
	\vskip5pt
	\centering
	\includegraphics[height=70mm, width=0.8\linewidth]{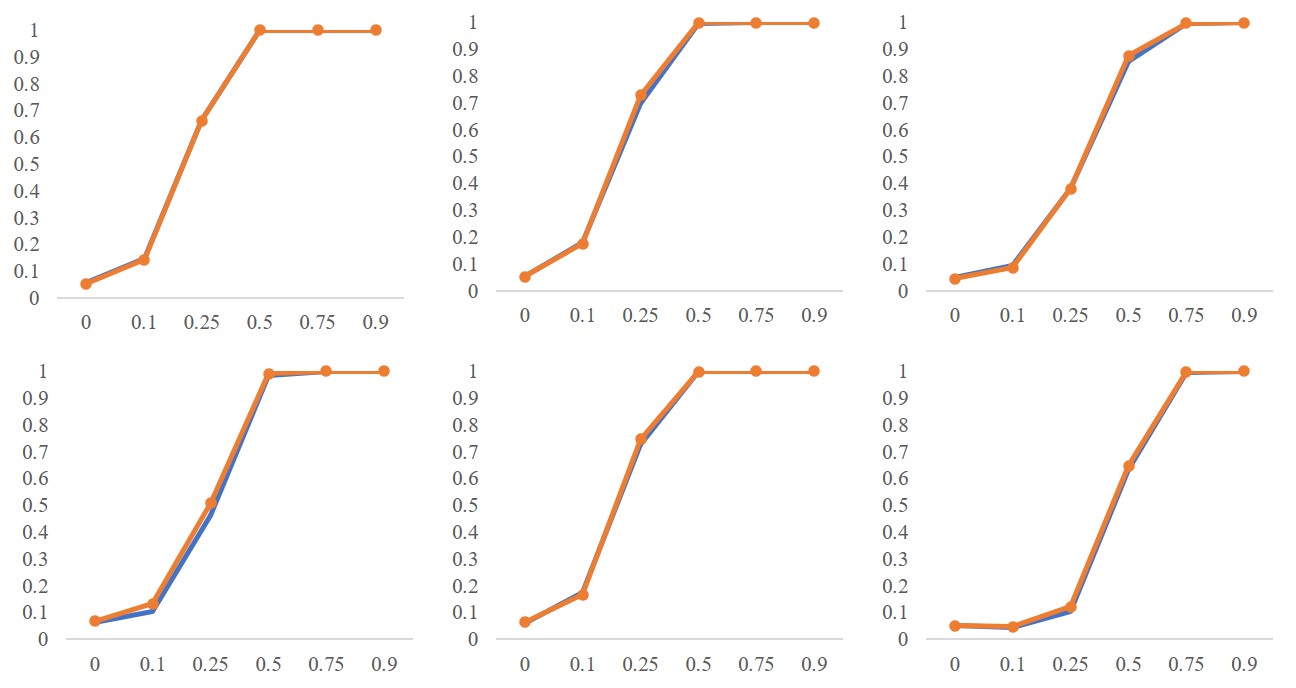}
	\caption{On $Y$-axes, estimated powers of $\kappa^*$, and three of the other top performing measures, under varying $\theta$ ($X$-axes). Line without marker for $\kappa^*$; with marker for the competing measure. \newline 
	%Kappa star
	Clockwise from top left: \newline 
	--samples from bivariate distributions--Normal, Uniform, Exponential, Laplace, Logistic, Chi-square.
	\newline
	 --the competing measures are: dCor, Hoeffding's $D$, dCor, Bergsma's $\tau$, Hoeffding's $D$ and dCor.}
	\label{fig:power_comp}
\end{figure}

In Table \ref{table:est_power} we present power computation for tests of independence, which are based on the three estimators of $\kappa$, as well as on other known measures of dependence.
\begin{table}[H]
\footnotesize
\caption{
%\label{table:1} 
\footnotesize Estimated powers for various dependence measures based on data from six bivariate distributions with varying degrees of dependence.}
\vskip3pt
\begin{center}
\begin{tabular}{l|cccccc|cccccc} \hline
Theta $\rightarrow$ & 0 & 0.1 & 0.25 & 0.5 & 0.75 & 0.9 & 0 & 0.1 & 0.25 & 0.5 & 0.75 & 0.9 \\ \hline
& \multicolumn{6}{c|}{Bivariate Normal} & \multicolumn{6}{c}{Bivariate Uniform} \\ \hline
Proposed $\kappa^{*}$ & 0.06 & 0.15 & 0.66 & 1 & 1 & 1 & 0.05 & 0.18 & 0.7 & 1 & 1 & 1 \\
Bergsma's $\tilde{\kappa}$ & 0.06 & 0.14 & 0.65 & 1 & 1 & 1 & 0.05 & 0.18 & 0.7 & 1 & 1 & 1 \\
Bergsma's $\hat \kappa$ & 0.05 & 0.14 & 0.65 & 1 & 1 & 1 & 0.06 & 0.18 & 0.7 & 1 & 1 & 1 \\ \hline
Copula Entropy & 0.05 & 0.05 & 0.12 & 0.64 & 1 & 1 & 0.07 & 0.06 & 0.23 & 0.97 & 1 & 1 \\
dCor & 0.05 & 0.14 & 0.66 & 1 & 1 & 1 & 0.05 & 0.17 & 0.71 & 1 & 1 & 1 \\
Hoeffding's $D$ & 0.05 & 0.14 & 0.63 & 1 & 1 & 1 & 0.05 & 0.17 & 0.73 & 1 & 1 & 1 \\
Bergsma's $\tau$ & 0.06 & 0.14 & 0.63 & 1 & 1 & 1 & 0.05 & 0.18 & 0.71 & 1 & 1 & 1 \\
Ball & 0.04 & 0.08 & 0.35 & 0.96 & 1 & 1 & 0.06 & 0.16 & 0.71 & 1 & 1 & 1 \\
Quant Asymp Dep & 0.04 & 0.08 & 0.29 & 0.93 & 1 & 1 & 0.04 & 0.09 & 0.33 & 0.97 & 1 & 1 \\
Multilinear copula & 0.06 & 0.14 & 0.64 & 1 & 1 & 1 & 0.05 & 0.17 & 0.7 & 1 & 1 & 1 \\
Xi & 0.04 & 0.05 & 0.13 & 0.68 & 1 & 1 & 0.05 & 0.05 & 0.18 & 0.85 & 1 & 1 \\ \hline
& \multicolumn{6}{c|}{Bivariate Exponential} & \multicolumn{6}{c}{Bivariate Laplace} \\ \hline
Proposed $\kappa^{*}$ & 0.05 & 0.1 & 0.38 & 0.86 & 1 & 1 & 0.06 & 0.1 & 0.46 & 0.98 & 1 & 1 \\
Bergsma's $\tilde\kappa$ & 0.05 & 0.12 & 0.44 & 0.88 & 1 & 1 & 0.06 & 0.1 & 0.49 & 0.99 & 1 & 1 \\
Bergsma's $\hat\kappa$ & 0.05 & 0.09 & 0.37 & 0.81 & 0.99 & 1 & 0.06 & 0.09 & 0.43 & 0.97 & 1 & 1 \\ \hline
Copula Entropy & 0.06 & 0.06 & 0.09 & 0.23 & 0.53 & 0.68 & 0.04 & 0.06 & 0.13 & 0.57 & 1 & 1 \\
dCor  & 0.05 & 0.09 & 0.38 & 0.88 & 1 & 1 & 0.06 & 0.11 & 0.47 & 0.99 & 1 & 1 \\
Hoeffding's $D$ & 0.07 & 0.09 & 0.37 & 0.84 & 1 & 1 & 0.07 & 0.13 & 0.5 & 0.99 & 1 & 1 \\
Bergsma's $\tau$ & 0.07 & 0.09 & 0.38 & 0.84 & 1 & 1 & 0.07 & 0.13 & 0.51 & 0.99 & 1 & 1 \\
Ball & 0.05 & 0.07 & 0.23 & 0.69 & 0.96 & 0.99 & 0.05 & 0.07 & 0.21 & 0.84 & 1 & 1 \\
Quant Asymp Dep & 0.05 & 0.09 & 0.15 & 0.51 & 0.88 & 0.96 & 0.05 & 0.06 & 0.23 & 0.88 & 1 & 1 \\
Multilinear copula & 0.07 & 0.09 & 0.38 & 0.85 & 1 & 1 & 0.07 & 0.13 & 0.51 & 0.99 & 1 & 1 \\
Chatterjee's $\psi$ & 0.05 & 0.05 & 0.09 & 0.25 & 0.54 & 0.68 & 0.04 & 0.04 & 0.14 & 0.62 & 1 & 1 \\ \hline
& \multicolumn{6}{c|}{Bivariate Logistic} & \multicolumn{6}{c}{Bivariate Chi-square} \\ \hline
Proposed $\kappa^{*}$ & 0.06 & 0.18 & 0.73 & 1 & 1 & 1 & 0.05 & 0.04 & 0.11 & 0.64 & 1 & 1 \\
Bergsma's $\tilde\kappa$ & 0.06 & 0.17 & 0.71 & 1 & 1 & 1 & 0.05 & 0.04 & 0.07 & 0.56 & 0.99 & 1 \\
Bergsma's $\hat\kappa$ & 0.06 & 0.18 & 0.71 & 1 & 1 & 1 & 0.05 & 0.05 & 0.1 & 0.59 & 0.99 & 1 \\ \hline
Copula Entropy & 0.06 & 0.06 & 0.23 & 0.97 & 1 & 1 & 0.04 & 0.06 & 0.05 & 0.11 & 0.67 & 1 \\
dCor & 0.06 & 0.17 & 0.73 & 1 & 1 & 1 & 0.05 & 0.05 & 0.12 & 0.65 & 1 & 1 \\
Hoeffding's $D$ & 0.06 & 0.17 & 0.75 & 1 & 1 & 1 & 0.05 & 0.05 & 0.06 & 0.36 & 1 & 1 \\
Bergsma's $\tau$ & 0.06 & 0.17 & 0.74 & 1 & 1 & 1 & 0.05 & 0.06 & 0.06 & 0.37 & 1 & 1 \\
Ball & 0.06 & 0.14 & 0.71 & 1 & 1 & 1 & 0.05 & 0.04 & 0.07 & 0.37 & 0.99 & 1 \\
Quant Asymp Dep & 0.05 & 0.09 & 0.38 & 0.99 & 1 & 1 & 0.05 & 0.05 & 0.06 & 0.17 & 0.9 & 1 \\
Multilinear copula & 0.06 & 0.18 & 0.74 & 1 & 1 & 1 & 0.05 & 0.06 & 0.06 & 0.36 & 1 & 1 \\
Chatterjee's $\psi$ & 0.05 & 0.07 & 0.19 & 0.87 & 1 & 1 & 0.04 & 0.05 & 0.04 & 0.11 & 0.66 & 1 \\ \hline
\end{tabular}\label{table:est_power}
\end{center}
\end{table}

\subsection{Computational Efficiency}\label{subsec:comptime}

We have carried out comparative simulations to assess the time taken by the multiple well known measures, and the three $\kappa$ based measures. Our proposed method along with Bergsma's methods appear to be many times faster than several of the other methods. This was found to be so even without using any advanced computational techniques.

We also noted an interesting feature regarding the time taken to compute these measures when the underlying sample size (say $n$) remains the same. It appears that several of the well known measures of dependence are adversely affected as the dependence between $X$ and $Y$ increases. Typically in the literature the computations times are presented as a function of $n$. The proposed method based on $\kappa$ does not seem to suffer from such drawbacks, as far as the mean and standard deviations of the time taken (see Figure \ref{fig:comp_time}).

\begin{figure}[h!]
	\vskip5pt
	\centering
	\includegraphics[height=70mm, width=0.7\linewidth]{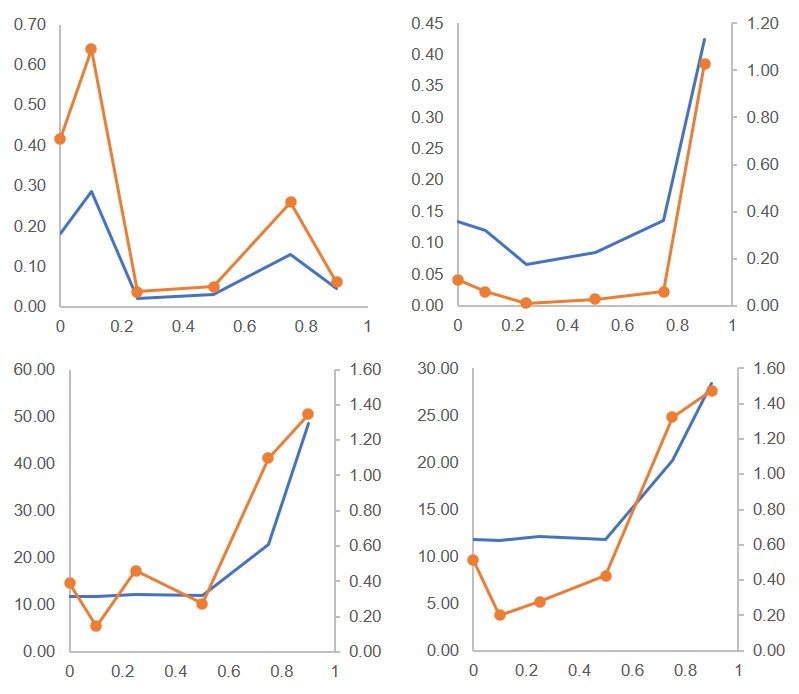}
	\caption{Estimated mean and standard deviations (over 10 replications) for time taken in seconds, to compute 100 values for four selected measures, based on 100 pairs of bivariate normal data. \newline 
	Summaries:  (1) Mean (solid line),  and (2) standard deviation (solid line with marker). 
	\newline Measures: $\kappa^*$ (top left), dCor (top right), Hoeffding's $D$ (bottom left), Bergsma's $\tau$ (bottom right). }
	\label{fig:comp_time}
\end{figure}

We now present efficiency data for computation of these measures of dependence mentioned earlier in and Table \ref{table:comp_time}.

\begin{table}[h!]
\footnotesize
\caption{
%\label{table:2} 
\footnotesize Estimated mean and standard deviations of time taken (in seconds) over 10 replications to compute 100 values each for the various measures of dependence based on 100 pairs of bivariate normal data with varying degree of dependence.}
\vskip3pt
\begin{center}
\begin{tabular}{l|cccccc|cccccc} \hline
Theta $\rightarrow$ &0&0.1&0.25& 0.5 & 0.75 & 0.9 & 0 & 0.1 & 0.25 & 0.5 & 0.75 & 0.9 \\ \hline
& \multicolumn{6}{c}{Mean time} & \multicolumn{6}{c}{Standard deviation of time} \\ \hline
$\kappa^{*}$ & 0.18 & 0.29 & 0.02 & 0.03 & 0.13 & 0.05 & 0.42 & 0.64 & 0.04 & 0.05 & 0.26 & 0.06 \\
$\tilde{\kappa}$ & 0.67 & 1.78 & 0.36 & 0.31 & 1.10 & 0.58 & 0.56 & 3.39 & 0.16 & 0.11 & 1.28 & 0.56 \\
$\hat \kappa$ & 0.54 & 0.83 & 0.30 & 0.31 & 0.67 & 0.96 & 0.49 & 1.06 & 0.04 & 0.05 & 0.40 & 1.73 \\
Copula Entropy & 0.95 & 1.15 & 0.54 & 0.52 & 1.37 & 1.15 & 0.91 & 1.27 & 0.05 & 0.04 & 1.24 & 1.31 \\
dCor & 0.13 & 0.12 & 0.07 & 0.08 & 0.14 & 0.42 & 0.11 & 0.06 & 0.01 & 0.03 & 0.06 & 1.03 \\
Hoeffding's $D$ & 11.87 & 11.89 & 12.21 & 11.99 & 22.75 & 48.46 & 0.39 & 0.15 & 0.46 & 0.27 & 1.10 & 1.35 \\
Bergsma's $\tau$ & 11.87 & 11.74 & 12.16 & 11.84 & 20.26 & 28.37 & 0.51 & 0.20 & 0.28 & 0.43 & 1.32 & 1.47 \\
Ball & 0.14 & 0.14 & 0.07 & 0.06 & 0.37 & 0.23 & 0.05 & 0.08 & 0.08 & 0.04 & 0.28 & 0.33 \\
QAD & 26.22 & 17.41 & 14.54 & 13.32 & 25.38 & 16.07 & 9.89 & 5.83 & 7.74 & 2.21 & 10.14 & 10.21 \\
Multilin. copula & 10.39 & 8.63 & 7.54 & 9.32 & 8.72 & 10.24 & 6.71 & 2.11 & 0.62 & 4.60 & 2.09 & 4.49 \\
Chatterjee's $\psi$ & 0.04 & 0.03 & 0.02 & 0.03 & 0.05 & 0.05 & 0.02 & 0.01 & 0.00 & 0.01 & 0.03 & 0.10 \\ \hline
\end{tabular}
\end{center}\label{table:comp_time}
\end{table}

\section{Discussion}\label{sec:disc}

An interesting measure of independence, namely $\kappa$, was introduced in \cite{bergsma2006new}. This becomes a special case of 
dCov, when the marginal distributions are one-dimensional. He also provided two estimates for $\kappa$, based on estimated $U$- and $V$-statistics. 

We have provided alternate descriptions of $\kappa$, and one of these leads us to a third (new) estimator of $\kappa$. The relation of $\kappa$ with a naturally arising dependence parameter in some standard bivariate distributions has been studied in the literature, we add the bivariate Gumbel distribution to this mix. We prove convexity of $\kappa$ that in case of BVN and in case of GBED-I $\kappa$ is an increasing function of $\theta$.

The asymptotic distributions of the two Bergsma's estimators under independence of the component variables was already known. While no distributional results were available in the more complex dependent case for the two original estimates, we have shown that the asymptotic distribution of all the three estimates are all Gaussian with the same parameters. In the independent case, we give a detailed proof for the asymptotic distribution of all the three estimators, which now are distributed as weighted sum of i.i.d. chi-square random variables.

We have carried out extensive simulations. This has brought out hitherto unexplored interesting properties of these estimators, such as the pattern of bias for small sample sizes in the $V$-statistics based estimator proposed in \cite{bergsma2006new}.

We have observed that the estimators are sensitive to the underlying dependence parameter values. This lends them to be used for testing independence. Through the simulation results, we explored the power of these three estimators along with other well known similar measures for testing independence. The porposed $\kappa^*$ estbalishes itself as a strong competitor both interms of power and computing efficiency.

Computationally it could be meaningful to explore dependency scenarios other than those covered here, for example dependence in presence of noise as suggested in \cite{chatterjee2021}. 

Future investigations could include theoretical investigations like obtaining an analytic formulae for the limiting variance in the dependence case for various parametric families. One could also investigate extensions of these estimators when underlying data is multivariate.

While multivariate data seems to be the flavour of the day, one of our primary motivation was situations, where we might have several univariate data sets where dependence is inherently true among the series (for example COVID 19 data) and where modelling it as a multivariate data with standard assumptions may not be feasible. In a recent work \cite{kappara2023} have successfully extended Bergsma's $\kappa$ in the spirit of other spatial measures, and applied it to the analysis of a spatial short time series COVID 19 data.

\vskip12pt

% \bibliography{references}
% \bibliographystyle{abbrvnat}

%\clearpage

\section*{Appendix}

\appendix
 %\newpage%
 \renewcommand{\thesection}{\Alph{section}}% For Alpha numeric number

 \section{$U$-statistics background}\label{sec:ustat}
	%In this section we collect some well-known results from $U$-statistics theory that we shall need. 
	%	In this section we collect some basic asymptotic results on $U$-statistics that we shall need.  
	%It is well-known that a $U$-statistic based on $i.i.d.$ observations has an asymptotic normal distribution. The proof of this result well-known as $U$-statistics Central Limit theorem(UCLT) is obtained by using the sums of projections of the $U$-statistic and applying the classical CLT for sample mean. \\
	%	The first projection(centered) of a $U$-statistic $U_n$, denoted by $\psi_1(\cdot)$, is the conditional expectation of $h$ given one of the coordinates: 
	%	\cite{bose2018u} 
	\noindent Let $X_{1},\ldots,X_{n}$ be independent and identically distributed (i.i.d.) random variables with distribution $F$. Let $h$ be a symmetric kernel of $m$ arguments. Then the $U$-statistic of order $m$ with kernel $h$ is defined as, 
	\begin{equation}
		U_n:={n\choose m}^{-1}\sum_{1\leq i_1<\cdots<i_{m}\leq n}h(X_{i_1},\ldots,X_{i_m}).
	\end{equation} 
	Suppose $\BE_{F} (h):=\BE_{F}[h(X_1, \ldots , X_m)]$ is finite. Then we can write 
	%	$Consider a $U$-statistic of order $m$ with a symmetric kernel $h$
	\begin{equation}
		U_n-\BE_{F} (h)={n\choose m}^{-1}\sum_{1\leq i_1<\cdots<i_{m}\leq n}h_m(X_{i_1},\ldots,X_{i_m}),
	\end{equation} 
	where
	$$h_m:=h-\BE_{F}(h),$$
	and successively for $1\leq d \leq m-1$,
	$$h_d(x_1,\ldots,x_d)=\BE_{F}\big[h_{d+1}(x_1,\ldots,x_d,X_{d+1})\big].$$
	Define the functions (projections) associate with $h$ or $U_n$ as (see \cite{serfling2009approximation}) 
	\begin{eqnarray}
		H_1(x_1)&:=&h_1(x_1) \label{firstproj}\\
		H_2(x_1,x_2)&:=&h_2(x_1,x_2)-H_1(x_1)-H_1(x_2)\nonumber\\
		&=&h_2(x_1,x_2)-h_1(x_1)-h_1(x_2)\nonumber\\
		H_3(x_1,x_2,x_3)&:=&h_3(x_1,x_2,x_3)-\sum_{1\leq i\leq 3}H_1(x_i)-\sum_{1\leq i<j\leq 3}H_2(x_i,x_j)\nonumber\\
		&\vdots&\nonumber\\
		%	&\vdots&\nonumber\\
		H_m(x_1,\ldots,x_m)&:=&h_m(x_1,\ldots,x_m)-\sum_{1\leq i\leq m}H_1(x_i)-\sum_{1\leq i<j\leq m}H_2(x_i,x_j) \ldots\nonumber\\ &&  -\sum_{1\leq i_1<\cdots<i_{m-1}\leq m}H_{m-1}(x_i,\ldots,x_{i_{m-1}}).
	\end{eqnarray}
	Note that these are all symmetric kernels 
	%$H_d$'s are symmetric functions 
	with mean $0$. They also satisfy,
	\begin{equation}
		\BE_{F}[ H_2(x_1,X_2)]=0, \ldots , \BE_{F}[ H_m(x_1,\ldots,x_{m-1},X_m)]=0.\label{ortho}
	\end{equation}
	%The above functions 
	For $1 \leq d \leq m$ define,
	\begin{equation}
		S_{dn}:=\sum_{1\leq i_1<\cdots<i_{d}\leq n}H_d(X_{i_1},\ldots,X_{i_d}).
	\end{equation}
	If $\BE_{F}[h^2(X_1, \ldots , X_m) < \infty$, then the summands $H_d(\cdot)$ within the $S_{dn}, 1\leq d \leq m$, are all pairwise orthogonal. 
	%given below in the following lemma is well-knoww.		The  is  is referred to as 
	\begin{lemma}
		Let $h=h(x_1, \ldots, x_m)$ be such that $\BE_{F}[h^2(X_1, \ldots , X_m)] < \infty$. Then,
		\begin{equation}
			U_n-\BE_{F} (h)=\sum_{d=1}^{m}{m \choose d}{n \choose d}^{-1}S_{dn}.\label{Hdecom}
		\end{equation}
		Suppose $R_n^{(c)}$ is defined by 
		%the remainder obtained by truncating the $H$- decomposition after $c$ terms, we get,
		\begin{equation}
			U_n-\BE_{F} (h)=\sum_{d=1}^{c}{m \choose d}{n \choose d}^{-1}S_{dn}+R_n^{(c)}, \ c\leq m.\label{Rn}
		\end{equation} 
		Then for every $c\leq m$, $R_n^{(c)}$ is a $U$-statistic with kernel $\sum_{j=c+1}^{m}H_j$, and $n^{c/2}R_n^{(c)} \rightarrow   0$ in probability.
		%for every $c\leq m.$
	\end{lemma}
The representation given in \eqref{Hdecom} for $U$-statistics  	
%expresses $U_n$ may be expressed as a linear combination of the sums $S_{1n},\ldots S_{mn}$, 
is known as the \textbf{Hoeffding decomposition}.
%	The following result is due to \cite{hocffding1948class}.
	\begin{theorem}[UCLT]\label{UCLT}
		If $\BVar[h(X_1,\ldots,X_m)]<\infty$, then 
		$n^{1/2}(U_n-\BE_{F} h)\xrightarrow{D} N(0,m^2\delta_1)$, where $\delta_1=\BVar({H}_1(X_1)).$
	\end{theorem}
	%(a)
	%\begin{equation*}
	%U_n-\theta=\frac{m}{n}\sum_{i=1}^{n}{H}_1(X_i)+R_n, 
	%\end{equation*}  
	%where $n^{1/2}R_n \xrightarrow P  0$.\\ 
	%(b)
	%		\begin{equation}
		%	\end{equation}
	%	We shall also need the $U$-statistic CLT for the multivariate case. 
	%hall need
	If $\delta_1=0$, then $H_1\equiv 0$, and the $U$ statistic is called degenerate. See \cite{bose2018u}. In this case, the limit in Theorem \ref{UCLT} is 
	%normal distribution is 
	degenerate. 
	%That is, $n^{1/2}(U_n-\theta) \xrightarrow{P} 0$. 
	In such cases 
	%we renormalize 
	$U_n$ can be normalised differently to obtain a non-degenerate limit. 
	%The explanation of degenerate case with examples is given in \cite{bose2018u}. 
	To explain this, let 
	$$L^2(F):=\big\{h: \mathbb{R}\to \mathbb{R},\ \  \int h^2(x) dF(x) < \infty\big\},$$
	\begin{equation*}
		L^2(F\otimes F):=\{h:\mathbb{R} \times \mathbb{R}\rightarrow \mathbb{R}, \int\int h^2(x,y)dF(x)dF(y)<\infty \}.
	\end{equation*}
	Fix $h\in  L^2(F\otimes F).$
	Define the operator $T_{h}:L^2(F) \to L^2(F)$ as 
	$$T_{h}f(z) :=\int h(x, y) f(y) dF(y), \hspace{0.3cm}f\in L^2(F). $$
	Then there exists eigenvalues $\{\lambda_k\}$ and 
	%Moreover, the eigenfunctions form 
	a complete orthonormal set of  
	%		corresponding 
	eigenfunctions $\{g_k\}\subset L^2(F)$ 
	%for the operator $T$, 
	such that;
	$$T_hg_k=\lambda_{k}g_k, \hspace{0.1cm} \forall k,$$
	$$\int{g}^{2}_k(x) dF(x)=1, \hspace{0.3cm} 	\int g_k(x)g_l(x)dF(z)=0, \hspace{0.2cm} \forall k \neq l,$$
	%	It follows from the Fredholm theory of integral equations that any symmetric function of two variables admits a series expansion of the form
	and
	\begin{equation}\label{eq:spectraldecom}
		h(x,y)=\sum_{k=1}^{\infty}\lambda_kg_k(x)g_k(y),
	\end{equation}
	where the equality (\ref{eq:spectraldecom}) is in $L^2$ sense.
	That is, if $X_1$, $X_2$ are i.i.d.~$F$ then as $n\rightarrow\infty$, $$\mathbb{E}[h(X_1,X_2)-\sum_{k=1}^{n}\lambda_kg_k(X_1)g_k(X_2)]^2\rightarrow 0. $$ 
	We will refer to this as the \textit{spectral decomposition} of $h$.
	We can now state the well-known result on degenerate $U$-statistics.  
	\begin{theorem}\label{theo:standard_degen_u}
		Suppose $U_n$ is a mean zero $U$-statistics with kernel $h$ of order $m$ such that $\BE_{F} [h^2(X_1, \ldots , X_m)] < \infty$ and $H_1\equiv 0$. Then 
		\begin{equation}U_n=\dfrac{m(m-1)}{2} \dfrac{1}{{n\choose 2}} \sum_{1\leq i_1 < i_2 \leq n} H_2(X_{i_{1}}, X_{i_{2}})+\epsilon_n,
		\end{equation}
		where $n \epsilon_n\xrightarrow{P} 0$. 
		%in probability. 
		Further, $nU_n\xrightarrow{D} {m\choose 2}\sum_{j=1}^\infty \lambda_j(C_j-1)$ , where $\{\lambda_i\}$ are the
		%Let $\{\lambda_i, f_i\}, i \geq 1$  be the set of all 
		eigenvalues of $T_{H_{2}}$ and $\{	C_j\}$ are i.i.d.~chi-square random variables with one degree of freedom. 
	\end{theorem}
\section{Computation of eigenvalues of the kernel $h_F$}\label{sec:ev}
%			\section{Computation of eigenvalues of the kernel $h_F$}\label{sec:ev}
The eigensystem of the kernel $h_F$ is the solution to the integral equation,
\begin{equation}
	\lambda g(z) = \mathbb{E}h_F(z,Z)g(Z).\label{inteq}
\end{equation}
In general, this equation does not admit a closed form solution. 
%is difficult to solve. 
For the case of discrete and continuous $F$, this is reduced to a simpler problem.
\subsubsection*{The eigensystem in the discrete case:}
Let $X$ be a discrete random variable with $P(X=x_m)=p_m$ and assume without loss of generality that $x_m<x_{m+1}$ for all $m=1,\ldots,t$.
\begin{lemma}[\cite{bergsma2006new}]
	With $c_m=(x_m-x_{m-1})^{-1}$ the non-zero eigenvalues and eigenvectors of $h_F$ are the solutions to the equations;
	\begin{eqnarray}
		\begin{aligned}
			p_1g(x_1)&=\lambda c_2[g(x_1)-g(x_2)] \\
			p_tg(x_t)&=-\lambda c_t[g(x_{t-1})-g(x_t)]\\
			p_mg(x_m)&=-\lambda [c_m[g(x_{m-1})-(c_m+c_{m+1})g(x_i)+c_{m+1}g(x_{m+1})], \hspace{0.5cm}   2\leq m \leq t-1.
		\end{aligned} \label{Discrete}
	\end{eqnarray}
\end{lemma}
In matrix notation we must solve the generalized eigenvalue problem 
\begin{equation*}
	D_pg=\lambda C g,
\end{equation*}
where $D_p$ is a diagonal matrix with $p_m$ on the main diagonal, $g$ is the eigenvector with corresponding eigenvalue $\lambda$ and $C$ is a matrix of coefficients $c_m$ from the above system, i.e.,\\
\[
C=\begin{pmatrix}
	c_2 & -c_2 & 0 & 0   & \\
	-c_2 & (c_2+c_3) & -c_3 & 0 &\ldots \\
	0 & -c_3 & (c_3+c_4) & -c_4 \\
	0 & 0 & -c_4 & (c_4+c_5)\\
	&   &     &          & \ddots\\
	\vdots&  &  &  &  &  & (c_{t-1}+c_t) & -c_t\\
	&  &  &  &  &  &  -c_t &c_t \\
\end{pmatrix}
\]
%\begin{remark}
The equation given in\cite{bergsma2006new} appear to have some typographical errors in the signs 
in his $C$ matrix and the corresponding difference equations. We have made the required corrections in the above expressions. 
%changes in both places and given the corrected expressions.
%\end{remark}
\subsubsection*{The eigensystem in the continuous case:}
\begin{lemma}[\cite{ bergsma2006new}]
	Suppose $F$ is strictly increasing on the support of the probability distribution and $f$ is the derivative of $F$. Let $g$ be the eigenfunction corresponding to the integral equation \eqref{inteq}.
	%invertible. 
	Let $y(x)=g(F^{-1}(x))$ and suppose $y$ is twice differentiable. Then any eigenvalue $\lambda$ and its corresponding eigenvector $y_k(x)$  
	%and eigenvectors 
	are solutions for the equation
	\begin{equation}
		\frac{d}{dx}f[F^{-1}(x)]y'(x)+\lambda^{-1}y(x)=0, \label{Diffeqn}
	\end{equation}
	subject to the condition
	\begin{equation*}
		f[F^{-1}(x)]y'(x)\rightarrow 0 \hspace{0.2cm} \text{as} \hspace{0.2cm} x\downarrow0 \hspace{0.2cm}\text{or} \hspace{0.2cm} x \uparrow 1.
	\end{equation*}
\end{lemma}
\subsubsection*{Discrete approximation of the continuous case:}
%In general 
For most 
%many 
distributions the 
%differential equation given in 
equation (\ref{Diffeqn}) does not have a closed from solution. 
%Some examples of such distributions with no closed form solutions are $Guassian$, $Laplace$ and $Chi$-$square$ distributions etc. 
In such cases we use a 
%the eigensystem of $h_F$ can be approximated by using a 
discrete approximation as follows.
% of $F$.\\
Let $X$ be a continuous random variable $X$ with distribution function $F$. For any (large) positive integer $t$, define a discrete approximation $X^{(t)}$ to $X$ as
%random variable 
%such that for $m=1,\ldots,t$,
%\begin{equation*}
\begin{equation*}
	P(X^{(t)}={x_m}^{(t)})=p_m=\frac{1}{t},\ \ \text{where}\ \ {x_m}^{(t)}=F^{-1}\big(\frac{m-(1/2)}{t}\big), \ m=1, \ldots, t.
\end{equation*}
%and,
%%		F(m/t)-F((m-1)/t).
%\end{equation*}
Let $F^{(t)}$ be the distribution function of $X^{(t)}$. Then the eigen pair for the kernel $h_{F^{(t)}}$, obtained by solving the %difference equations given in 
(\ref{Discrete}) given in the discrete case with the coefficients $c_m={({x_m}^{(t)}-{x_{m-1}}^{(t)})}^{-1}$, serves as an approximate eigen pair for $h_F$.
%	\section{Methodology}
\section{Six bivariate distributions} \label{sec:bvd}

Standard univariate distributions can be extended to multivariate distributions in multiple ways.
% they can be extended to bi-(or multi-)variate distributions. 
For illustration we have chosen six distributions. 
% and in the following 
We present the specific bivariate extensions used in our simulations.
\vskip5pt

\noindent {\bf Normal}: This is simply 
%Here we would use 
the bivariate normal distribution 
%random variables 
with zero means, unit variances and correlation denoted by $\theta$.
\vskip5pt

\noindent {\bf Uniform}:  %The method of generating 
A pair of correlated uniform random variables $(X, Y)$ is 
generated 
%as follows: Let 
%starts with drawing two uncorrelated uniform random numbers, say 
by starting with two independent uniform variables $U$ and $V$. 
Consider an association parameter $\theta\in [-1,\ 1]$. Note that $\theta = \pm1$ indicate perfect correlation. Thus in that case the uniform pair $(X, Y)$ equals $(U, 1-U)$ or $(U, U)$. In other cases we draw another random number, say $W$, from Beta$(\alpha, 1)$, where the shape parameter $\alpha$ equals
%is given by the following.
$$ \alpha = \frac{1}{2} \left[ \sqrt{\frac{49+\theta}{1+\theta}} -5  \right].$$
%Following which 
Then we define the dependent pair $(X, Y)$ by $X=U$ and 
%$Y$ as
% given below.
$$Y = \left\{
\begin{array}{ll}
	|W-X| & \mbox{if} \, \, \, V < 1/2 \\
	1-|1-W-X| & \mbox{if} \, \, \, V \geq 1/2.
\end{array} \right.
$$

\noindent {\bf Exponential}: 
%Define 
The dependent exponential pair $(X, Y)$ is obtained as follows. Let $X$ be Exponential(1) random variable. Let $U$ be a uniform(0,1) random variable. Then define $Y$ as
$$Y = \left\{
\begin{array}{ll}
	Exponential(1 + \theta X))) & \mbox{if} \, \, \, E/(E + G)) < U \\
	Exponential(2 + \theta X)))  & \mbox{if} \, \, \, E/(E + G)) \geq U,
\end{array} \right.
$$
where $E = ((1 - \theta + \theta X)/\exp(X))/(1 + \theta X)$ and $G = ((\theta + \theta^2 X)/\exp(X))/((1 + \theta X)^2).$
% x = rexp(1)
% u = runif
%  y = rgamma(n, (u > wExp/(wExp + wGamma)) + 1, (1 + theta*x)))
\vskip5pt

\noindent {\bf Laplace}: $Z = (X, Y)$  is said to be bivariate standard Laplace, if their joint density is as follows.

$$
f(z) = \frac{2}{2\pi {|\Sigma|}^{1/2}} \left\{\frac{\pi}{2 \sqrt{2 z^T \Sigma^{-1} z}}\right\}^{1/2} \exp\big(\sqrt{-2 z^T \Sigma^{-1} z} \big)
$$
where $\Sigma$  is the matrix with unit diagonals and $\theta$ as off-diagonals.
\vskip5pt

\noindent {\bf Logistic}: Let ${U, V}$ be two correlated uniform random variables as defined above. Define $X = \log(U) - \log(1-U)$ and $Y = \log(V) - \log(1-V)$. Then $(X,Y)$ has bivariate logistic distribution. 
\vskip5pt

\noindent {\bf Chi-square}: Let ${U, V}$ be bivariate mean zero normal rvs with unit variance and correlation $\theta$. Define $X = U^{2}$ and $Y = V^{2}$.
%, each has, 
Clearly, each has $\chi^{2}$ distribution and have correlation $\theta^2$ and $(X,Y)$ has bivariate Chi-square distribution.

\end{document}